\documentclass[a4paper,10pt]{article}
\usepackage[utf8]{inputenc}
\usepackage{fullpage,color}
\usepackage{amsmath,amsfonts,amssymb,amsthm,enumerate}
\usepackage{graphicx,caption,subcaption}

\newtheorem{theorem}{Theorem}[section]
\newtheorem{proposition}[theorem]{Proposition}
\newtheorem{lemma}[theorem]{Lemma}
\newtheorem{corollary}[theorem]{Corollary}
\theoremstyle{definition}
\newtheorem{definition}{Definition}[section]

\providecommand{\NN}{\mathbb{N}}
\providecommand{\ZZ}{\mathbb{Z}}
\providecommand{\RR}{\mathbb{R}}
\providecommand{\cF}{\mathcal{F}}
\providecommand{\cG}{\mathcal{G}}
\providecommand{\cH}{\mathcal{H}}
\providecommand{\cX}{\mathcal{X}}
\providecommand{\wsup}{w_{\sup}}
\providecommand{\winfty}{w_{\infty}}
\providecommand{\shift}{\mathbb{S}}
\providecommand{\upset}[1]{\langle #1 \rangle}
\providecommand{\bin}{\mathrm{Bin}}
\providecommand{\unitcirc}{\mathbb{T}}
\providecommand{\dist}{\mathrm{dist}}
\providecommand{\ul}{\underline}
\DeclareMathOperator{\symdiff}{\triangle}
\DeclareMathOperator*{\EE}{\mathbb{E}}

\title{Ahlswede--Khachatrian Theorems: \\ Weighted, Infinite, and Hamming}
\author{Yuval Filmus\footnote{Department of Computer Science, Technion --- Israel Institute of Technology, Israel.}}

\begin{document}

\maketitle

\begin{abstract}
The seminal complete intersection theorem of Ahlswede and Khachatrian gives the maximum cardinality of a $k$-uniform $t$-intersecting family on $n$ points, and describes all optimal families. We extend this theorem to several other settings: the weighted case, the case of infinitely many points, and the Hamming scheme.

The weighted Ahlswede--Khachatrian theorem gives the maximal $\mu_p$ measure of a $t$-intersecting family on $n$ points, where $\mu_p(A) = p^{|A|} (1-p)^{n-|A|}$. As has been observed by Ahlswede and Khachatrian and by Dinur and Safra, this theorem can be derived from the classical one by a simple reduction. However, this reduction fails to identify the optimal families, and only works for $p < 1/2$. We translate the two original proofs of Ahlswede and Khachatrian to the weighted case, thus identifying the optimal families in all cases. We also extend the theorem to the case $p > 1/2$, using a different technique of Ahlswede and Khachatrian (the case $p = 1/2$ is Katona's intersection theorem). We then extend the weighted Ahlswede--Khachatrian theorem to the case of infinitely many points.

The Ahlswede--Khachatrian theorem on the Hamming scheme gives the maximum cardinality of a subset of $\ZZ_m^n$ in which any two elements $x,y$ have $t$~positions $i_1,\ldots,i_t$ such that $x_{i_j} - y_{i_j} \in \{-(s-1),\ldots,s-1\}$. We show that this case corresponds to $\mu_p$ with $p = s/m$, extending work of Ahlswede and Khachatrian, who considered the case $s = 1$. We also determine the maximum cardinality families. We obtain similar results for subsets of $[0,1]^n$, though in this case we are not able to identify all maximum cardinality families.

\end{abstract}

\section{Introduction} \label{sec:introduction}

\subsection{Background} \label{sec:background}

The Erd\H{o}s--Ko--Rado theorem~\cite{EKR}, a basic result in extremal combinatorics, states that when $k \leq n/2$, a $k$-uniform intersecting family on $n$ points contains at most $\binom{n-1}{k-1}$ sets; and furthermore, when $k < n/2$ the only families achieving these bounds are \emph{stars}, consisting of all sets containing some fixed point.

The original paper of Erd\H{o}s, Ko and Rado gave birth to an entire area in extremal combinatorics. Erd\H{o}s--Ko--Rado theorems have been proved in many other domains, such as vector spaces~\cite{FranklWilson}, permutations~\cite{EFP}, weighted families~\cite{Friedgut}, and many others (see the recent monograph~\cite{GodsilMeagher-book}).

A different research direction is extending the original theorem to $t$-intersecting families, in which every two sets are guaranteed to contain $t$ points in common rather than just one. Extending earlier results of Frankl~\cite{Frankl78}, Wilson~\cite{Wilson} showed that when $n \geq (t+1)(k-t-1)$, a $t$-intersecting family contains at most $\binom{n-t}{k-t}$ sets, and for $n > (t+1)(k-t-1)$ this is achieved uniquely by $t$-stars. When $n < (t+1)(k-t-1)$, there exist $t$-intersecting families which are larger than $t$-stars. Frankl~\cite{Frankl78} conjectured that for every $n,k,t$, the maximum families are always of the form
\[
 \cF_{t,r} = \{ S : |S \cap [t+2r]| \geq t+r \}.
\]
Frankl's conjecture was finally proved by Ahlswede and Khachatrian~\cite{AK2,AK3}, who gave two different proofs (see also the monograph~\cite{AhlswedeBlinovsky}). They also determined the maximum families under the condition that the intersection of all sets in the family is empty~\cite{AK1}, as well as the maximum $t$-intersecting families without a restriction on the size~\cite{AK4} (``Katona's theorem''). They also proved an analogous theorem for the Hamming scheme~\cite{AK5}.

A different generalization of the Ahlswede--Khachatrian theorem appears in the work of Dinur and Safra~\cite{DinurSafra} on the hardness of approximating vertex cover. Dinur and Safra were interested in the maximum $\mu_p$ measure of a $t$-intersecting family on $n$ points, where the $\mu_p$ measure is given by $\mu_p(A) = p^{|A|} (1-p)^{n-|A|}$. When $p \leq 1/2$, they related this question to the setting of the original Ahlswede--Khachatrian theorem with parameters $K,N$ satisfying $K/N \approx p$. A similar argument appears in work of Ahlswede--Khachatrian~\cite{AK4,AK5} in different guise. The $\mu_p$ setting has since been studied in many works, and has been used by Friedgut~\cite{Friedgut} and by Keller and Lifshitz~\cite{KellerLifshitz} to prove stability versions of the Ahlswede--Khachatrian theorem.

\subsection{Our results} \label{sec:results}

While not stated explicitly in either work, the methods of Dinur--Safra~\cite{DinurSafra} and Ahlswede--Khachatrian~\cite{AK5} give a proof of an Ahlswede--Khachatrian theorem in the $\mu_p$ setting for all $p < 1/2$, without any constraint on the number of points. More explicitly, let $w(n,t,p)$ be the maximum $\mu_p$-measure of a $t$-intersecting family on $n$ points, and let $\wsup(t,p) = \sup_n w(n,t,p)$. The techniques of Dinur--Safra and Ahlswede--Khachatrian show that when $\frac{r}{t+2r-1} \leq p \leq \frac{r+1}{t+2r+1}$, $\wsup(t,p) = \mu_p(\cF_{t,r})$. This theorem is incomplete, for three different reasons:
\begin{enumerate}
 \item The theorem describes $\wsup(t,p)$ rather than $w(n,t,p)$. If $n < t+2r$ then the family $\cF_{t,r}$ does not fit into $n$ points, and so $w(n,t,p) < \wsup(t,p)$.
 \item Whereas the classical Ahlswede--Khachatrian theorem describes all optimal families, the methods of Dinur--Safra and Ahlswede--Khachatrian cannot produce such a description in the $\mu_p$ case.
 \item The theorem only handles the case $p < 1/2$.
\end{enumerate}
Katona~\cite{Katona2} solved the case $p = 1/2$, which became known as ``Katona's theorem''. Ahlswede and Khachatrian gave a different proof~\cite{AK4}, and their technique applies also to the case $p > 1/2$. We complete the picture by finding $w(n,t,p)$ for all $n,t,p$ and determining all families achieving this bound. We do this by rephrasing the two original proofs~\cite{AK2,AK3} of the Ahlswede--Khachatrian theorem in the $\mu_p$ setting. Essentially the same proofs allow us to obtain the tighter results. Repeating the proofs in the easier setting of $\mu_p$ has the additional advantage of simplifying some of the notations and calculations in the proofs. Curiously, whereas the classical Ahlswede--Khachatrian theorem can be proven using either of the techniques described in~\cite{AK2,AK3}, our proof needs to use both. Part of these results appear in the author's PhD thesis~\cite{Filmus}.

We also generalize Katona's circle argument~\cite{Katona1} to the $\mu_p$ setting, thus giving a particularly simple proof of the Erd\H{o}s--Ko--Rado theorem in the $\mu_p$ setting (the work of Dinur and Friedgut~\cite{DinurFriedgut} contains a similar but much longer proof, and Friedgut~\cite{Friedgut05} gives a more complicated proof in the same style). %\comment{Frankl~\cite{Frankl76} generalized Katona's circle argument to the case in which any $s \geq 2$ sets in the family intersect. We generalize his argument as well, showing that any family $\cF$ satisfying this condition satisfies $\mu_p(\cF) \leq p$ for any $p \leq 1-1/s$.}

The methods of Dinur--Safra and Ahlswede--Khachatrian calculate the maximum $\mu_p$ measure of a $t$-intersecting family on an unbounded number of points. It is natural to ask what happens when we allow \emph{infinitely} many points rather than just an unbounded number. Formally, let $\winfty(p,t)$ be the supremum $\mu_p$-measure of a $t$-intersecting subset of $\{0,1\}^{\aleph_0}$; we only consider $\mu_p$-measurable subsets. When $p < 1/2$, we obtain $\winfty(p,t) = \wsup(p,t)$, and moreover families achieving this depend (up to measure zero) on a finite number of points, and thus must be of the form $\cF_{t,r}$. We also have $\winfty(1/2,1) = \wsup(1/2,1) = 1/2$, and this is achieved only on families depending (up to measure zero) on a finite number of points. When $t > 1$, we have $\winfty(1/2,t) = \wsup(1/2,t) = 1/2$, but no family achieves this. Finally, when $p > 1/2$ we have $\winfty(p,t) = \wsup(p,t) = 1$, and in fact there is a family of $\mu_p$-measure~$1$ in which any two sets have infinite intersection.

Ahlswede and Khachatrian~\cite{AK5} considered a natural generalization of their theorem to the Hamming scheme. In this variant, a family is a subset of $\ZZ_m^n$, and a family $\cF$ is $t$-intersecting if any two vectors in $\cF$ agree on at least $t$ coordinates. Ahlswede and Khachatrian show, in effect, that this corresponds to $p = 1/m$. This setting thus constitutes a combinatorial analog of $\mu_p$ for $p$ of the form $1/m$. It is natural to ask whether we can find similar combinatorial analogs for other values of $p$. We show that it is indeed the case, by describing a combinatorial analog for $p = s/m$, where $s \leq m/2$. In this analog, we still consider subsets of $\ZZ_m^n$, but the definition of $t$-intersecting is different: two vectors $x,y$ are $t$-intersecting if there are $t$ positions satisfying $x_i - y_i \in \{-(s-1),\ldots,s-1\}$, the subtraction taking place in $\ZZ_m$. We show that the maximum cardinality of such a family is $w(n,t,s/m) m^n$. Moreover, when $s/m < 1/2$, all optimal families arise from families of the form $\cF_{t,r}$.

A similar setting, which works for all $p < 1/2$, is that in which $\ZZ_m$ is replaced by the unit circle $[0,1]$. We consider (Lebesgue) measurable subsets of $[0,1]^n$ in which any two vectors $x,y$ have $t$ coordinates satisfying $x_i - y_i \in [-p,p]$, the subtraction taking place on the circle (that is, modulo~$1$). We show that the maximum $\mu_p$ measure of such a family is $w(n,t,p)$. Unfortunately, our argument is not strong enough to determine all optimal families.

\paragraph{Organization of the paper} The main part of the paper is Section~\ref{sec:weighted}, in which we prove the Ahlswede--Khachatrian theorem in the $\mu_p$ setting. Following an exploration of Katona's circle argument in Section~\ref{sec:katona} (which is needed for subsequent sections), we prove the infinite version of the Ahlswede--Khachatrian theorem in Section~\ref{sec:infinite}, and a version for the Hamming scheme in Section~\ref{sec:hamming}. Finally, we describe in Section~\ref{sec:reduction} the arguments of Dinur--Safra and Ahlswede--Khachatrian, which lift the Ahlswede--Khachatrian theorem to the $\mu_p$ setting.

\paragraph{Notation} We will use $[n]$ for $\{1,\ldots,n\}$, and $\binom{[n]}{k}$ for all subsets of $[n]$ of size $k$. We also use the somewhat unorthodox notation $\binom{[n]}{\geq k}$ for all subsets of $[n]$ of size at least $k$. The set of all subsets of a set $A$ will be denoted $2^A$. The notation $\bin(n,p)$ is used for the binomial distribution with $n$ trials and success probability $p$.

\section{Weighted Ahlswede--Khachatrian theorem} \label{sec:weighted}

In this section we prove an Ahlswede--Khachatrian theorem for families of sets, measured according to the $\mu_p$ measures. We start with several basic definitions.

\begin{definition} \label{def:basic}
 A \emph{family on $n$ points} is a collection of subsets of $[n]$. A family $\cF$ is \emph{$t$-intersecting} if any $A,B \in \cF$ satisfy $|A \cap B| \geq t$. Two families $\cF,\cG$ are \emph{cross-$t$-intersecting} if any $A \in \cF$ and $B \in \cG$ satisfy $|A \cap B| \geq t$.
 
 We will refer to a $1$-intersecting family or a pair of cross-$1$-intersecting families as an \emph{intersecting} family or a pair of \emph{cross-intersecting} families.
\end{definition}

\begin{definition} \label{def:mup}
 For any $p \in (0,1)$ and any $n$, the product measure $\mu_p$ is a measure on the set of subsets of $[n]$ given by
\[
 \mu_p(A) = p^{|A|} (1-p)^{n-|A|}.
\]
\end{definition}

\begin{definition} \label{def:upset}
 A family $\cF$ on $n$ points is \emph{monotone} if whenever $A \in \cF$ and $B \supseteq A$ then $B \in \cF$. Given a family $\cF$, its \emph{up-set} $\cG$ is the smallest monotone family containing $\cF$, namely
\[
 \cG = \{ B \supseteq A : A \in \cF \}.
\]
 We also denote $\cG$ by $\upset{\cF}$.
\end{definition}

Note that if $\cF$ is $t$-intersecting then so it its up-set. It follows that if $\mu_p(\cF) = w(n,t,p)$ then $\cF$ is monotone.

One main goal of this section is to determine the two quantities $w(n,t,p)$ and $\wsup(t,p)$, which are defined as follows.

\begin{definition} \label{def:wntp}
 For $n \geq t \geq 1$ and $p \in (0,1)$, the parameter $w(n,t,p)$ is the maximum of $\mu_p(\cF)$ over all $t$-intersecting families on $n$ points.
\end{definition}

\begin{definition} \label{def:wsup}
 For $t \geq 1$ and $p \in (0,1)$, the parameter $\wsup(t,p)$ is given by
\[
 \wsup(t,p) = \sup_n w(n,t,p).
\]
\end{definition}

We can also define $\wsup(t,p)$ as a limit instead of a supremum.

\begin{lemma} \label{lem:wsup-lim}
 For $t \geq 1$ and $p \in (0,1)$,
\[
 \wsup(t,p) = \lim_{n\to\infty} w(n,t,p).
\]
\end{lemma}
\begin{proof}
 If $\cF$ is a $t$-intersecting family on $n$ points then $\cF' = \{ S \subseteq [n+1] : S \cap [n] \in \cF \}$ is a $t$-intersecting family on $n+1$ points having the same $\mu_p$-measure. Therefore the sequence $w(n,t,p)$ is non-decreasing, and the lemma follows.
\end{proof}

As mentioned in the introduction, $\wsup(t,p)$ for $p < 1/2$ can be derived from the classical Ahlswede--Khachatrian theorem via the methods of Dinur--Safra and Ahlswede--Khachatrian, derivations that we describe in Section~\ref{sec:reduction}.

The other main goal of this section is finding all $t$-intersecting families on $n$ points which achieve the bound. The families, first described by Frankl~\cite{Frankl78}, are the following.

\begin{definition} \label{def:frankl}
 For $t \geq 1$ and $r \geq 0$, the \emph{$(t,r)$-Frankl family} on $n$ points is the $t$-intersecting family
\[
 \cF_{t,r} = \{ A \subseteq [n] : |A \cap [t+2r]| \geq t+r \}.
\]
 A family $\cF$ on $n$ points is \emph{equivalent} to a $(t,r)$-Frankl family if there exists a set $S \subseteq [n]$ of size $t+2r$ such that
\[
 \cF = \{ A \subseteq [n] : |A \cap S| \geq t+r \}.
\]
\end{definition}

We will prove the following theorem.

\begin{theorem} \label{thm:weighted-main}
 Let $n \geq t \geq 1$ and $p \in (0,1)$. If $\cF$ is $t$-intersecting then
\[
 \mu_p(\cF) \leq \max_{r\colon t+2r \leq n} \mu_p(\cF_{t,r}).
\]

 Moreover, unless $t = 1$ and $p \geq 1/2$, equality holds only if $\cF$ is equivalent to a Frankl family $\cF_{t,r}$.
 
 When $t = 1$ and $p > 1/2$, the same holds if $n + t$ is even, and otherwise $\cF = \cG \cup \binom{[n]}{\geq \frac{n+t+1}{2}}$ where $\cG \subseteq \binom{[n]}{\frac{n+t-1}{2}}$ contains exactly $\binom{n-1}{\frac{n+t-1}{2}}$ sets.
\end{theorem}

When $t = 1$ and $p = 1/2$ there are many optimal families. For example, the families $\cF_{1,r}$ all have $\mu_{1/2}$-measure~$1/2$, as does the family $(\{ S : 1 \in S \}) \setminus \{\{1\}\} \cup \{\{2,\ldots,n\}\}$. %\comment{Reference with more discussion?}

Similarly, when $t = 1$, $p > 1/2$ and $n + 1$ is odd there are many optimal families. For example, both of the following families are optimal:
\begin{align*}
 &\{ S \subseteq [n] : |S| \geq \frac{n}{2} + 1 \} \cup \{ S \subseteq [n] : |S| = \frac{n}{2} \text{ and } 1 \notin S \}, \\
 &\{ S \subseteq [n] : |S| \geq \frac{n}{2} + 1 \} \cup \{ S \subseteq [n] : |S| = \frac{n}{2} \text{ and } 1 \in S \}.
\end{align*}

Our proof actually implies the following more detailed corollary, which can also be deduced from the theorem using Lemma~\ref{lem:frankl-mup} below.

\begin{corollary} \label{cor:wntp}
 Let $n \geq t \geq 1$. Define $r^\ast$ as the maximal integer satisfying $t + 2r^\ast \leq n$.
 
 If $t = 1$ then
\[
 w(n,1,p) =
 \begin{cases}
  p & p \leq \frac{1}{2}, \\
  \mu_p(\cF_{1,r^\ast}) & p \geq \frac{1}{2}.
 \end{cases}
\]
 Furthermore, if $\cF$ is an intersecting family of $\mu_p$-measure $w(n,1,p)$ for $p \in (0,1)$ then:
\begin{itemize}
 \item If $p < \frac{1}{2}$ then $\cF$ is equivalent to $\cF_{1,0}$.
 \item If $p > \frac{1}{2}$ and $n$ is odd then $\cF$ is equivalent to $\cF_{1,\frac{n-1}{2}}$.
 \item If $p > \frac{1}{2}$ and $n$ is even then $\cF = \cG \cup \binom{[n]}{\geq n/2+1}$, where $\cG$ contains half the sets in $\binom{[n]}{n/2}$: exactly one of each pair $A,[n] \setminus A$.
\end{itemize}
 
 If $t \geq 2$ then
\[
 w(n,t,p) =
 \begin{cases}
  \mu_p(\cF_{t,r}) & \frac{r}{t+2r-1} \leq p \leq \frac{r+1}{t+2r+1} \text{ for some } r < r^\ast, \\
  \mu_p(\cF_{t,r^\ast}) & \frac{r^\ast}{t+2r^\ast-1} \leq p.
 \end{cases}
\]
 Furthermore, if $\cF$ is a $t$-intersecting family of $\mu_p$-measure $w(n,t,p)$ for $p \in (0,1)$ then:
\begin{itemize}
 \item If $\frac{r}{t+2r-1} < p < \frac{r+1}{t+2r+1}$ for some $r < r^\ast$ then $\cF$ is equivalent to $\cF_{t,r}$.
 \item If $\frac{r^\ast}{t+2r^\ast-1} < p$ then $\cF$ is equivalent to $\cF_{t,r^\ast}$.
 \item If $p = \frac{r+1}{t+2r+1}$ for some $r < r^\ast$ then $\cF$ is equivalent to $\cF_{t,r}$ or to $\cF_{t,r-1}$.
\end{itemize}
\end{corollary}

As a corollary, we can compute $w(t,p)$.

\begin{corollary} \label{cor:wtp}
 We have
\[
 w(1,p) =
 \begin{cases}
  p & p \leq \frac{1}{2}, \\
  1 & p > \frac{1}{2}
 \end{cases}
\]

 For $t \geq 2$, we have
\[
 w(t,p) =
 \begin{cases}
  \mu_p(\cF_{t,r}) & \frac{r}{t+2r-1} \leq p \leq \frac{r+1}{t+2r+1}, \\
  \frac{1}{2} & p = \frac{1}{2}, \\
  1 & p > \frac{1}{2}.
 \end{cases}
\]
\end{corollary}
\begin{proof}
 We start with the case $p < 1/2$. If $t = 1$ then $w(n,1,p) = p$ for all $n \geq 1$, implying the formula for $w(1,p)$. If $t \geq 2$ and $\frac{r}{t+2r-1} \leq p \leq \frac{r+1}{t+2r+1}$ then $w(n,t,p) = \mu_p(\cF_{t,r})$ for all $n \geq t+2r$, implying the formula for $w(t,p)$.

 Suppose now that $p \geq 1/2$. In this case we have $w(t+2r,t,p) = w(t+2r+1,t,p) = \cF_{t,r}$, and $\mu_p(\cF_{t,r}) = \Pr[\bin(t+2r,p) \geq t+r]$. If $p = 1/2$ then
\[
 \mu_{\frac{1}{2}}(\cF_{t,r}) = \Pr[\bin(t+2r,\tfrac{1}{2}) \geq t+r] = \Pr[\bin(t+2r,\tfrac{1}{2}) \geq \tfrac{t}{2}+r] - o_r(1) = \frac{1}{2} - o_r(1),
\]
 using the central limit theorem, and so $w(t,1/2) = 1/2$. If $p > 1/2$ then for any $q \in (1/2,p)$, for large enough $r$ we have
\[
 \mu_p(\cF_{t,r}) = \Pr[\bin(t+2r,p) \geq t+r] \geq \Pr[\bin(t+2r,p) \geq q(t+2r)] \xrightarrow{r\to\infty} 1,
\]
 using the central limit theorem.
\end{proof}

Figure~\ref{fig:weighted} illustrates Corollary~\ref{cor:wntp} and Corollary~\ref{cor:wtp}.

\begin{figure}
  \centering
  \begin{subfigure}{0.45\textwidth}
    \includegraphics[width=\textwidth]{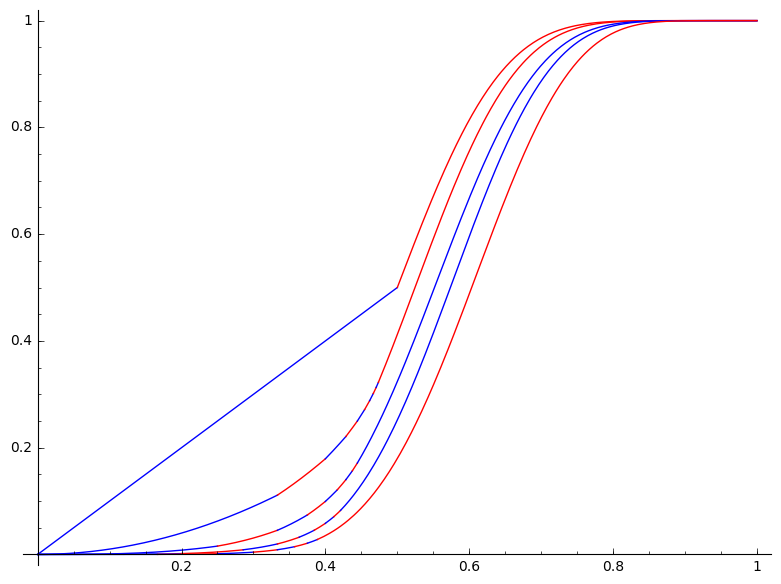}
  \end{subfigure}
  \begin{subfigure}{0.45\textwidth}
    \includegraphics[width=\textwidth]{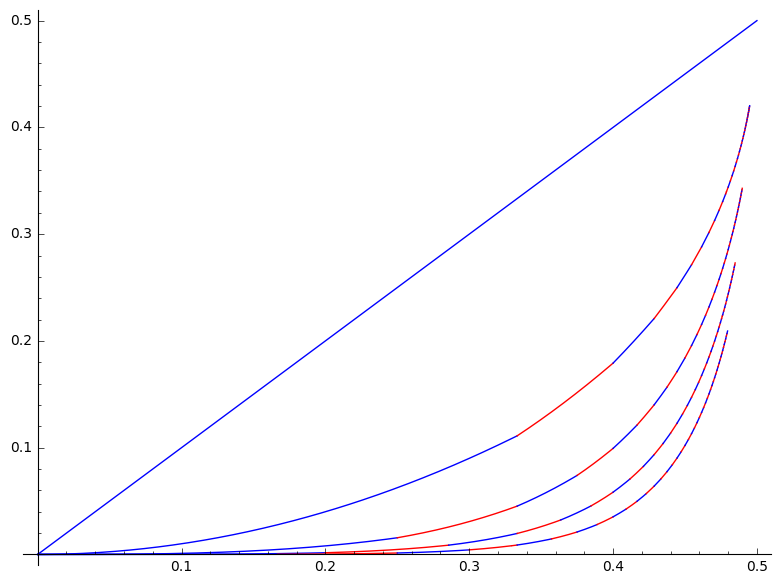}
  \end{subfigure}
  \caption{The function $w(20,t,p)$ for $1 \leq t \leq 5$ (left) and the function $w(t,p)$ for $1 \leq t \leq 5$ (right). In both cases, larger functions correspond to smaller $p$. The colors switch at each of the breakpoints $\frac{r}{t+2r-1}$ for $r \leq r^\ast$ (left) or for each $r$ (right).}
  \label{fig:weighted}
\end{figure}

\paragraph{Organization of the section} Section~\ref{sec:weighted-overview} gives a brief overview of the entire proof. Section~\ref{sec:frankl} analyzes the $\mu_p$-measures of the families $\cF_{t,r}$. The proof of the main theorem in the cases $p < 1/2$ and $p = 1/2$ appears in Section~\ref{sec:psmall} and Section~\ref{sec:phalf} (respectively), using techniques developed in Sections~\ref{sec:shifting}--\ref{sec:symmetrization}. Finally, the case $p > 1/2$ is handled in Section~\ref{sec:plarge}.

\subsection{Proof overview} \label{sec:weighted-overview}

The cases $p \leq 1/2$ and $p > 1/2$ have completely different proofs, which we describe separately.

When $p > 1/2$, the proof uses a shifting technique due to Ahlswede and Khachatrian~\cite{AK4}, which they developed with the case $p = 1/2$ (Katona's theorem) in mind. Their shifting is parametrized by two disjoint sets $A,B$ of sizes $|A| = s$ and $|B| = s+1$. If a set $S$ in the family contains $A$ and is disjoint from $B$, the shifting replaces $S$ with $(S \setminus A) \cup B$, unless the latter is already in the family. This shifting increases the $\mu_p$-measure of the family when $p > 1/2$, and if done correctly maintains the property of being $t$-intersecting. Eventually we reach a family which is stable under all such shifts. A simple argument shows that in a stable $t$-intersecting family, all sets have cardinality at least $\frac{n+t-1}{2}$. When $n+t$ is even, this directly yields the desired result. When $n+t$ is odd, the argument is a bit more subtle, involving a characterization of all maximum size $t$-intersecting subsets of $\binom{[t+2r+1]}{t+r}$. The entire argument is described in Section~\ref{sec:plarge}.

\medskip

The case $p \leq 1/2$ combines two different techniques of Ahlswede and Khachatrian: the technique of generating sets~\cite{AK2}, and the pushing-pulling technique~\cite{AK3}. In the classical setting, each of these techniques suffices by itself to prove the Ahlswede--Khachatrian, using the fact that if $\cF \subseteq \binom{[n]}{k}$ is $t$-intersecting then $\{ [n] \setminus S : S \in \cF \} \subseteq \binom{[n]}{n-k}$ is $(n-2k+t)$-intersecting. In our setting, this argument doesn't seem to work, and so we need to use both techniques together.

\paragraph{Shifting} The first step is to apply classical shifting to the family to make it left-compressed: if $A$ is in the family and $i < j$ are such that $j \in A$ and $i \notin A$, then $(A \setminus \{j\}) \cup \{i\}$ is also in the family. The shifting maintains the property of being $t$-intersecting and doesn't change the $\mu_p$-measure. This (standard) step is described in Section~\ref{sec:shifting}.

\paragraph{Generating sets} The first technique, that of generating sets, is used to show that an optimal family depends on few coordinates. We can assume without loss of generality that the family $\cF$ in question is monotone. If $\cG$ consists of all inclusion-minimal sets of $\cF$ (also known as \emph{minterms} or, following Ahlswede and Khachatrian, \emph{generating sets}), then $\cF$ simply consists of all supersets of sets in $\cG$. We can thus work with $\cG$ instead of $\cF$. The idea is that since $\cF$ is left-compressed, the number of coordinates that $\cF$ depends on is exactly the maximal number appearing in a set of $\cG$. Denote this number by $m$, and let $\cG^*$ consist of all sets in $\cG$ containing $m$.

Suppose that we would like to prove that if $\cF$ has maximum measure then it depends on at most $t+2r$ coordinates. If $m > t+2r$ then we would like to modify $\cF$ to a $t$-intersecting family of larger measure. Since the optimal family depends on fewer than $m$ coordinates, it is natural to modify $\cF$ by getting rid of its dependency on $m$. One way of doing this is by deleting $m$ from sets in $\cG^*$. Case-by-case analysis shows that the only reason the new family will not be $t$-intersecting is if we removed $m$ from two sets $A,B \in \cG^*$ which intersect in \emph{exactly} $t$ elements. Moreover, this can only happen if $|A| + |B| = m + t$. In view of this, we partition $\cG^*$ into sets $\cG^*_0,\ldots,\cG^*_m$ according to their size.

Going ahead with the idea of deleting $m$ from sets in $\cG^*$ while keeping in mind the need for the family to remain $t$-intersecting, we operate on pairs $\cG^*_a,\cG^*_b$ for $a + b = m + t$. If $a \neq b$ then we can delete $m$ from all sets in $\cG^*_a$ and, in order to guarantee that the result is $t$-intersecting, \emph{remove} all sets in $\cG^*_b$. Alternatively, we can delete $m$ from all sets in $\cG^*_b$, and remove all sets in $\cG^*_a$. Calculation shows that one of these two families has measure larger than $\cF$ whenever $p < 1/2$.

We cannot execute the preceding operation on $\cG^*_a$ when $a = \frac{m+t}{2}$. Instead, we choose an element $i < m$, delete $m$ from all sets in $\cG^*_a$ \emph{not} containing $i$, and remove all sets in $\cG^*_a$ containing $i$. The resulting family is $t$-intersecting for all $i < m$, and when $p < \frac{m-t}{2(m-1)}$, one of the new families has measure larger than $\cF$. This case cannot happen when $m = t+2r+1$, and when $m = t+2r+2$, the critical probability is $\frac{m-t}{2(m-1)} = \frac{r+1}{t+r+1}$, what we expect.

These arguments appear in Section~\ref{sec:generating-sets}.

\paragraph{Pushing-pulling} The second technique, that of pushing-pulling, is used to show that an optimal family is \emph{symmetric}, that is, invariant to permutations of coordinates it depends on; or in other words, a family of the form $\{S : |S \cap X| \geq s\}$. If a family is not completely symmetric then perhaps it is symmetric in at least some of its coordinates. We define $\ell$ to be the largest number such that the family is invariant under permutations of the first $\ell$ coordinates. The parameter $\ell$ plays a role similar to that of the parameter $m$ in the preceding technique.

Suppose that the family $\cF$ depends on $m > \ell$ coordinates. Since $\cF$ is left-compressed, the only thing preventing it from being more symmetric are sets $S \in \cF$ such that $(S \setminus \{i\}) \cup \{\ell+1\} \notin \cF$, where $i \in S$ and $\ell+1 \notin S$. We denote the collection of all these sets by $\cX$. The natural way of making $\cF$ more symmetric is to take some $A \in \cX$ and add $(A \setminus \{i\}) \cup \{\ell+1\}$ to $\cF$ for all $i \in A$; we call this operation \emph{expansion}. Case-by-case analysis shows that the only reason the new family will not be $t$-intersecting is if we applied this operation to two sets $A,B \in \cX$ having intersection \emph{exactly}~$t$. Moreover, this can only happen if $|A \cap [\ell]| + |B \cap [\ell]| = \ell + t$. In view of this, we partition $\cX$ into sets $\cX_0,\ldots,\cX_\ell$ according to the size of their intersection with $[\ell]$.

Mirroring our efforts in the preceding technique, we will operate on pairs $\cX_a,\cX_b$ for $a + b = \ell + t$. If $a \neq b$ then we can expand all sets in $\cX_a$ and delete all sets in $\cX_b$ or vice versa. In both cases the result will be $t$-intersecting. Calculation shows that for \emph{all} $p$, one of the two new families will have measure larger than $\cF$.

In order to handle $\cX_a$ for $a = \frac{\ell+t}{2}$ we will need a different approach: we expand $\cX_a$, but for both $\cX_a$ and its expansions we keep only sets containing some new common element $s \neq \ell+1$. Such a common element can be found if $m < n$ but not if $m = n$, and this corresponds to the fact that the optimal family $\cF_{t,r}$ satisfies $t+2r \leq n$. This operation increases the measure as long as $p > \frac{\ell-t+2}{2(\ell+1)}$. Assuming that $m = t+2r$, if $\ell = m-1$ then the new operation isn't needed (since $\ell+t$ is odd), and if $\ell = m-2$ then the critical probability is $\frac{\ell-t+2}{2(\ell+1)} = \frac{t}{t+2r-1}$, again what we expect.

These arguments appear in Section~\ref{sec:symmetrization}.

\paragraph{Completing the proof} When $\frac{r}{t+2r-1} < p < \frac{r+1}{t+2r+1}$, applying both techniques immediately shows that any left-compressed $t$-intersecting family of maximum measure must equal $\cF_{t,r}$. An arbitrary $t$-intersecting family can be made left-compressed by applying a sequence of shifts. Since the shifts don't affect the measure, this shows that any $t$-intersecting family has measure at most $\mu_p(\cF_{t,r})$. A more delicate argument (taken from~\cite{AK2} and appearing in Section~\ref{sec:shifting}) shows that if shifting results in a family equivalent to $\cF_{t,r}$ then the original family is also equivalent to $\cF_{t,r}$, and this shows that any $t$-intersecting family of maximum measure must be equivalent to $\cF_{t,r}$.

When $\frac{r^\ast}{t+2r^\ast-1} < p \leq 1/2$ and when $p = \frac{r}{t+2r-1}$ the application of the techniques becomes more subtle, but otherwise the approach is very similar.

The argument for $p < 1/2$ appears in Section~\ref{sec:psmall}, and the one for $p = 1/2$ appears in Section~\ref{sec:phalf}.

\subsection{Frankl families} \label{sec:frankl}

Before starting the proper part of the proof, it is prudent to study the Frankl families $\cF_{t,r}$. It is easy to see that they are all $t$-intersecting.  The following lemma analyzes their measures.

\begin{lemma} \label{lem:frankl-mup}
 Let $t \geq 1$ and $r \geq 0$ be parameters, and let $p_{t,r} = \frac{r+1}{t+2r+1}$. If $p < p_{t,r}$ then $\mu_p(\cF_{t,r}) > \mu_p(\cF_{t,r+1})$. If $p = p_{t,r}$ then $\mu_p(\cF_{t,r}) = \mu_p(\cF_{t,r+1})$. If $p > p_{t,r}$ then $\mu_p(\cF_{t,r}) < \mu_p(\cF_{t,r+1})$.
\end{lemma}
\begin{proof}
 Suppose for simplicity that we consider the two families $\cF_{t,r}$ and $\cF_{t,r+1}$ as families on $t+2r+2$ points. It is then not too hard to check that $\cF_{t,r} \setminus \cF_{t,r+1} = \binom{[t+2r]}{t+r}$ and $\cF_{t,r+1} \setminus \cF_{t,r} = \binom{[t+2r]}{t+r-1} \times \{t+2r+1,t+2r+2\}$. Therefore
\begin{align*}
 \mu_p(\cF_{t,r}) - \mu_p(\cF_{t,r+1}) &=
 \binom{t+2r}{t+r} p^{t+r} (1-p)^{r+2} -
 \binom{t+2r}{t+r-1} p^{t+r+1} (1-p)^{r+1} \\ &=
 \binom{t+2r}{t+r-1} p^{t+r} (1-p)^{r+1}
 \left[ \frac{r+1}{t+r} (1-p) - p \right].
\end{align*}
The bracketed expression equals $\frac{r+1}{t+r} - \frac{t+2r+1}{t+r} p$, and so is positive for $p < \frac{r+1}{t+2r+1}$, negative for $p > \frac{r+1}{t+2r+1}$, and vanishes when $p = \frac{r+1}{t+2r+1}$.
\end{proof}

\begin{corollary} \label{cor:frankl-mup}
 Let $t \geq 2$. If $p_{t,r-1} < p < p_{t,r}$ then the maximum of $\mu_p(\cF_{t,s})$ is attained at $s = r$.
 
 If $p = p_{t,r}$ then the maximum of $\mu_p(\cF_{t,s})$ is attained at $s = r$ and $s = r+1$.
\end{corollary}

Using this lemma one can deduce Corollary~\ref{cor:wntp} from Theorem~\ref{thm:weighted-main}.

\subsection{Shifting} \label{sec:shifting}

The proofs of Ahlswede and Khachatrian use the technique of \emph{shifting} to obtain families which are easier to analyze.

\begin{definition} \label{def:shifting}
 Let $\cF$ be a family on $n$ points and let $i,j \in [n]$ be two different indices. The shift operator $\shift_{i,j}$ acts on $\cF$ as follows. Let $\cF_{i,j}$ consist of all sets in $\cF$ containing $i$ but not $j$. Then
\begin{align*}
 \shift_{i,j}(\cF) = (\cF \setminus \cF_{i,j}) &\cup \{ A : A \in \cF_{i,j} \text{ and } (A \setminus \{i\}) \cup \{j\} \in \cF \} \\ &\cup \{ (A \setminus \{i\}) \cup \{j\} : A \in \cF_{i,j} \text{ and } (A \setminus \{i\}) \cup \{j\} \notin \cF \}.
\end{align*}
 In words, we try to ``shift'' each set $A \in \cF_{i,j}$ by replacing it with $A' = (A \setminus \{i\}) \cup \{j\}$. If $A' \notin \cF$ then we replace $A$ with $A'$, and otherwise we don't change $A$.
\end{definition}

Shifting clearly preserves the $\mu_p$-measure of the family, a property which we record.

\begin{lemma} \label{lem:shifting-mup}
 For any family $\cF$ and indices $i,j$ and for all $p \in (0,1)$, $\mu_p(\cF) = \mu_p(\shift_{i,j}(\cF))$.
\end{lemma}

Shifting preserves the property of being $t$-intersecting. This is a well-known property which we prove for completeness.

\begin{lemma} \label{lem:shifting-int}
 If $\cF$ is $t$-intersecting then so is $\shift_{i,j}(\cF)$ for any $i,j$.
\end{lemma}
\begin{proof}
 Suppose $A,B \in \shift_{i,j}(\cF)$. There are three cases to consider.
 
 \textbf{Case 1: $A,B \in \cF$.} In this case it is clear that $|A \cap B| \geq t$.
 
 \textbf{Case 2: $A,B \notin \cF$.} In this case $j \in A,B$, $i \notin A,B$, and the two sets $A' = (A \setminus \{j\}) \cup \{i\}$ and $B' = (B \setminus \{j\}) \cup \{i\}$ are in $\cF$. It is not hard to check that $|A \cap B| = |A' \cap B'| \geq t$.
 
 \textbf{Case 3: $A \in \cF$ and $B \notin \cF$.} In this case $j \in B$, $i \notin B$, and $B' = (B \setminus \{j\}) \cup \{i\} \in \cF$. Thus $|A \cap B'| \geq t$. If $i \notin A$ or $j \in A$ then $|A \cap B| \geq |A \cap B'| \geq t$. If $i \in A$ and $j \notin A$ then $A' = (A \setminus \{j\}) \cup \{i\} \in \cF$ (since otherwise $A$ would have been shifted), and so $|A \cap B| \geq |A' \cap B'| \geq t$.
\end{proof}

By shifting $\cF$ repeatedly we can obtain a \emph{left-compressed} family. This is another well-known property which we prove for completeness.

\begin{definition} \label{def:left-compressed}
 A family $\cF$ on $n$ points is \emph{left-compressed} if whenever $A \in \cF$, $i \in A$, $j \notin A$, and $j < i$, then $(A \setminus \{i\}) \cup \{j\} \in \cF$. (Informally, we can \emph{shift} $i$ to $j$.)
\end{definition}

\begin{lemma} \label{lem:left-compressed}
 Let $\cF$ be a $t$-intersecting family on $n$ points. There is a left-compressed $t$-intersecting family $\cG$ on $n$ points with the same $\mu_p$-measure for all $p \in (0,1)$. Furthermore, $\cG$ can be obtained from $\cF$ by applying a sequence of shift operators.
\end{lemma}
\begin{proof}
 A family is left-compressed if $\shift_{i,j}(\cF) = \cF$ for all $j < i$, as a simple examination of the definitions reveals. Consider now the potential function $\Phi$, whose value on a family $\cH$ is the total sum of all elements appearing in all sets of $\cH$. It is easy to check that if $\shift_{i,j}(\cF) \neq \cF$ for some $j < i$ then $\Phi(\shift_{i,j}(\cH)) < \Phi(\cH)$. Since $\Phi$ is a non-negative integer, if we apply the shifts $\shift_{i,j}$ for all $j < i$ repeatedly on $\cF$, eventually we will reach a family whose potential is not decreased by any of the shifts, and this family is left-compressed.
\end{proof}

Lemma~\ref{lem:left-compressed} shows that in order to determine $w(n,t,p)$ it is enough to focus on left-compressed families. Moreover, since the up-set of a $t$-intersecting family is also $t$-intersecting, we will assume in most of what follows that $\cF$ is a monotone left-compressed $t$-intersecting family. We will show that except for the case $p \geq 1/2$ and $t = 1$, such a family can only have maximum $\mu_p$-measure if it is a Frankl family with the correct parameters. We will deduce that general $t$-intersecting families of measure $w(n,t,p)$ are equivalent to a Frankl family using the following lemma, whose proof closely follows the argument of Ahlswede and Khachatrian~\cite{AK2}.

\begin{lemma} \label{lem:shifting-frankl}
 Let $\cF$ be a monotone $t$-intersecting family on $n$ points, and let $i,j \in [n]$. If $\shift_{i,j}(\cF)$ is equivalent to $\cF_{t,r}$ then so is $\cF$.
\end{lemma}
\begin{proof} 
 Let $S \subseteq [n]$ be the set of size $t+2r$ such that
\[
 \shift_{i,j}(\cF) = \{A \subseteq [n] : |A \cap S| \geq t+r\}.
\]
 
 Suppose first that $i,j \in S$ or $i,j \notin S$. If $A \in \shift_{i,j}(\cF)$ then $A \in \cF$, since otherwise $A$ would have originated from $A' = (A \setminus \{j\}) \cup \{i\}$, but that is impossible since $A' \in \shift_{i,j}(\cF)$. It follows that $\shift_{i,j}(\cF) \subseteq \cF$ and so $\shift_{i,j}(\cF) = \cF$, since shifting preserves cardinality. Therefore the lemma trivially holds.
 
 The case $i \in S$ and $j \notin S$ cannot happen. Indeed, consider some set $A \subseteq S$ containing $i$ but not $j$ of size $t+r$. Then $A \in \shift_{i,j}(\cF)$ and so, by definition of the shift, $A' = (A \setminus \{i\}) \cup \{j\} \in \shift_{i,j}(\cF)$. However, $|A' \cap S| = t+r-1$, and so $A' \notin \shift_{i,j}(\cF)$, and we reach a contradiction.
 
 It remains to consider the case $i \notin S$ and $j \in S$. Suppose first that $r = 0$. Then $S \in \shift_{i,j}(\cF)$, and so either $S \in \cF$ or $S' = (S \setminus \{j\}) \cup \{i\} \in \cF$. In both cases, since $\cF$ is monotone, it contains all supersets of $S$ or of $S'$. Since shifting preserves cardinality, $\cF$ must consist exactly of all supersets of $S$ or of $S'$, and thus is equivalent to a $(t,0)$-Frankl family.
 
 Suppose next that $r > 0$. Let $V$ be the collection of all subsets of $S \setminus \{j\}$ of size exactly $t+r-1$. For each $A \in V$ we have $A \cup \{j\} \in \shift_{i,j}(\cF)$, and so either $A \cup \{j\} \in \cF$ or $A \cup \{i\} \in \cF$.
 
 If $\cF$ contains $A \cup \{j\}$ for all $A \in V$ then $\cF$ contains all subsets of $S$ of size $t+r$ (since other subsets are not affected by the shift). Monotonicity forces $\cF$ to contain all of $\shift_{i,j}(\cF)$, and thus $\cF = \shift_{i,j}(\cF)$ as before.
 
 If $\cF$ contains $A \cup \{i\}$ for all $A \in V$, then in a similar way we deduce that $\cF$ is equivalent to the $(t,r)$-Frankl family based on $(S \setminus \{j\}) \cup \{i\}$.
 
 It remains to consider the case in which $\cF$ contains $A \cup \{i\}$ for some $A \in V$, and $B \cup \{j\}$ for some other $B \in V$. We will show that in this case, $\cF$ is not $t$-intersecting. Consider the graph on $V$ in which two vertices are connected if their intersection has the minimal size $t-1$. This graph is a generalized Johnson graph, and we show below that it is connected. This implies that there must be two sets $A,B$ satisfying $|A \cap B| = t-1$ such that $A \cup \{i\},B \cup \{j\} \in \cF$. Since $|A \cap B| = t-1$, we have reached a contradiction.
 
 To complete the proof, we prove that the graph is connected. For reasons of symmetry, it is enough to give a path connecting $x = \{1,\ldots,t+r-1\}$ and $y = \{2,\ldots,t+r\}$. Indeed, the vertex $\{2,\ldots,t,t+r+1,\ldots,t+2r\}$ is connected to both $x$ and $y$.
\end{proof}

The preceding lemmas allow us to reduce the proof of Theorem~\ref{thm:weighted-main} to the left-compressed case. The final lemma of this section abstracts this argument.

\begin{lemma} \label{lem:compression}
 Fix $n,t,p$. If every monotone left-compressed $t$-intersecting family on $n$ points has $\mu_p$-measure at most $q$, then every $t$-intersecting family on $n$ points has $\mu_p$-measure at most $q$.

 Suppose moreover that the only monotone left-compressed $t$-intersecting families on $n$ points having $\mu_p$-measure $q$ are $\cF_{t,r}$, where $r$ ranges over some set $R$. Then every $t$-intersecting family on $n$ points having $\mu_p$-measure $q$ is equivalent to a Frankl family $\cF_{t,r}$ for some $r \in R$.
\end{lemma}
\begin{proof}
 Suppose first that every monotone left-compressed $t$-intersecting family on $n$ points has $\mu_p$-measure at most $q$, and let $\cF$ be a $t$-intersecting family on $n$ points. Lemma~\ref{lem:left-compressed} gives a left-compressed $t$-intersecting family $\cG$, obtained from $\cF$ by applying a sequence of shifts, which has the same $\mu_p$-measure as $\cF$. Let $\cH$ be the up-set of $\cG$, which is a monotone left-compressed $t$-intersecting family. By assumption, $\mu_p(\cH) \leq q$, and so $\mu_p(\cF) = \mu_p(\cG) \leq \mu_p(\cH) \leq q$.
 
 Suppose next that the only monotone left-compressed $t$-intersecting families on $n$ points having $\mu_p$-measure $q$ are $\cF_{t,r}$ for $r \in R$, and suppose that $\cF$ has $\mu_p$-measure $q$. In that case necessarily $\cG = \cH$, and so $\cG = \cF_{t,r}$ for some $r \in R$, by assumption. Repeatedly applying Lemma~\ref{lem:shifting-frankl} (starting with the last shifting applied and working backwards), we see that $\cF$ is equivalent to $\cF_{t,r}$.
\end{proof}

\subsection{Generating sets} \label{sec:generating-sets}

The goal of the first part of the proof, which follows~\cite{AK2}, is to show that any monotone left-compressed $t$-intersecting family of maximum $\mu_p$-measure has to depend on a small number of points. We will use a representation of monotone families in which this property has a simple manifestation. Our definition is simpler than the original one, due to the simpler setting.

\begin{definition} \label{def:generating-sets}
 A family $\cF$ on $n$ points is \emph{non-trivial} if $\cF \notin \{\emptyset,2^{[n]}\}$.

 Let $\cF$ be a non-trivial monotone family. A \emph{generating set} is a set $S \in \cF$ such that no subset of $S$ belongs to $\cF$. The \emph{generating family} of $\cF$ consists of all generating sets of $\cF$. The \emph{extent} of $\cF$ is the maximal index appearing in a generating set of $\cF$. The \emph{boundary generating family} of $\cF$ consists of all generating sets of $\cF$ containing its extent.
 
 If $\cG$ is the generating family of $\cF$ then we use the notation $\cG^*$ for the boundary generating family of $\cF$. For each integer $a$, we use the notation $\cG^*_a$ for the subset of $\cG^*$ consisting of sets of size $a$.
\end{definition}

Generating sets are also known as \emph{minterms}. If $\cG$ is the generating family of $\cF$ then $\cG$ is an antichain and $\cF$ is the up-set of $\cG$ (and this gives an alternative definition of $\cG$). If $\cF$ has extent $m$ then $\cF$ depends only on the first $m$ coordinates: $S \in \cF$ iff $S \symdiff \{i\} \in \cF$ for all $i > m$. For this reason, for the rest of the section we treat a family having extent $m$ as a family on $m$ points.

One reason to focus on the boundary generating family of $\cF$ is the following simple observations.

\begin{lemma} \label{lem:gs-0a}
 Let $\cF$ be a non-trivial monotone left-compressed family of extent $m$ with generating family $\cG$ and boundary generating family $\cG^*$. For any subset $G \subseteq \cG^*$,
\[
 \upset{\cG \setminus G} = \cF \setminus G.
\]
\end{lemma}
\begin{proof}
 Since $\cG$ is an antichain, no $A \in G$ is a superset of any other set in $\cG$. For this reason, $\upset{\cG \setminus G} \subseteq \cF \setminus G$.
 
 On the other hand, let $S \in \cF \setminus G$. If $S$ is not a superset of any $A \in G$ then clearly $S \in \upset{\cG \setminus G}$. If $S \supseteq A$ for some $A \in G$ then since $S \neq A$, there is an element $i \in S \setminus A$. The set $S' = S \setminus \{m\}$ is a superset of $(A \setminus \{m\}) \cup \{i\}$, and so $S' \in \cF$. Thus $S'$ is a superset of some $B \in \cG$. Since $m \notin S'$, necessarily $B \notin G$. As $S \supseteq B$, we conclude that $S \in \upset{\cG \setminus G}$.
\end{proof}

\begin{lemma} \label{lem:gs-0b}
 Let $\cF$ be a non-trivial monotone left-compressed family of extent $m$ with generating family $\cG$ and boundary generating family $\cG^*$. For any subset $G \subseteq \cG^*$,
\[
 \upset{(\cG \setminus G) \cup \{ A \setminus \{m\} : A \in G\}} = \cF \cup \{ A \setminus \{m\} : A \in G \}.
\]
\end{lemma}
\begin{proof}
 Denote by $\cF'$ the left-hand side. Clearly $\cF' \supseteq \cF \cup \{ A \setminus \{m\} : A \in G \}$.
 
 On the other hand, suppose that $S \in \cF' \setminus \cF$. Then for some $A \in G$, $S$ is a superset of $A \setminus \{m\}$ but not of $A$. In particular, $m \notin S$. We claim that $S = A \setminus \{m\}$. Otherwise, there exists an element $i \in S \setminus (A \setminus \{m\})$. Since $\cF$ is left-compressed, $A' = (A \setminus \{m\}) \cup \{i\} \in \cF$. Since $\cF$ is monotone and $S \supseteq A'$, we conclude that $S \in \cF$, contradicting our assumption. Thus $\cF' \setminus \cF = \{ A \setminus \{m\} : A \in G \}$.
\end{proof}

The following crucial observation drives the entire approach, and explains why we want to classify the sets in the boundary generating family according to their size.

\begin{lemma} \label{lem:gs-1}
 Let $\cF$ be a non-trivial monotone left-compressed $t$-intersecting family with extent $m$ and boundary generating family $\cG^*$. If $A,B \in \cG^*$ intersect in exactly $t$~elements then $|A| + |B| = m + t$.
\end{lemma}
\begin{proof}
 We will show that $A \cup B = [m]$. It follows that
\[
 |A| + |B| = |A \cup B| + |A \cap B| = m + t.
\]
 Since $A \cup B \subseteq [m]$ and $m \in A \cap B$ by definition, we have to show that every element $i < m$ belongs to either $A$ or $B$. Suppose that some element $i$ belongs to neither. Since $\cF$ is left-compressed, the set $B' = (B \setminus \{m\}) \cup \{i\}$ also belongs to $\cF$. However, $|A \cap B'| = |A \cap B| - 1 = t - 1$, contradicting the assumption that $\cF$ is $t$-intersecting.
\end{proof}

Our goal now is to show that if $m$ is too large then we can remove the dependency on~$m$ while keeping the family $t$-intersecting and increasing its $\mu_p$-measure, for appropriate values of~$p$. The idea is to remove $m$ from sets in the boundary generating family. The only obstructions for doing so are sets $A,B$ in the boundary generating family whose intersection contains \emph{exactly} $t$~elements, and here we use Lemma~\ref{lem:gs-1} to guide us: this can only happen if $|A| + |B| = m + t$. Accordingly, our modification will involve generating sets in $\cG^*_a$ and $\cG^*_b$ for $a + b = m + t$. There are two cases to consider: $a \neq b$ and $a = b$. The first case is simpler.

\begin{lemma} \label{lem:gs-2}
 Let $\cF$ be a non-trivial monotone left-compressed $t$-intersecting family with extent $m$, generating family $\cG$, and boundary generating family $\cG^*$. Let $a \neq b$ be parameters such that $a + b = m + t$ and $\cG^*_a,\cG^*_b$ are not both empty. Consider the families $\cF_1 = \upset{\cG_1}$ and $\cF_2 = \upset{\cG_2}$, where
\begin{align*}
 \cG_1 &= (\cG \setminus (\cG^*_a \cup \cG^*_b)) \cup \{ S \setminus \{m\} : S \in \cG^*_b \}, \\
 \cG_2 &= (\cG \setminus (\cG^*_a \cup \cG^*_b)) \cup \{ S \setminus \{m\} : S \in \cG^*_a \}.
\end{align*}
 Both families $\cF_1,\cF_2$ are $t$-intersecting. Moreover, if $p < 1/2$ then $\max(\mu_p(\cF_1),\mu_p(\cF_2)) > \mu_p(\cF)$; and if $p = 1/2$, $\max(\mu_p(\cF_1),\mu_p(\cF_2)) \geq \mu_p(\cF)$, with equality only if $\mu_p(\cF_1) = \mu_p(\cF_2) = \mu_p(\cF)$. 
\end{lemma}
\begin{proof}
 We start by showing that $\cF_1$ (and so $\cF_2$) is $t$-intersecting. Clearly, it is enough to show that its generating family $\cG_1$ is $t$-intersecting. Suppose that $S,T \in \cG_1$. We consider several cases.
 
 If $S,T \in \cG$ then $|S \cap T| \geq t$ since $\cG$ is $t$-intersecting.
 
 If $S \in \cG$ and $T \notin \cG$ then $T' = T \cup \{m\} \in \cG$ and so $|T'| = b$. If $m \notin S$ then $|S \cap T| = |S \cap T'| \geq t$. If $m \in S$ then by construction $|S| \neq a$, and so $|S \cap T| = |S \cap T'| - 1 \geq t$, using Lemma~\ref{lem:gs-1}.
 
 If $S,T \notin \cG$ then $S' = S \cup \{m\} \in \cG$ and $T' = T \cup \{m\} \in \cG$, and so $|S'| = |T'| = b$. As in the preceding case, $|S \cap T| = |S' \cap T'| - 1 \geq t$ due to Lemma~\ref{lem:gs-1}.
 
 \smallskip
  
 Lemma~\ref{lem:gs-0a} and Lemma~\ref{lem:gs-0b} show that
\[
 \cF_1 = (\cF \setminus \cG^*_a) \cup \{ S \setminus \{m\} : S \in \cG^*_b \}.
\]
 Since  $\mu_p(S \setminus \{m\}) = \frac{1-p}{p} \mu_p(S)$ whenever $m \in S$,
\[
 \mu_p(\cF_1) = \mu_p(\cF) - \mu_p(\cG^*_a) + \frac{1-p}{p} \mu_p(\cG^*_b).
\]
 Similarly,
\[
 \mu_p(\cF_2) = \mu_p(\cF) - \mu_p(\cG^*_b) + \frac{1-p}{p} \mu_p(\cG^*_a).
\]
 Taking the average of both estimates, we obtain
\[
 \frac{\mu_p(\cF_1) + \mu_p(\cF_2)}{2} = \mu_p(\cF) + \frac{1}{2}\left(\frac{1-p}{p} - 1\right) (\mu_p(\cG^*_a) + \mu_p(\cG^*_b)).
\]
 When $p < 1/2$, the second term is positive, and so $\max(\mu_p(\cF_1), \mu_p(\cF_2)) > \mu_p(\cF)$. When $p = 1/2$ it vanishes, and so $\max(\mu_p(\cF_1),\mu_p(\cF_2)) \geq \mu_p(\cF)$.
\end{proof}

When $a = b$ the construction in Lemma~\ref{lem:gs-2} cannot be executed, and we need a more complicated construction. The new construction will only work for small enough $p$, mirroring the fact that the optimal families for larger $p$ depend on more points.

\begin{lemma} \label{lem:gs-3}
 Let $\cF$ be a non-trivial monotone left-compressed $t$-intersecting family with extent $m>1$, generating family $\cG$, and boundary generating family $\cG^*$. Suppose that $a = \frac{m+t}{2}$ is an integer and that $\cG^*_a$ is non-empty. For each $i \in [m-1]$, let $\cG^*_{a,i} = \{ S \in \cG^*_a : i \in S \}$, and define
\[
 \cG_i = (\cG \setminus \cG^*_a) \cup \{ S \setminus \{m\} : S \in \cG^*_a \setminus \cG^*_{a,i} \}.
\]
 All families $\cF_i = \upset{\cG_i}$ are $t$-intersecting. Moreover, if $p < \frac{m-t}{2(m-1)}$ then $\mu_p(\cF_i) > \mu_p(\cF)$ for some $i \in [m-1]$. 
\end{lemma}
\begin{proof}
 We start by showing that the families $\cF_i$ are $t$-intersecting. Clearly, it is enough to show that $\cG_i$ is $t$-intersecting. Let $S,T \in \cG_i$. We consider several cases.
 
 If $S,T \in \cG$ then $|S \cap T| \geq t$ since $\cG$ is $t$-intersecting.
 
 If $S \in \cG$ and $T \notin \cG$ then $T' = T \cup \{m\} \in \cG^*_a$ and $i \notin S$. If $m \notin S$ then $|S \cap T| = |S \cap T'| \geq t$. If $m \in S$ and $S \notin \cG^*_a$ then $|S \cap T| \geq |S \cap T'| - 1 \geq t$, according to Lemma~\ref{lem:gs-1}. If $m \in S$ and $S \in \cG^*_a$ then by construction $i \notin S$. Since $\cF$ is left-compressed, $S' = (S \setminus \{m\}) \cup \{i\} \in \cF$. Therefore $|S \cap T| = |S' \cap T'| \geq t$.
 
 If $S,T \notin \cG$ then $S' = S \cup \{m\}$ and $T' = T \cup \{m\}$ both belong to $\cG^*_a$, and $i$ belongs to neither. Since $\cF$ is left-compressed, $T'' = T \cup \{i\} \in \cF$, and so $|S \cap T| = |S' \cap T''| \geq t$.

 \smallskip

 Lemma~\ref{lem:gs-0a} and Lemma~\ref{lem:gs-0b} show that
\[
 \cF_i = (\cF \setminus \cG^*_{a,i}) \cup \{ S \setminus \{m\} : S \in \cG^*_a \setminus \cG^*_{a,i} \}.
\]
 Since $\mu_p(S \setminus \{m\}) = \frac{1-p}{p} \mu_p(S)$ whenever $m \in S$,
\begin{align*}
 \mu_p(\cF_i) &= \mu_p(\cF) - \mu_p(\cG^*_{a,i}) + \frac{1-p}{p} (\mu_p(\cG^*_a) - \mu_p(\cG^*_{a,i})) \\ &=
 \mu_p(\cF) + \frac{1-p}{p} \mu_p(\cG^*_a) - \frac{1}{p} \mu_p(\cG^*_{a,i}).
\end{align*}
 Averaging over all $i \in [m-1]$, we obtain
\[
 \frac{1}{m-1} \sum_{i=1}^{m-1} \mu_p(\cF_i) = \mu_p(\cF) + \frac{1-p}{p} \mu_p(\cG^*_a) - \frac{1}{p(m-1)} \sum_{i=1}^{m-1} \mu_p(\cG^*_{a,i}).
\]
 Since the sets in $\cG^*_a$ contain exactly $a$ elements, each set is counted $a-1$ times in $\sum_{i=1}^{m-1} \mu_p(\cG^*_{a,i})$, and so
\[
 \frac{1}{m-1} \sum_{i=1}^{m-1} \mu_p(\cF_i) = \mu_p(\cF) + \left(\frac{1-p}{p} - \frac{a-1}{p(m-1)}\right) \mu_p(\cG^*_a).
\]
 When $1-p > \frac{a-1}{m-1} = \frac{m+t-2}{2(m-1)}$, the bracketed quantity is positive, and so $\max_i \mu_p(\cF_i) > \mu_p(\cF)$.
\end{proof}

\subsection{Pushing-pulling} \label{sec:symmetrization}

The goal in the second part of the proof, which follows~\cite{AK3}, is to show that any monotone left-compressed $t$-intersecting family of maximum $\mu_p$-measure is symmetric within its extent, or in other words, of the form $\cF_{t,r}$.

The analog of extent in this part is the \emph{symmetric extent}.

\begin{definition} \label{def:sym-extent}
Let $\cF$ be a left-compressed family on $n$ points. Its \emph{symmetric extent}	is the largest integer $\ell$ such that $\shift_{ij}(\cF) = \cF$ for $i,j \leq \ell$.

If $\ell < n$ then the \emph{boundary} of $\cF$ is the collection
\[
 \cX = \{ A \in \cF : \ell+1 \notin A \text{ and } (A \setminus \{i\}) \cup \{\ell+1\} \text{ for some } i  \in A \cap [\ell] \}.
\]
In other words, $\cX$ consists of those sets in $\cF$ preventing it from having larger symmetric extent.

The definition of symmetric extension guarantees that $\cX$ can be decomposed as
\[
 \cX = \sum_{a=0}^\ell \binom{[\ell]}{a} \times \cX_a,
\]
where $\cX_a$ is a collection of subsets of $[n] \setminus [\ell+1]$, a notation we use below.
\end{definition}

The symmetric extent of a family is always bounded by its extent, apart from one trivial case.

\begin{lemma} \label{lem:extents}
Let $\cF$ be a non-trivial monotone family on $n$ points having extent $m$ and symmetric extent $\ell$. Then $\ell \leq m$.
\end{lemma}
\begin{proof}
 The family $\cF$ has the general form
\[
 \cF = \bigcup_{i=0}^\ell \binom{[\ell]}{i} \times \cF_i,
\]
 where $\cF_1,\ldots,\cF_\ell$ are collections of subsets of $[n] \setminus [\ell]$. We claim that if $m < \ell$ then all $\cF_i$ are equal. Indeed, let $i < \ell$. For each $A \in \cF_i$, we have $[i] \cup A \in \cF$. Since the extent of $\cF$ is smaller than $\ell$, $[i] \cup \{\ell\} \cup A \in \cF$, implying $A \in \cF_{i+1}$. Similarly, for each $A \in \cF_{i+1}$ we have $[i] \cup \{\ell\} \cup A \in \cF$, and so $[i] \cup \{\ell\} \in \cF$, implying $A \in \cF_i$.
 
 We have shown that $\cF = 2^{[\ell]} \times \cF_0$. Since the extent of $m$ is at most $\ell$, necessarily $\cF = 2^{[n]}$.
\end{proof}

The following crucial observation is the counterpart of Lemma~\ref{lem:gs-1}.

\begin{lemma} \label{lem:sym-1}
Let $\cF$ be a left-compressed $t$-intersecting family of symmetric extent $\ell$ and boundary $\cX$. If $|A \cap B| = t$ for some $A,B \in \cX$ then $|A \cap [\ell]| + |B \cap [\ell]| = \ell + t$.
\end{lemma}
\begin{proof}
 We start by showing that $A \cap B \subseteq [\ell]$. Indeed, suppose that $i \in A \cap B$ for some $i > \ell$. Since neither of $A,B$ contains $\ell+1$, in fact $i > \ell+1$. Since $\cF$ is left-compressed, $A' = (A \setminus \{i\}) \cup \{\ell+1\} \in \cF$. However, $|A' \cap B| = |A \cap B| - 1 = t-1$, contradicting the assumption that $\cF$ is $t$-intersecting.
 
 Next, we show that $A \cup B \supseteq [\ell]$. Indeed, suppose that $i \notin A \cup B$ for some $i \in \ell$. By definition of $\cX$, the set $A$ must contain some element $j \in [\ell]$. By definition of symmetric extent (if $j < i$) or by the fact that $\cF$ is left-compressed (if $j > i$), $A' = (A \setminus \{j\}) \cup \{i\} \in \cF$. However, $|A' \cap B| = |A \cap B| - 1 = t-1$, contradicting the assumption that $\cF$ is $t$-intersecting.
 
 Finally, let $A' = A \cap [\ell]$ and $B' = B \cap [\ell]$. Since $A' \cap B' = A \cap B$ and $A' \cup B' = [\ell]$, we deduce that
\[
 |A'| + |B'| = |A' \cup B'| + |A' \cap B'| = \ell + t. \qedhere
\]
\end{proof}

Our goal now is to try to eliminate $\cX$, thus increasing the symmetric extent. We do this by trying to add $\binom{[\ell]}{a-1} \times \{\ell+1\} \times \cX_a$ to $\cF$. The obstructions are described by Lemma~\ref{lem:sym-1}, which explains why we decompose $\cX$ according to the size of the intersection with $[\ell]$. Accordingly, our modification will act on the sets in $\cX_a,\cX_b$ for $a + b = \ell + t$. As in the preceding section, we have to consider two cases, $a \neq b$ and $a = b$, and the first case is simpler.

\begin{lemma} \label{lem:sym-2}
Let $\cF$ be a left-compressed $t$-intersecting family on $n$ points of symmetric extent $\ell < n$. Let $a \neq b$ be parameters such that $a + b = \ell + t$ and $\cX_a,\cX_b$ are not both empty. Consider the two families
\begin{align*}
\cF_1 &= \left(\cF \setminus \binom{[\ell]}{b} \times \cX_b\right) \cup \binom{[\ell]}{a-1} \times \{\ell+1\} \times \cX_a, \\
\cF_2 &= \left(\cF \setminus \binom{[\ell]}{a} \times \cX_a\right) \cup \binom{[\ell]}{b-1} \times \{\ell+1\} \times \cX_b.
\end{align*}
Both families $\cF_1,\cF_2$ are $t$-intersecting. Moreover, if $t > 1$ then for all $p \in (0,1)$, $\max(\mu_p(\cF_1), \mu_p(\cF_2)) > \mu_p(\cF)$.
\end{lemma}
\begin{proof}
We start by showing that $\cF_1$ (and so $\cF_2$) is $t$-intersecting. Suppose that $S,T \in \cF_1$. We consider several cases.

If $S,T \in \cF$ then $|S \cap T| \geq t$ since $\cF$ is $t$-intersecting.

If $S \in \cF$ and $T \notin \cF$ then $T \in \binom{[\ell]}{a-1} \times \{\ell+1\} \times \cX_a$. Choose $i \in [\ell] \setminus T$ arbitrarily, and notice that $T' = (T \setminus \{\ell+1\}) \cup \{i\} \in \binom{[\ell]}{a} \times \cX_a$, and so $T' \in \cF$. If $i \notin S$ or $\ell+1 \in S$ then $|S \cap T| \geq |S \cap T'| \geq t$. Suppose therefore that $i \in S$ and $\ell+1 \notin S$. If $S' = (S \setminus \{i\}) \cup \{\ell+1\} \in \cF$ then $|S \cap T| = |S' \cap T'| \geq t$. Otherwise, by definition of $\cX$, $S \in \cX$. By definition of $\cF_1$, $|S \cap [\ell]| \neq b$, and so Lemma~\ref{lem:sym-1} shows that $|S \cap T| \geq |S \cap T'| - 1 \geq t$.

If $S,T \notin \cF$ then $S,T \in \binom{[\ell]}{a-1} \times \{\ell+1\} \times \cX_a$. Choose $i \in [\ell] \setminus S$ and $j \in [\ell] \setminus T$ arbitrarily, and define $S' = (S \setminus \{\ell+1\}) \cup \{i\}$ and $T' = (T \setminus \{\ell+1\}) \cup \{j\}$. As before, $S',T' \in \cX$, and so Lemma~\ref{lem:sym-1} shows that $|S' \cap T'| \geq t+1$. Since $S \cap T \supseteq ((S' \cap T') \setminus \{i,j\}) \cup \{\ell+1\}$, we see that $|S \cap T| \geq |S' \cap T'| - 1 \geq t$.

\smallskip

We calculate the measures of $\cF_1$ and $\cF_2$ in terms of the quantities $m_a = \mu_p(\binom{[\ell]}{a} \times \cX_a)$ and $m_b = \mu_p(\binom{[\ell]}{b} \times \cX_b)$:
\begin{align*}
 \mu_p(\cF_1) &= \mu_p(\cF) - m_b + \frac{\binom{\ell}{a-1}}{\binom{\ell}{a}} m_a = \mu_p(\cF) - m_b + \frac{a}{\ell-a+1} m_a, \\
 \mu_p(\cF_2) &= \mu_p(\cF) - m_a + \frac{\binom{\ell}{b-1}}{\binom{\ell}{b}} m_b = \mu_p(\cF) - m_a + \frac{b}{\ell-b+1} m_b.	
\end{align*}
Multiply the first inequality by $\frac{\ell-a+1}{\ell-t+2}$, the second inequality by $\frac{\ell-b+1}{\ell-t+2}$, and add; note that $\ell-t+2 = (\ell-a+1) + (\ell-b+1) > 0$. The result is
\begin{align*}
 \frac{\ell-a+1}{\ell-t+2} \mu_p(\cF_1) +
 \frac{\ell-b+1}{\ell-t+2} \mu_p(\cF_2) &=
 \mu_p(\cF) + \left[\frac{a}{\ell-t+2} - \frac{\ell-b+1}{\ell-t+2}\right] m_a + \left[\frac{b}{\ell-t+2} - \frac{\ell-a+1}{\ell-t+2}\right] m_b \\ &=
 \mu_p(\cF) + \frac{t-1}{\ell-t+2} (m_a + m_b),
\end{align*}
using $a + b = \ell + t$. We conclude that when $t > 1$, $\max(\mu_p(\cF_1),\mu_p(\cF_2)) > \mu_p(\cF)$.
\end{proof}

When $a = b$, the construction increases the extent $m$ (defined in the preceding section), and works only for large enough $p$.

\begin{lemma} \label{lem:sym-3}
Let $\cF$ be a non-trivial monotone left-compressed $t$-intersecting family on $n$ points of extent $m < n$ and symmetric extent $\ell$, and let $s \in [n]$ be an index satisfying $s > m$ and $s \neq \ell+1$ (such an element exists if $\ell < m$ or if $m \leq n-2$). Suppose that $a = \frac{\ell+t}{2}$ is an integer and that $\cX_a$ is non-empty. Let $\cX'_a = \{ S \in \cX_a : s \in S \}$ and define
\[
 \cF' = \left(\cF \setminus \binom{[\ell]}{a} \times \cX_a\right) \cup \binom{[\ell+1]}{a} \times \cX'_a.
\]
The family $\cF'$ is $t$-intersecting. Moreover, if $p > \frac{\ell-t+2}{2(\ell+1)}$ then $\mu_p(\cF') > \mu_p(\cF)$.

\end{lemma}
\begin{proof}
We start by showing that $\cF'$ is $t$-intersecting. Suppose that $S,T \in \cF'$. We consider several cases.

If $S,T \in \cF$ then $|S \cap T| \geq t$ since $\cF$ is $t$-intersecting.

If $S \in \cF$ and $T \notin \cF$ then $T \in \binom{[\ell+1]}{a} \times \cX'_a$ and $\ell+1,s \in T$. Choose $i \in [\ell] \setminus T$ arbitrarily, and notice that $T' = (T \setminus \{\ell+1\}) \cup \{i\} \in \binom{[\ell]}{a} \times \cX_a \in \cF$. If $i \notin S$ or $\ell+1 \in S$ then $|S \cap T| \geq |S \cap T'| \geq  t$. Suppose therefore that $i \in S$ and $\ell+1 \notin S$. If $S' = (S \setminus \{i\}) \cup \{\ell+1\} \in \cF$ then $|S \cap T| = |S' \cap T'| \geq t$. Otherwise, $S \in \cX$. If $|S \cap [\ell]| \neq a$ then Lemma~\ref{lem:sym-1} shows that $|S \cap T| \geq |S \cap T'|-1 \geq t$. If $|S \cap [\ell]| = a$ then by construction, $s \in S$. Since the extent of $\cF$ is $m < s$, also $S' = S \setminus \{s\} \in \cF$. Therefore $|S \cap T| \geq |S' \cap T'| \geq t$, since $s \in S \cap T$ but $s \notin S'$.

If $S,T \notin \cF$ then $S,T \in \binom{[\ell+1]}{a} \times \cX'_a$ and $\ell+1,s \in S,T$. Choose $i \in [\ell] \setminus S$ and $j \in [\ell] \setminus T$, so that $S' = (S \setminus \{\ell+1\}) \cup \{i\}$ and $T' = (T \setminus \{\ell+1\}) \cup \{j\}$ are both in $\cF$. By construction, $s$ belongs to $S$ and $T$ and so to $S'$ and $T'$. Since the extent of $\cF$ is~$m < s$, $S'' = S \setminus \{s\}$ and $T'' = T \setminus \{s\}$ also belong to $\cF$. Observe that $S \cap T \subseteq ((S'' \cap T'') \setminus \{i,j\}) \cup \{\ell+1,s\}$, and so $|S \cap T| \geq |S'' \cap T''| \geq t$.

\smallskip

We calculate the measure of $\cF'$ in terms of the quantity $m_a = \mu_p(\binom{[\ell]}{a} \times \cX_a)$:
\begin{align*}
 \mu_p(\cF') &= \mu_p(\cF) - m_a + p \frac{\binom{\ell+1}{a}}{\binom{\ell}{a}} m_a \\ &= \mu_p(\cF) - m_a + p \frac{\ell+1}{\ell+1-a} m_a \\ &= \mu_p(\cF) + \frac{a - (1-p)(\ell+1)}{\ell+1-a} m_a.
\end{align*}
Thus $\mu_p(\cF') > \mu_p(\cF)$ as long as $1-p < \frac{a}{\ell+1} = \frac{\ell+t}{2(\ell+1)}$.
\end{proof}

Lemma~\ref{lem:sym-3} cannot be applied when $m = n$. However, if $n$ has the correct parity, we can combine Lemma~\ref{lem:sym-3} with Lemma~\ref{lem:gs-2} to handle this issue.

\begin{lemma} \label{lem:sym-3plus}
Let $\cF$ be a non-trivial monotone left-compressed $t$-intersecting family on $n$ points of extent $m$ and symmetric extent $\ell$, where either $\ell < m$ or $m < n$.
If $n + t$ is even and $\frac{\ell-t+2}{2(\ell+1)} < p \leq \frac{1}{2}$ then there exists a $t$-intersecting family on $n$ points with larger $\mu_p$-measure.
\end{lemma}
\begin{proof}
 Consider first the case $\ell < m$.
 If $m < n$ then the statement follows from Lemma~\ref{lem:sym-2} and Lemma~\ref{lem:sym-3}, so suppose that $m = n$. Let $\cF' = \cF \cup \cF \times \{n+1\}$, and note that this is a non-trivial monotone left-compressed $t$-intersecting family on $n+1$ points. We can apply Lemma~\ref{lem:sym-3} to obtain a non-trivial monotone left-compressed $t$-intersecting family $\cG$ on $n+1$ points satisfying $\mu_p(\cG) > \mu_p(\cF') = \mu_p(\cF)$. Since $n+1+t$ is odd, we can apply Lemma~\ref{lem:gs-2} repeatedly to obtain a non-trivial monotone $t$-intersecting family $\cH$ on $n+1$ points and extent $n$ which satisfies $\mu_p(\cH) \geq \mu_p(\cG) > \mu_p(\cF)$. Since $\cH$ has extent $n$, there is a $t$-intersecting family on $n$ points having the same $\mu_p$-measure.
 
 Consider next the case $\ell = m < n$. If $m \leq n-2$ then the statement follows from Lemma~\ref{lem:sym-2} and Lemma~\ref{lem:sym-3}, so suppose that $m = n-1$. In this case $m + t$ is odd, and so the statement follows from Lemma~\ref{lem:gs-2}.
\end{proof}

\subsection{The case $p < 1/2$} \label{sec:psmall}

In this section we prove Theorem~\ref{thm:weighted-main} in the case $p < 1/2$. In view of Lemma~\ref{lem:compression}, it suffices to consider monotone left-compressed families.

We first settle the case $t = 1$, which corresponds to the classical Erd\H{o}s--Ko--Rado theorem (in the weighted setting).

\begin{lemma} \label{lem:psmall-1}
Let $\cF$ be a monotone left-compressed intersecting family on $n$ points of maximum $\mu_p$-measure, for some $p \in (0,1/2)$. Then $\cF = \cF_{1,0}$.	
\end{lemma}
\begin{proof}
 Let $m$ be the extent of $\cF$. Since $\frac{m-t}{2(m-1)} = 1/2$, Lemma~\ref{lem:gs-2} and Lemma~\ref{lem:gs-3} together show that $m = 1$, and so $\cF = \cF_{1,0}$.
\end{proof}

The case $t \geq 2$ requires more work.

\begin{lemma} \label{lem:psmall}
 Let $\cF$ be a monotone left-compressed $t$-intersecting family on $n$ points of maximum $\mu_p$-measure, for some $p \in (0,1/2)$ and $t > 1$. Let $r$ be the maximal integer satisfying $p \geq \frac{r}{t+2r-1}$ and $t+2r \leq n$. If $p \neq \frac{r}{t+2r-1}$ then $\cF = \cF_{t,r}$, and if $p = \frac{r}{t+2r-1}$ then $\cF \in \{\cF_{t,r},\cF_{t,r-1}\}$.
\end{lemma}
\begin{proof}
 Our definition of $r$ guarantees that one of the following two alternatives holds: either $p < \frac{r+1}{t+2r-1}$, or $n \leq t+2r+1$.

 Let $m$ be the extent of $\cF$. We claim that $m \leq t+2r$. If $n \leq t+2r+1$ then Lemma~\ref{lem:gs-2} shows that $m+t$ is even, and so $m \leq t+2r$. Suppose therefore that $p < \frac{r+1}{t+2r-1}$ and $m > t+2r$. Lemma~\ref{lem:gs-2} shows that in fact $m \geq t+2r+2$, and so
\[
 \frac{m-t}{2(m-1)} = \frac{1}{2} - \frac{t-1}{2(m-1)} \geq \frac{1}{2} - \frac{t-1}{2(t+2r+1)} = \frac{r+1}{t+2r+1}.
\]
 Therefore Lemma~\ref{lem:gs-2} and Lemma~\ref{lem:gs-3} contradict the assumption that $\cF$ has maximum $\mu_p$-measure.
 
 We now turn to consider the symmetric extent $\ell$ of $\cF$. We first consider the case in which $p > \frac{r}{t+2r-1}$. We claim that in this case $\ell = m$. If $\ell < m$ then Lemma~\ref{lem:gs-2} and Lemma~\ref{lem:sym-2} show that both $m+t$ and $\ell+t$ are even, and so $\ell \leq m-2 \leq t+2r-2$. This implies that
\[
 \frac{\ell-t+2}{2(\ell+1)} = \frac{1}{2} - \frac{t-1}{2(\ell+1)} \leq \frac{1}{2} - \frac{t-1}{2(t+2r-1)} = \frac{r}{t+2r-1}.
\]
 Therefore Lemma~\ref{lem:sym-2} and Lemma~\ref{lem:sym-3plus} contradict the assumption that $\cF$ has maximum $\mu_p$-measure.
 
 We have shown that if $p > \frac{r}{t+2r-1}$ then $\ell = m \leq t+2r$, and moreover $m+t$ is even. Thus $\ell = m = t + 2s$ for some $s \leq r$. Since $\cF$ is $t$-intersecting, $\cF \subseteq \cF_{t,s}$ for some $s \leq r$. The fact that $\cF$ has maximum $\mu_p$-measure forces $\cF = \cF_{t,s}$. In view of Lemma~\ref{lem:frankl-mup}, necessarily $s = r$.
 
 The case $p = \frac{r}{t+2r-1}$ is slightly more complicated. Suppose first that $\ell = m$. In that case, as before, $\cF = \cF_{t,s}$ for some $s \leq r$. This time Lemma~\ref{lem:frankl-mup} shows that $s \in \{r,r-1\}$.
 
 Suppose next that $\ell < m$. The same argument as before shows that $\ell \geq m-2$. Lemma~\ref{lem:gs-2} and Lemma~\ref{lem:sym-2} show that both $\ell+t$ and $m+t$ are even, and so $\ell = m-2$ in this case. In the remainder of the proof, we show that this leads to a contradiction. To simplify notation, we will assume that $m = n$. As $m+t$ is even, we can write $m = t+2s$ for some $s \leq r$.
 
 Since $\cF$ is monotone and has symmetric extent $m-2$, it can be decomposed as follows:
\begin{gather*}
 \cF = \binom{[t+2s-2]}{\geq a} \cup \binom{[t+2s-2]}{\geq b} \times \{t+2s+1\} \cup \\ \binom{[t+2s-2]}{\geq c} \times \{t+2s\} \cup \binom{[t+2s-2]}{\geq d} \times \{t+2s-1,t+2s\}.
\end{gather*}
 Since the family is $t$-intersecting, we must have $2d-(t+2s-2)+2 \geq t$, and so $d \geq t+s-2$. If $d \geq t+s-1$ then monotonicity implies that $a,b,c \geq d \geq t+s-1$, and so $\cF \subseteq \cF_{t,s-1}$. Since $\cF$ has maximum $\mu_p$-measure, necessarily $\cF = \cF_{t,s-1}$, in which case the extent is $t+2s-2$, contrary to assumption.

 We conclude that $d = t+s-2$. The fact that $\cF$ is $t$-intersecting implies that $c+d-(t+2s-2)+1 \geq t$, and so $c \geq t+s-1$. Similarly $b \geq t+s-1$, and moreover $a+d-(t+2s-2) \geq t$, implying $a \geq t+s$. Thus $\cF \subseteq \cF_{t,s}$. Since $\cF$ has maximum $\mu_p$-measure, necessarily $\cF = \cF_{t,s}$, in which case the symmetric extent is $t+2s$, contrary to assumption.
\end{proof}

\subsection{The case $p = 1/2$} \label{sec:phalf}

In this section we prove Theorem~\ref{thm:weighted-main} in the case $p = 1/2$. This case is known as Katona's theorem, after Katona's paper~\cite{Katona2}. In view of Lemma~\ref{lem:compression}, it suffices to consider monotone left-compressed families.

We first settle the case $t = 1$.

\begin{lemma} \label{lem:phalf-1}
 If $\cF$ is intersecting then $\mu_{1/2}(\cF) \leq 1/2$.
\end{lemma}
\begin{proof}
 Since $\cF$ is intersecting, for any set $A \subseteq [n-1]$ it can contain either $A$ or $[n] \setminus A$, but not both. Thus $\cF$ contains at most $2^{n-1}$ sets, implying $\mu_{1/2}(\cF) \leq 1/2$.
\end{proof}

The case $t \geq 2$ requires more work.

\begin{lemma} \label{lem:phalf}
 Let $\cF$ be a monotone left-compressed $t$-intersecting family on $n$ points of maximum $\mu_{1/2}$-measure, for some $t > 1$.	If $n \in \{t + 2r,t + 2r+1\}$ then $\cF = \cF_{t,r}$.
\end{lemma}
\begin{proof}
 Let $m$ be the extent of $\cF$, and $\ell$ be its symmetric extent.

 Suppose first that $n = t + 2r$. Lemma~\ref{lem:sym-3plus} shows that $\ell = m = n$, which easily implies $\cF = \cF_{t,r}$.	
 
 Suppose next that $n = t + 2r + 1$. If $m \leq t + 2r$ then the previous case $n = t + 2r$ shows that $\cF = \cF_{t,r}$, so suppose that $m = n$. Since $m + t$ is odd, Lemma~\ref{lem:gs-2} shows that there is a family $\cH$ of extent at most $m - 1 = t + 2r$ such that $\mu_p(\cH) \geq \mu_p(\cF)$. In view of the preceding case, this shows that $\cH = \cF_{t,r}$, and so $\mu_p(\cF) \leq \mu_p(\cF_{t,r})$. It remains to show that $\cF = \cF_{t,r}$ when $\mu_p(\cF) = \mu_p(\cF_{t,r})$.
 
 The family $\cH$ is constructed by repeatedly applying the following operation, where $\cG^*$ is the boundary generating family of $\cF$, and $a + b = n + t$: remove $\cG^*_a$ and $\cG^*_b$, and add either $\{S \setminus \{m\} : S \in \cG_a^*\}$ or $\{S \setminus \{m\} : S \in \cG_b^*\}$. All options must lead eventually to the same family $\cF_{t,r}$, and this can only happen if $\cG^*_a = \cG^*_b = \emptyset$ for all $a,b$. However, in that case the extent of $\cF$ is in fact $n-1$, contradicting our assumption.
\end{proof}

Ahlswede and Khachatrian~\cite{AK4} give a different shifting proof of this theorem (without characterizing the equality cases).

\subsection{The case $p > 1/2$} \label{sec:plarge}

In this section we prove Theorem~\ref{thm:weighted-main} in the case $p > 1/2$. The proof in this case differs from that of the other cases: it uses a different shifting argument, also due to Ahlswede and Khachatrian~\cite{AK4}, who used it for the case $p = 1/2$.

The idea is to use a different kind of shifting.

\begin{definition} \label{def:heavy-shifting}
 Let $\cF$ be a family on $n$ points. For two disjoint sets $A,B \subseteq [n]$, the shift operator $\shift_{A,B}$ acts on $\cF$ as follows. Let $\cF_{A,B}$ consist of all sets in $S$ containing $A$ and disjoint from $B$. Then
\begin{align*}
 \shift_{A,B}(\cF) = (F \setminus \cF_{A,B})
 &\cup \{ S : S \in \cF_{A,B} \text{ and } (S \setminus A) \cup B \in \cF \} \\
 &\cup \{ (S \setminus A) \cup B : S \in \cF_{A,B} \text{ and } (S \setminus A) \cup B \notin \cF \}.
\end{align*}
 (This is a generalization of the original shifting operator: $\shift_{i,j}$ is the same as $\shift_{\{i\},\{j\}}$.) 
 
 A family is \emph{$(s,s+1)$-stable} if $\shift_{A,B}(\cF) = \cF$ for any disjoint sets $A,B$ of cardinalities $|A| = s$ and $|B| = s + 1$.
\end{definition}

This kind of shift is useful when $p > 1/2$ for the following simple reason.

\begin{lemma} \label{lem:heavy-shifting-weight}
 If $|B| > |A|$ then $\mu_p(\shift_{A,B}(\cF)) \geq \mu_p(\cF)$ for any $p \in (1/2,1)$, with equality if only if $\shift_{A,B}(\cF) = \cF$.
\end{lemma}
\begin{proof}
 If $\shift_{A,B}(\cF) = \cF$ then clearly both sets have the same measure. Otherwise, the shifting replaces certain sets $S$ with the corresponding $S' = (S \setminus A) \cup B$. Since $|S'| > |S|$ and $p > 1/2$, this increases the measure of the family.	
\end{proof}

When done correctly, $\shift_{A,B}$ preserves the property of being $t$-intersecting.

\begin{lemma} \label{lem:heavy-shifting-int}
 Let $\cF$ be a $t$-intersecting family on $n$ points, and let $A,B \subseteq [n]$ be disjoint sets of cardinalities $|A| = s$ and $|B| = s+1$. If $\cF$ is $(r,r+1)$-stable for all $r < s$ then $\shift_{A,B}(\cF)$ is $t$-intersecting 	as well.
\end{lemma}
\begin{proof}
 If $s = 0$ then $\shift_{A,B}(\cF) \subseteq \upset{\cF}$, and so $\shift_{A,B}(\cF)$	is $t$-intersecting. Suppose therefore that $s \geq 1$. Let $S,T \in \shift_{A,B}(\cF)$. We consider three cases.
 
\textbf{Case 1:} If $S,T \in \cF$ then $|S \cap T| \geq t$ since $\cF$ is $t$-intersecting.
 
\textbf{Case 2:} If $S,T \notin \cF$ then $S' = (S \setminus B) \cup A$ and $T' = (T \setminus B) \cup A$ are both in $\cF$, and so $|S \cap T| = |S' \cap T'| - |A| + |B| \geq t+1$.
 
\textbf{Case 3:} The remaining case is when $S \in \cF$ and $T \notin \cF$. In this case $T' = (T \setminus B) \cup A \in \cF$. Since $|S \cap T| = |S \cap T'| - |S \cap A| + |S \cap B| \geq t - |S \cap A| + |S \cap B|$, if $|S \cap B| \geq |S \cap A|$ then $|S \cap T| \geq t$. Suppose therefore that $|S \cap B| < |S \cap A|$.
 
 Let $r = \min(|S \cap A|, |B \setminus S| - 1)$. Since $|B \setminus S| = |B| - |S \cap B| > |B| - |S \cap A| \geq 1$, we see that $r \geq 0$. Also, clearly $r \leq s$. If $r = s$ then $A \subseteq S$ whereas $S \cap B = \emptyset$. Therefore $S' = (S \setminus A) \cup B \in \cF$, since otherwise $S$ would have been replaced by $S'$. Then $|S \cap T| = |S' \cap T'| \geq t$.
 
 Suppose therefore that $r < s$. Let $C$ be an arbitrary subset of $S \cap A$ of size $r$, and let $D$ be an arbitrary subset of $B \setminus S$ of size $r+1$. Since $\cF$ is $r$-stable, $S' = (S \setminus C) \cup D \in \cF$. Write
\begin{align*}
 S' &= ((S \cap A) \setminus C) \cup (S \setminus A) \cup D, \\
 T' &= (T \setminus B) \cup (A \cap S) \cup (A \setminus S).	
\end{align*}
 The term $(S \cap A) \setminus C$ is a subset of $A \cap S$; the term $S \setminus A$ intersects only $T \setminus B$; and $D \subseteq B$ is disjoint from $T'$. Therefore
\[
 S' \cap T' = ((S \cap A) \setminus C) \cup ((S \setminus A) \cap (T \setminus B)).
\]
 Since $T$ is disjoint from $A$, $(S \setminus A) \cap (T \setminus B) = S \cap (T \setminus B) = (S \cap T) \setminus (S \cap B)$. Thus
\[
 |S' \cap T'| = |S \cap T| - |S \cap B| + |S \cap A| - |C|.
\]
 Since $|S' \cap T'| \geq t$, to complete the proof it suffices to show that $|S \cap A| \leq |S \cap B| + r$ (recalling that $|C| = r$). 
 If $r = |S \cap A|$ then this clearly holds, and if $r = |B \setminus S| - 1$ then
\[
 |S \cap B| + r = |S \cap B| + |B \setminus S| - 1 = |B| - 1 = |A| \geq |S \cap A|. \qedhere
\] 
\end{proof}

As in the case of the simpler shifting operator $\shift_{i,j}$, we can convert any family to a stable family while maintaining its being $t$-intersecting, by repeatedly applying a shifting operation on sets $A,B$ with minimal $|A|$. 
One way of expressing this idea is the following lemma.

\begin{lemma} \label{lem:heavy-compression}
 Let $p \in (1/2,1)$ If $\cF$ is a $t$-intersecting family on $n$ points having maximum $\mu_p$-measure then $\cF$ is $(s,s+1)$-stable for all $s$.
\end{lemma}
\begin{proof}
 Suppose that $\cF$ is not $(s,s+1)$-stable for all $s$. Choose the minimal $s$ such that $\cF$ is not $(s,s+1)$-stable, say $\shift_{A,B}(\cF) \neq \cF$, where $|A|=s$ and $|B|=s+1$. Lemma~\ref{lem:heavy-shifting-int} shows that $\shift_{A,B}(\cF)$ is also $t$-intersecting, and Lemma~\ref{lem:heavy-shifting-weight} shows that $\mu_p(\shift_{A,B}(\cF)) > \mu_p(\cF)$. Thus $\cF$ is not a $t$-intersecting family of maximum $\mu_p$-measure.
\end{proof}

Ahlswede and Khachatrian~\cite{AK4} in fact show that if $\cF$ is $(r,r+1)$-stable for all $r < s$ then for $|A|=s$ and $|B|=s+1$, the family $\shift_{A,B}(\cF)$ is also $(r,r+1)$-stable for all $r < s$. While this simplifies the process of compression, we do not need this result here, and refer the reader to its proof in~\cite{AK4}.

The importance of stable families is the following simple observation.

\begin{lemma} \label{lem:heavy-sizes}
 If a $t$-intersecting family $\cF$ on $n$ points is $(s,s+1)$-stable for all $s$ then every $A,B \in \cF$ satisfy $|A| + |B| \geq n+t-1$.
\end{lemma}
\begin{proof}
 Let $A,B \in \cF$. If $A \cup B = [n]$ then $|A| + |B| = |A \cup B| + |A \cap B| \geq n + t$, so suppose $|A \cup B| < n$.
 
 Define $s = \min(|A \cap B|, n-|A \cup B|-1) \geq 0$, and choose a subset $C \subseteq A \cap B$ of size $s$ and a subset $D \subseteq [n] \setminus (A \cup B)$ of size $s+1$. Since $\cF$ is $(s,s+1)$-stable, $A' = (A \setminus C) \cup D \in \cF$. We have $|A' \cap B| = |A \cap B| - |C| = |A \cap B| - s$, showing that $|A \cap B| \geq s + t$. In particular, $s = n-|A \cup B|-1$, and so
 \[ |A| + |B| = |A \cup B| + |A \cap B| \geq (n - s - 1) + (s + t) = n + t - 1. \qedhere \]
\end{proof}

We comment that the same result holds, with the same proof, for cross-$t$-intersecting families $\cF,\cG$: any $A \in \cF$ and $B \in \cG$ satisfy $|A| + |B| \geq n+t-1$.

The bound $n+t-1$ is tight: when $n = t+2r+1$, the family $\cF_{t,r}$ is $(s,s+1)$-stable for all $s$, and two sets $A,B$ of size $t+r$ satisfy $|A| + |B| = n + r - 1$.

We need one more lemma, on uniform families.

\begin{lemma} \label{lem:plarge-uniform}
 Let $\cF \subseteq \binom{[t+2r+1]}{t+r}$ be a $t$-intersecting family of maximum size, and define a family $\cF'_{t,r}$, the uniform analog of $\cF_{t,r}$, as follows:
\[
 \cF'_{t,r} = \{ S \in \binom{[t+2r+1]}{t+r} : |S \cap [t+2r]| = t+r \}.
\]

 If $t \geq 2$ then $\cF$ is equivalent to $\cF'_{t,r}$ (that is, equals a similar family with $[t+2r]$ possibly replaced by some other subset of $[t+2r+1]$ of size $t+2r$), and if $t = 1$ then $|\cF| \leq |\cF'_{t,r}|$.
\end{lemma}
\begin{proof}
 Define $\cG = \{ \overline{A} : A \in \cF \}$ (where $\overline{A} = [t+2r+1] \setminus A$), so that $\cG \subseteq \binom{[t+2r+1]}{r+1}$. Since
\[
 |\overline{A} \cap \overline{B}| = |\overline{A \cup B}| = t+2r+1 - |A \cup B| = t+2r+1 - |A| - |B| + |A \cap B| = |A \cap B| - (t-1),
\]
 we see that the condition that $\cF$ is $t$-intersecting is equivalent to the condition that $\cG$ is intersecting.
 
 Since $r+1 \leq \frac{t+2r+1}{2}$ (with equality only for $t = 1$), the Erd\H{o}s--Ko--Rado theorem shows that $|\cG| \leq \binom{t+2r}{r} = \binom{t+2r}{t+r}$.
 Moreover, when $t \geq 2$, equality holds only when $\cG = \{ S \in \binom{[t+2r+1]}{r+1} : i \in S \}$ for some $i \in [t+2r+1]$. In that case, $\cF = \{ S \in \binom{[t+2r+1]}{t+r} : i \notin S \}$, and so $\cF$ is equivalent to $\cF'_{t,r}$ (in the family $\cF_{t,r}$ itself, $i = t+2r+1$).
\end{proof}

We can now prove Theorem~\ref{thm:weighted-main} in the case $p > 1/2$.

\begin{lemma} \label{lem:plarge}
 Let $\cF$ be a $t$-intersecting family on $n$ points of maximum $\mu_p$-measure, for some $p \in (1/2,1)$. Suppose that $n \in \{t+2r,t+2r+1\}$.
 
 If $t \geq 2$ or $n = t+2r$ then $\cF$ is equivalent to $\cF_{t,r}$.
 
 If $t = 1$ and $n = t+2r+1$ then $\mu_p(\cF) \leq \mu_p(\cF_{t,r})$ and $\cF = \cG \cup \binom{[t+2r+1]}{\geq t+r+1}$, where $\cG \subseteq \binom{[t+2r+1]}{t+r}$ contains exactly $\binom{t+2r}{t+r}$ sets.
\end{lemma}
\begin{proof}
 Lemma~\ref{lem:heavy-compression} shows that $\cF$ is $(s,s+1)$-stable for all $s$, and so Lemma~\ref{lem:heavy-sizes} shows that any $A,B \in \cF$ satisfy $|A| + |B| \geq n+t-1$. In particular, any set $A$ has cardinality at least $\frac{n+t-1}{2}$. We now consider two cases, according to the parity of $n+t$.
 
 Suppose first that $n = t+2r$. Then $\frac{n+t-1}{2} = t + r - \frac{1}{2}$, and so all sets in $\cF$ have cardinality at least $t + r$. In other words, $\cF \subseteq \cF_{t,r}$. Since $\cF$ has maximum $\mu_p$-measure, $\cF = \cF_{t,r}$.
 
 Suppose next that $n = t+2r+1$. Then $\frac{n+t-1}{2} = t + r$, and so all sets in $\cF$ have cardinality at least $t + r$. If $|A| \geq t+r$ and $|B| \geq t+r+1$ then $|A \cap B| \geq |A|+|B|-n = t$, and so the fact that $\cF$ has maximum $\mu_p$-measure shows that $\cF = \cG \cup \binom{[n]}{\geq t+r+1}$, where $\cG \subseteq \binom{[n]}{t+r}$ is $t$-intersecting.
 
 We can now complete the proof using Lemma~\ref{lem:plarge-uniform}. If $t \geq 2$ then $\cG$ is equivalent to $\cF'_{t,r}$, say
 \[ \cG = \{ S \in \binom{[t+2r+1]}{t+r} : |S \cap X| = t+r \}, \]
 where $|X| = t+2r$. Since any set of size at least $t+r+1$ intersects $X$ in at least $t+r$ points,
 \[ \cF = \{ S \subseteq [n] : |S \cap X| \geq t+r \}, \]
 and so $\cF$ is equivalent to $\cF_{t,r}$.
 
 When $t = 1$, Lemma~\ref{lem:plarge-uniform} shows that $|\cG| \leq |\cF'_{t,r}|$, and so the same reasoning shows that $\mu_p(\cF) \leq \mu_p(\cF_{t,r})$.
\end{proof}

\section{Katona's circle argument} \label{sec:katona}

Katona~\cite{Katona1} gave a particularly simple proof of the Erd\H{o}s--Ko--Rado theorem, using what has become known as the \emph{circle method}. The same proof goes through in the $\mu_p$ setting, with much the same proof, as we show in Section~\ref{sec:katona-int}. We are able to use the same argument to obtain a description of all optimal families.

The heart of Katona's argument is the following seemingly trivial observation: if $S$ is a measurable set on the unit circumference circle in which any two points are at distance at most $p < 1/2$, then the length of $S$ is at most $p$, the optimal sets being intervals. In Section~\ref{sec:katona-cross} we give two versions of this argument for the case of two sets: one in a discrete setting, and the other in a continuous setting. These results will be used later on when we prove Ahlswede--Khachatrian theorems for families on infinitely many points (Section~\ref{sec:infinite}) and for the Hamming scheme (Section~\ref{sec:hamming}).

\subsection{Intersecting families} \label{sec:katona-int}

The Erd\H{o}s--Ko--Rado theorem, in our setting, states that if $\cF$ is an intersecting family then $\mu_p(\cF) \leq p$ for all $p \leq 1/2$.
Katona~\cite{Katona1} gave a particularly simple proof of the theorem in its original setting, and we adapt his proof to our setting.

\begin{lemma} \label{lem:ekr-mup}
 Let $n \geq 1$ and $p \in (0,1/2)$. Then $\mu_p(\cF) \leq p$ for all intersecting families $\cF$.
 
 Moreover, if $\cF$ is an intersecting family on $n$ points such that $\mu_p(\cF) = p$ then $\cF$ is equivalent to a $(1,0)$-Frankl family. In other words, if $\mu_p(\cF) = p$ then for some $i \in [n]$ we have
\[
 \cF = \{ A \subseteq [n] : A \ni i \}.
\]
\end{lemma}
\begin{proof}
 The idea is to come up with a probabilistic model for the distribution $\mu_p$, and use it to show that $\mu_p(\cF) \leq p$. Since $\mu_p(\cF_{1,0}) = p$, this shows that $w(n,1,p) = p$. We will then use the same probabilistic model to identify all families satisfying $\mu_p(\cF) = p$.
 
 Let $\unitcirc$ be the circle of unit circumference. Choose $n$ points $x_1,\ldots,x_n$ at random on $\unitcirc$. Choose another point $t$ on $\unitcirc$ at random, and consider the arc $(t,t+p)$ (where $t+p$ is taken modulo~$1$). The set of indices of points $S_t$ which lie inside the arc has distribution $\mu_p$, and so $\mu_p(\cF) = \Pr[S_t \in \cF]$.
 
 We will prove that for \emph{each} setting of $x_1,\ldots,x_n$ we have $\Pr[S_t \in \cF] \leq p$, the probability taken over the choice of $t$. Let $I = \{ t : S_t \in \cF \}$, and note that $I$ is a union of intervals. The crucial observation is that if $t_1,t_2 \in I$ then the corresponding arcs $(t_1,t_1+p),(t_2,t_2+p)$ must intersect, and so $d(t_1,t_2) < p$, where $d(\cdot,\cdot)$ is shortest distance on the circumference of the circle. Consider now any $t_1 \in I$. All $t_2 \in I$ must lie in the interval $(t_1-p,t_1+p)$, and moreover for each $s \in (0,p)$, at most one of $t_1+s,t_1-p+s$ can be in $I$ (since $d(t_1+s,t_1+s-p) = p$). It follows that the measure of $I$ is at most $p$. In other words, $\Pr[S_t \in \cF] \leq p$, implying that $\mu_p(\cF) \leq p$.
  
 For future use, we also need to identify the cases in which $I$ has measure exactly $p$. We will show that $I$ must be an interval of length $p$. The argument above shows that $I$ is a union of non-empty intervals of two types, $J_1,\ldots,J_a \subseteq (t_1-p,t_1]$ and $K_1,\ldots,K_b \subseteq [t_1,t_1+p)$, such that the intervals $J_i+p,K_j$ together partition $[t_1,t_1+p)$. If $a = 0$ then $I = [t_1,t_1+p)$, and if $b = 0$ then $I = (t_1-p,t_1]$. If there is a unique interval $K_1$ which is of the form $[t_1,t_1+s)$ (or $[t_1,t_1+s]$) then $I = [t_1+s-p,t_1+s)$ (or $(t_1+s-p,t_1+s]$). Otherwise, there must be some interval $K_j$ whose left end-point $y$ is larger than $t_1$. There is a corresponding interval $J_i$ whose right end-point is $y-p$. However, since $p < 1/2$, a point on $J_i$ slightly to the left of $y-p$ has distance larger than $p$ from a point on $K_j$ slightly to the right of $y$, contradicting our assumptions. We conclude that $I$ must be an interval of length $p$.
 
 \smallskip
 
 We proceed to identify the families $\cF$ which achieve the upper bound, that is, satisfy $\mu_p(\cF) = p$. Clearly any family $\cF$ equivalent to $\cF_{1,0}$ satisfies $\mu_p(\cF) = p$. We will show that these are the only such families. First, notice that if $\mu_p(\cF) = p$ then $\cF$ must be monotone (otherwise its upset is an intersecting family of measure larger than $p$). Moreover, the set $I$ defined above is an interval of length $p$ almost surely, with respect to the choice of $x_1,\ldots,x_n$. In particular, there is a choice of $x_1,\ldots,x_n$, all distinct and none at distance exactly $p$, such that $I$ is an interval of length $p$, say $I = [y,y+p)$. For small $\epsilon > 0$, the arcs $(y+\epsilon,y+p+\epsilon),(y+p-\epsilon,y+2p-\epsilon)$ intersect at a small neighborhood of $y+p$. Since the corresponding sets $S_{y+\epsilon},S_{y+p-\epsilon}$ intersect, there must be some point $x_i = y+p$. We will show that $\{ x_i \} \in \cF$, and so monotonicity implies that $\cF = \{ A \subseteq [n] : A \ni i \}$.
 
 Let $A$ consist of all points in $(y,y+p)$. Thus $A \cup \{ x_i \} \in \cF$ (since $y+\epsilon \in I$ for small $\epsilon$) while $A \notin \cF$ (since $y-\epsilon \notin I$ for small $\epsilon$). Consider now the set of configurations $x'_1,\ldots,x'_n$ in which $x'_j \in (x'_i-p,x'_i)$ for all $j \in A$ and $x'_j \notin (x'_i-p,x'_i+p)$ for all $j \notin A \cup \{i\}$. This set of configurations has positive measure, and so there must exist one whose corresponding set $I'$ is an interval of measure $p$. By construction, $I'$ contains all points $t$ just to the right of $x'_i-p$ (since the corresponding $S_t$ is $A \cup \{i\}$) but not $x'_i-p$ (since the corresponding $S_{x'_i-p}$ is $A$), and so $x'_i-p$ is the left end-point of the interval. The right end-point is thus $x'_i$ itself. By construction, $S_{x'_i-\epsilon} = \{i\}$ for small enough $\epsilon > 0$, showing that $\{i\} \in \cF$. This completes the proof.
\end{proof}

Other proofs of the upper bound which employ similar arguments appear in Dinur and Friedgut~\cite{DinurFriedgut} and in Friedgut~\cite{Friedgut05}.
Unfortunately, it seems that Katona's idea doesn't extent to $t$-intersecting families for $t > 1$. For a discussion of this, see Howard, K\'arolyi and Sz\'ekely~\cite{HKS}. %On the other hand, Frankl~\cite{Frankl76} extended the argument to the multiply-intersecting case. His argument extends to show that if $\cF$ is a family in which every $s \geq 2$ sets have a non-empty intersection then $\mu_p(\cF) \leq p$ for all $p \leq 1-1/s$.

% \begin{lemma} \label{lem:ekr-mup-multiple}
%  Let $n \geq 1$, $s \geq 2$ and $p \in (0,1-1/s]$. If $\cF$ is a family on $n$ points in which any $s$ sets have a common intersection then $\mu_p(\cF) \leq p$. \comment{Uniqueness?}
% \end{lemma}
% \begin{proof}
%  \comment{TBD.}
% \end{proof}

\subsection{Cross-intersecting settings} \label{sec:katona-cross}

As explained in the introduction to this section, the heart of Katona's proof is a result on sets in which any two points are close. The proof of Lemma~\ref{lem:ekr-mup} closely follows the original argument. In this section we give alternative arguments for the cross-intersecting counterparts in both discrete and continuous settings.

\subsubsection{Discrete setting} \label{sec:katona-cross-discrete}

We start with the easier, discrete setting. First, a few definitions.

\begin{definition} \label{def:s-agreeing}
 A set $A \subseteq \ZZ_m$ is \emph{$s$-agreeing} if every $a,b \in A$ satisfy $a-b \in \{-(s-1),\ldots,s-1\}$. Two sets $A,B \subseteq \ZZ_m$ are \emph{cross-$s$-agreeing} if every $a \in A$ and $b \in B$ satisfy $a-b \in \{-(s-1),\ldots,s-1\}$.
\end{definition}

\begin{definition} \label{def:intervals}
 A set $A \subseteq \ZZ_m$ is an \emph{interval} if it is of the form $\{x-\ell,\ldots,x+\ell\}$ (for $2\ell+1 < m$) or $\{x-\ell,\ldots,x+\ell-1\}$ (for $2\ell < m$). In the first case, $x$ is the \emph{center} of $A$, and in the second, $x-1/2$ is the center of $A$.
\end{definition}

The following simple lemma will simplify the argument below.

\begin{lemma} \label{lem:interval-intersection}
 Let $s \leq m/2$. If $A \subseteq \ZZ_m$ is cross-$s$-agreeing with the non-empty interval $\{x,\ldots,y\}$ of length at most $2s-1$ then $A \subseteq \{y-(s-1),\ldots,x+(s-1)\}$.
\end{lemma}
\begin{proof}
 The proof is by induction on the length of the interval. If $y = x$ then trivially $A \subseteq \{x-(s-1),\ldots,x+(s-1)\}$. Suppose now that we have already shown that if $A$ is cross-$s$-agreeing with $\{x,\ldots,y\}$ then $A \subseteq \{y-(s-1),\ldots,x+(s-1)\}$. If $A$ is cross-$s$-agreeing with $\{x,\ldots,y+1\}$ then
\[
 A \subseteq \{y-(s-1),\ldots,x+(s-1)\} \cap \{y+1-(s-1),\ldots,x+1+(s-1)\} = \{y+1-(s-1),\ldots,x+(s-1)\}. \qedhere
\]
\end{proof}

We can now state and prove the result in the discrete setting.

\begin{lemma} \label{lem:katona-cross-discrete}
 Let $m,s \geq 1$ be integers satisfying $s < m/2$. If $A,B \subseteq \ZZ_m$ are non-empty $s$-agreeing sets then $|A| + |B| \leq 2s$. Moreover, if $|A| + |B| = 2s$ then $A,B$ are intervals centered at the same point or half-point. In particular, if $A$ is $s$-agreeing then $|A| \leq s$, with equality only when $A$ is an interval.
 
 When $s = m/2$ it still holds that $|A| + |B| \leq 2s$, with equality when $B = \{ x : x + s \notin A \}$, and this holds even without the assumption that $A,B$ be non-empty. In particular, if $A$ is $s$-agreeing then $|A| \leq s$, with equality only when $A$ contains one point out of each pair $x,x+s$.
\end{lemma}
\begin{proof}
 We start by proving that $|A| + |B| \leq 2s$ when $s < m/2$. The idea is to use a shifting argument to transform $A,B$ into intervals without decreasing $|A| + |B|$. We construct a sequence of non-empty $s$-agreeing sets, starting with $(A_0,B_0) = (A,B)$. Given $(A_i,B_i)$, note first that $A_i \neq \ZZ_m$, since otherwise $B_i$ would have to be empty. Therefore there exists a point $x \in A_i$ such that $x+1 \notin A_i$. We take $A_{i+1} = A_i \cup \{x+1\}$. If $A_{i+1},B_i$ are not cross-$s$-agreeing then there must be a point $y \in B_i$ such that $x-y \in \{-(s-1),\ldots,s-1\}$ but $x+1-y \notin \{-(s-1),\ldots,s-1\}$. This can only happen if $x-y=s-1$, and so there is at most one such point $y = x-(s-1)$. We therefore take $B_{i+1} = B_i \setminus \{x-(s-1)\}$. If $B_{i+1} = \emptyset$, then we stop the sequence at $(A_i,B_i)$, and otherwise we continue. Note that $|A_{i+1}| = |A_i| + 1$ and $|B_{i+1}| \geq |B_i| - 1$, and so $|A_{i+1}| + |B_{i+1}| \geq |A_i| + |B_i|$.
 
 Since the size of $A_i$ keeps increasing, there must be a last pair in the sequence, say $(A_t,B_t)$. Our stopping condition guarantees that $|B_t| = 1$, say $B_t = \{y\}$. This forces $A_t \subseteq \{y-(s-1),\ldots,y+(s-1)\}$, and so $|A_t| + |B_t| \leq (2s-1) + 1 = 2s$. Since $|A_0| + |B_0| \leq |A_t| + |B_t| \leq 2s$, this completes the proof that $|A| + |B| \leq 2s$.
 
 \smallskip
 
 Next, we show that if $|A| + |B| = 2s$ then $A,B$ are intervals centered at the same point. We do this by reverse induction on the sequence $(A_0,B_0),\ldots,(A_t,B_t)$. For the base case, the argument in the preceding paragraph shows that $|A_t| + |B_t| = 2s$ if $A_t = \{y-(s-1),\ldots,y+(s-1)\}$ and $B_t = \{y\}$, and so both sets are intervals centered at the point $y$.
 
 Suppose now that $(A_{i+1},B_{i+1})$ are intervals centered (without loss of generality) at $0$ or $-1/2$, depending on the parity of $|A_{i+1}|$. %We consider two cases, according to the parity of $|A_{i+1}|$.
 Suppose first that $|A_{i+1}|$ is odd. Then for some $0 \leq \ell \leq s-1$ we have $A_{i+1} = \{-\ell,\ldots,\ell\}$ and $B_{i+1} = \{-(s-1-\ell),\ldots,s-1-\ell\}$. By construction, $-\ell \in A_i$. Suppose that also $\ell \in A_i$, so that $A_i = \{-\ell,\ldots,x\} \cup \{x+2,\ldots,\ell\}$. Lemma~\ref{lem:interval-intersection} shows that
\[
 B_i \subseteq \{x-(s-1),\ldots,s-1-\ell\} \cap \{-(s-1-\ell),\ldots, x+2+(s-1)\} \subseteq \{-(s-1-\ell),\ldots,s-1-\ell\},
\]
 and so $B_i \subseteq B_{i+1}$, implying that $|A_i| + |B_i| < 2s$. We conclude that $\ell \notin A_i$, and so $A_i = \{-\ell,\ldots,\ell-1\}$, corresponding to $x=\ell-1$. By construction, $B_i = B_{i+1} \cup \{x-(s-1)\} = \{-(s-\ell),\ldots,s-1-\ell\}$. Both sets are centered at $-1/2$.
 
 The second case, when $|A_{i+1}|$ is even, is similar. In this case for some $1 \leq \ell \leq s-1$ we have $A_{i+1} = \{-\ell,\ldots,\ell-1\}$ and $B_{i+1} = \{-(s-\ell),\ldots,s-\ell-1\}$. By construction, $-\ell \in A_i$. Suppose that also $\ell-1 \in A_i$, so that $A_i = \{-\ell,\ldots,x\} \cup \{x+2,\ldots,\ell-1\}$. Lemma~\ref{lem:interval-intersection} shows that
\[
 B_i \subseteq \{x-(s-1),\ldots,s-1-\ell\} \cap \{-(s-\ell),\ldots, x+2+(s-1)\} \subseteq \{-(s-\ell),\ldots,s-1-\ell\},
\]
 and so $B_i \subseteq B_{i+1}$. As before, this leads to a contradiction, and we conclude that $A_i = \{-\ell,\ldots,\ell-2\}$ is centered at $-1$. Moreover, $B_i = A_i \cup \{\ell-2-(s-1)\} = \{-(s-\ell)-1,\ldots,s-1-\ell\}$ is also centered at $-1$. This completes the proof.
 
 \smallskip
 
 It remains to consider the case $s = m/2$. The $s$-agreeing condition states that if $a \in A$ and $b \in B$ then $a - b \neq s$. Thus $B \subseteq \{x : x + s \notin A\}$, which implies $|A| + |B| \leq m = 2s$. This bound is tight only when $B = \{x : x + s \notin A\}$.
\end{proof}

As an easy corollary, we can derive the classical Erd\H{o}s--Ko--Rado theorem.

\begin{corollary} \label{cor:ekr}
 Let $n \geq k \geq 1$ be parameters such that $k \leq n/2$. If $\cF \subseteq \binom{[n]}{k}$ is intersecting then $|\cF| \leq \binom{n-1}{k-1}$. Furthermore, if $k < n/2$ and $|\cF| = \binom{n-1}{k-1}$ then $\cF$ consists of all sets containing some $i \in [n]$.
\end{corollary}
\begin{proof}
 The proof is very similar to the proof of Lemma~\ref{lem:ekr-mup}. Let $\pi$ be a random permutation of $[n]$, and choose $t \in [n]$ at random. Let $S_t = \{ \pi(t+1),\ldots,\pi(t+k) \}$, where indices are taken modulo $n$. Since $S_t$ is a random set from $\binom{[n]}{k}$, we see that $|\cF|/\binom{n}{k}$ is the probability that $S_t \in \cF$.
 
 For any setting of $\pi$, let $I = \{t \in [n] : S_t \in \cF\}$. Since $\cF$ is intersecting, if $a,b \in I$ then $\{\pi(a+1),\ldots,\pi(a+k)\},\{\pi(b+1),\ldots,\pi(b+k)\}$ must intersect, and this implies that $I$ is $k$-agreeing (in the sense of Definition~\ref{def:s-agreeing}). Lemma~\ref{lem:katona-cross-discrete} shows that $|I| \leq k$, and so $|\cF|/\binom{n}{k} \leq k/n$, or $|\cF| \leq \binom{n-1}{k-1}$.
 
 Suppose now that $k < n/2$ and $|\cF| = \binom{n-1}{k-1}$. Lemma~\ref{lem:katona-cross-discrete} shows that for every permutation $\pi$, the set $I$ must be an interval of length $k$. In particular, this is the case for the identity permutation. Suppose without loss of generality that in this case, $I = \{1,\ldots,k\}$. Thus $\{2,\ldots,k+1\},\ldots,\{k+1,\ldots,2k\} \in \cF$ but $\{1,\ldots,k\} \notin \cF$ (since $0 \notin I$). Let $S \in \binom{[n]}{k}$ be any set containing $k+1$ but not $1$, and write $S = \{k+1\} \cup A \cup B$, where $A \subseteq \{2,\ldots,k\}$. Let $\pi$ be any permutation which starts $1,\{2,\ldots,k\}\setminus A,A,k+1,B$. Since $\{\pi(1),\ldots,\pi(k)\} = \{1,\ldots,k\} \notin \cF$ whereas $\{\pi(2),\ldots,\pi(k+1)\} = \{2,\ldots,k+1\} \in \cF$, the set $I$ contains $1$ but not $0$, and so must be $\{1,\ldots,k\}$. This implies that $S \in \cF$.
 
 To finish the proof, let $T \in \binom{[n]}{k}$ be any set not containing $k+1$. Let $U$ be $k-1$ elements disjoint from $T$ and not containing~$1$ or~$k+1$; such elements exist since $n-(k+2) \geq (2k+1)-(k+2) = k-1$. Let $\pi$ be any permutation starting $T,k+1,U$, where we put~$1$ first if $1 \in T$. The earlier paragraph shows that $I \supseteq \{1,\ldots,k\}$, and in particular $T \notin \cF$. Thus all sets in $\cF$ contain $k+1$, and since $|\cF| = \binom{n-1}{k-1}$, it must contain all such sets. This completes the proof.
\end{proof}

\subsubsection{Continuous setting} \label{sec:katona-cross-continuous}

We proceed with a continuous analog of Lemma~\ref{lem:katona-cross-discrete}. First, the pertinent definitions.

\begin{definition} \label{def:p-agreeing}
 The \emph{unit circle} $\unitcirc$ consists of the interval $[0,1]$ with its two end-points pasted. The distance between two points $x,y \in \unitcirc$ is their distance on the unit circle. We denote the Lebesgue measure on $\unitcirc$ by $\mu$.

 Let $p \leq 1/2$. Two measurable sets $A,B \subseteq \unitcirc$ are \emph{$p$-agreeing} if any $a \in A$ and $b \in B$ are at distance less than $p$.
\end{definition}

We now state and prove the analog of Lemma~\ref{lem:katona-cross-discrete}. We only state and prove the upper bound part, leaving the identification of optimal sets to the reader.

\begin{lemma} \label{lem:katona-cross-continuous}
 Let $p < 1/2$. If two non-empty measurable sets $A,B \subseteq \unitcirc$ are $p$-agreeing then $\mu(A) + \mu(B) \leq 2p$. %Moreover, if $\mu(A) + \mu(B) = 2p$ then $A,B$ are contained inside intervals $A',B'$ of total measure $2p$ centered at the same point. In particular, if $A$ is $p$-intersecting that $\mu(A) \leq p$, with equality only when $A$ is contained in an interval of measure~$p$. %When $p = 1/2$ we have $\mu(A) + \mu(B) \leq 2p$ even without the assumption that $A,B$ are non-empty. Moreover, $\mu(A) + \mu(B) = 1$ if $B$ equals $\{ x : x + 1/2 \notin A \}$ up to measure zero. In particular, if $A$ is $p$-intersecting then $\mu(A) \leq 1/2$.
\end{lemma}
\begin{proof}
 The proof uses an approximation argument. Let $\epsilon > 0$ be a parameter satisfying $p + \epsilon < 1/2$. Since $A$ is measurable, there is a sequence $A_i$ of intervals of total length smaller than $\mu(A) + \epsilon$ such that $A \subseteq \bigcup_i A_i$. Note that every point in $\bigcup_i A_i$ is at distance at most $\epsilon/2$ from a point in $A$, since otherwise $\mu(\bigcup_i A_i \setminus A) \geq \epsilon$. Similarly, there is a sequence $B_j$ of intervals of total length at most $\mu(B) + \epsilon$ such that $B \subseteq \bigcup_j B_j$, and any point in $\bigcup_j B_j$ is at distance at most $\epsilon/2$ from a point in $B$. Thus $\bigcup_i A_i$ and $\bigcup_j B_j$ are cross-$(p+\epsilon)$-agreeing.
 
 Choose $I$ so that $\sum_{i>I} \mu(A_i) < \epsilon$, and set $A^* = \bigcup_{i=1}^I A_i$. Similarly, choose $J$ so that $\sum_{j>J} \mu(B_j) < \epsilon$, and set $B^* = \bigcup_{j=1}^J B_j$. Thus $A^*,B^*$ are cross-$(p+\epsilon)$-agreeing, and each is a union of finitely many intervals.
 
 Let $M$ be a large integer, and define $A^*_M = \bigcup_{x=0}^{M-1} \{[x/M,(x+1)/M] : [x/M,(x+1)/M] \subseteq A^*\}$. Note that $\mu(A^*_M) \geq \mu(A^*) - 2I/M$. Define $B^*_M$ similarly. We can view $A^*_M,B^*_M$ as subsets of $\ZZ_M$. These subsets are cross-$\lfloor (p+\epsilon)M \rfloor$-agreeing: if $[a/M,(a+1)/M] \in A^*_M$ and $[b/M,(b+1)/M] \in B^*_M$ then $a/M,(b+1)/M$ are at distance less than $p+\epsilon$. Lemma~\ref{lem:katona-cross-discrete} thus shows that $\mu(A^*_M) + \mu(B^*_M) \leq 2(p+\epsilon)$. On the other hand, $\mu(A^*_M) + \mu(B^*_M) \geq \mu(A^*) + \mu(B^*) - 2(I+J)/M$. Taking the limit $M\to\infty$, we deduce that $\mu(A^*) + \mu(B^*) \leq 2(p+\epsilon)$. Since $\mu(A^*) + \mu(B^*) \geq \mu(A) + \mu(B) - 2\epsilon$, we conclude that $\mu(A) + \mu(B) \leq 2p + 4\epsilon$. Taking the limit $\epsilon\to0$, we deduce the lemma.
\end{proof}

As a corollary, we obtain the following useful result.

\begin{corollary} \label{cor:cross-intersecting}
 Let $\cF,\cG$ be cross-intersecting families on $n$ points. For any $p \leq 1/2$, $\mu_p(\cF) + \mu_p(\cG) \leq 1$.
\end{corollary}
\begin{proof}
 We first settle the case $p = 1/2$. If $\cF,\cG$ are cross-intersecting then $\cG$ is disjoint from $\{\overline{A} : A \in \cF\}$. Since the latter set has measure $\mu_p(\cF)$, we conclude that $\mu_p(\cF) + \mu_p(\cG) \leq 1$.
 
 Suppose now that $p < 1/2$. We will follow the argument of Lemma~\ref{lem:ekr-mup}. We choose $n$ points $x_1,\ldots,x_n$ at random on $\unitcirc$, and two starting points $t_\cF,t_\cG \in \unitcirc$ at random. Let $S_t$ be the set of points that lie inside the interval $(t,t+p)$, let $I_\cF$ be the set of indices $t_\cF$ such that $S_{t_\cF} \in \cF$, and define $I_\cG$ analogously. Thus $\mu_p(\cF) + \mu_p(\cG) = \EE[\mu(I_\cF) + \mu(I_\cG)]$. Since $\cF,\cG$ are cross-intersecting, we see that $I_\cF,I_\cG$ are cross-$p$-agreeing. The lemma shows that $\mu(I_\cF) + \mu(I_\cG) \leq 2p \leq 1$ if neither $I_\cF$ nor $I_\cG$ are empty, and $\mu(I_\cF) + \mu(I_\cG) \leq 1$ trivially holds if one of the sets is empty. The corollary follows.
\end{proof}

We comment that the corollary remains holding if the families in question are on infinitely many points, in the sense of Section~\ref{sec:infinite}, with the same proof; the only difference is that we choose infinitely many points rather than just $n$.

\section{Infinite Ahlswede--Khachatrian theorem} \label{sec:infinite}

The goal of Section~\ref{sec:weighted} was to calculate the quantities $w(n,t,p)$ and $\wsup(t,p)$ which are, respectively, the maximum $\mu_p$-measure of a $t$-intersecting family on $n$ points, and the supremum $\mu_p$-measure of a $t$-intersecting family on any number of points. In this section we will consider what happens when we allow families on infinitely many points.

\begin{definition} \label{def:infinite-defs}
 The infinite product measure $\mu_p$ on $\cF \subseteq 2^\NN$ is the infinite product measure extending the finite $\mu_p$ measures.

 %A \emph{family on infinitely many points} is a subset $\cF \subseteq 2^\NN$ which is measurable with respect to $\mu_p$ for any $p \in (0,1)$, and on which $\mu_p$ is a continuous function of $p$. The family is \emph{$t$-intersecting} if any two sets $A,B \in \cF$ have at least $t$ points in common; here $t$ can also be $\aleph_0$.
 
 Given $p \in (0,1)$, A \emph{family on infinitely many points} is a subset $\cF \subseteq 2^\NN$ which is measurable with respect to $\mu_p$.
 The family is \emph{$t$-intersecting} if any two sets $A,B \in \cF$ have at least $t$ points in common; here $t$ can also be $\aleph_0$.
 
 A family $\cF$ is \emph{finitely determined} if it is the \emph{extension} of some family $\cG$ on $n$ points: $\cF = \{ A : A \cap [n] \in \cG \}$. Given $p \in (0,1)$, the family $\cF$ is \emph{essentially finitely determined} if it differs from a finitely determined family by a $\mu_p$-null set.
\end{definition}

%The reader might object to the condition that $\cF$ be measurable with respect to \emph{all} measures $\mu_p$, and in fact most of our proofs would only depend on measurablity with respect to a particular $\mu_p$. However, in some cases it would be useful to have $\mu_p(\cF)$ be a continuous function of $p$. \comment{Is this really needed? Seems not!} Our goal in this section is to determine the following analog of $w(n,t,p)$.

\begin{definition} \label{def:winfty}
 For $t \geq 1$ and $p \in (0,1)$, the parameter $\winfty(t,p)$ is the supremum of $\mu_p(\cF)$ over all $t$-intersecting families on infinitely many points.
\end{definition}

The following theorem summarizes the results proved in this section.

\begin{theorem} \label{thm:infinity}
 Let $t \geq 1$ and $p \in (0,1)$.
 
 If $p < 1/2$ then $\winfty(t,p) = \wsup(t,p)$. Furthermore, if $\cF$ is a $t$-intersecting family on infinitely many points satisfying $\mu_p(\cF) = \winfty(t,p)$ then $\cF$ is essentially finitely determined.
 
 If $p = 1/2$ then $\winfty(t,p) = 1/2$. Furthermore, if $t>1$ then no $t$-intersecting family on infinitely many points has measure $1/2$.
 
 If $p > 1/2$ then $\winfty(t,p) = 1$. In fact, there is an $\aleph_0$-intersecting family $\cF$ on infinitely many points which satisfying $\mu_p(\cF) = 1$ for all $p > 1/2$.
\end{theorem}

Note that there are many intersecting families with $\mu_{1/2}$-measure $1/2$, for example $\cF_{1,r}$ for every $r$. There are also families which are not essentially finitely determined. One example is the family of sets $S$ such that for some $n \geq 0$,
\begin{enumerate}[(a)]
\item $S$ contains exactly one of $\{2m+1,2m+2\}$ for all $m < n$.
\item $S$ contains both $\{2n+1,2n+2\}$.	
\end{enumerate}

We prove the different cases in the theorem one by one, starting with the case $p < 1/2$. We start with the following technical proposition, which follows from the definition of $\mu_p$.

\begin{proposition} \label{pro:infinity-mup}
 Let $\cF$ be a family on infinitely many points. For every $p \in (0,1)$ and for every $\epsilon > 0$ there is a finitely determined family $\cG$ such that $\mu_p(\cF \symdiff \cG) < \epsilon$.
\end{proposition}

Our proof will need a stability version of Theorem~\ref{thm:weighted-main}, due to Keller and Lifshitz~\cite{KellerLifshitz}.

\begin{proposition} \label{pro:ak-stability}
 Fix $p \in (0,1/2)$ and $t \geq 1$. If $\cF$ is a $t$-intersecting family on $n$ points of measure $(1-\epsilon)w(n,t,p)$ then there exists a $t$-intersecting family $\cG$ on $n$ points of measure $w(n,t,p)$ such that $\mu_p(\cF \symdiff \cG) = O(\epsilon^{\log_{1-p} p})$. Here the hidden constant depends on $p,t$ but not on $n$.
\end{proposition}

This will allow us to settle the case $p < 1/2$.

\begin{lemma} \label{lem:infinity-psmall}
 Let $t \geq 1$ and $p \in (0,1/2)$. Then $\winfty(t,p) = \wsup(t,p)$. Furthermore, if $\cF$ is a $t$-intersecting family on infinitely many points satisfying $\mu_p(\cF) = \winfty(t,p)$ then $\cF$ is essentially finitely determined.
\end{lemma}
\begin{proof}
 Clearly $\winfty(t,p) \geq \wsup(t,p)$. For the other direction, let $\cF$ be a $t$-intersecting family on infinitely many points. Let $\epsilon \in (0,1/4)$ be a parameter. Proposition~\ref{pro:infinity-mup} shows that there is a family $\cG$ depending on $N$ points such that $\mu_p(\cF \symdiff \cG) < \epsilon$. We can assume that $\cG$ depends on the first $N$ points. For $S \subseteq [N]$, let $\cF_S = \{ A \in 2^{\NN \setminus [N]} : S \cup A \in \cF \}$.
 
 Let $\dist(x,\{0,1\}) = \min(|x|,|x-1|)$, and notice that
\[
 \EE_{S \sim \mu_p([N])}[\dist(\mu_p(\cF_S),\{0,1\})] \leq \EE_{S \sim \mu_p([N])}[\mu_p(\cF_S \symdiff \cG_S)] = \mu_p(\cF \symdiff \cG) < \epsilon.
\]
 Thus with probability at least $1-\sqrt{\epsilon}$, $\dist(\mu_p(\cF_S),\{0,1\}) < \sqrt{\epsilon}$.
 
 Define now $\cF' = \{ A \in \cF : \mu_p(\cF_{A \cap [N]}) > 1-\sqrt{\epsilon} \}$. In words, $\cF'$ is the subset of $\cF$ obtained by removing all fibers $\cF_S$ whose $\mu_p$-measure is at most $1-\sqrt{\epsilon}$. In particular, we remove all fibers whose $\mu_p$-measure is less than $\sqrt{\epsilon}$, and all fibers whose measure is between $\sqrt{\epsilon}$ and $1-\sqrt{\epsilon}$. The former have measure at most $\sqrt{\epsilon}$, and the latter also have a measure at most $\sqrt{\epsilon}$ due to the preceding paragraph. Thus $\mu_p(\cF \setminus \cF') \leq 2\sqrt{\epsilon}$.
 
 If $S,T \subseteq [N]$ are such that $\mu_p(\cF_S),\mu_p(\cF_T) > 1-\sqrt{\epsilon} > 1/2$ then $\cF_S,\cF_T$ cannot be cross-intersecting, due to Corollary~\ref{cor:cross-intersecting} (while we stated the corollary for finite families, it holds for infinite families with exactly the same proof). It follows that $|S \cap T| \geq t$, and so the family $\cG = \{ A : \mu_p(\cF_{A \cap [N]}) > 1-\sqrt{\epsilon} \}$ containing $\cF'$ is $t$-intersecting. Since $\cF'$ is contained in a $t$-intersecting family depending on $N$ points, $\mu_p(\cF') \leq \wsup(t,p)$, and so $\mu_p(\cF) \leq \wsup(t,p) + 2\sqrt{\epsilon}$. Taking the limit $\epsilon \to 0$, we deduce that $\mu_p(\cF) \leq \wsup(t,p)$, and so $\winfty(t,p) = \wsup(t,p)$.
 
 \smallskip
 
 Suppose now that $\mu_p(\cF) = \wsup(t,p)$. For every $\epsilon > 0$, we have constructed above a $t$-intersecting family $\cF' = \cF'_\epsilon$ which is contained in a $t$-intersecting family $\cG_\epsilon$ depending on $N_\epsilon$ points and satisfies $\mu_p(\cF'_\epsilon) \geq \mu_p(\cF) - 2\sqrt{\epsilon}$. In particular, $\mu_p(\cG_\epsilon) \geq \mu_p(\cF) - 2\sqrt{\epsilon}$. For small enough $\epsilon > 0$, this can only happen if $w(N_\epsilon,t,p) = \wsup(t,p)$. Furthermore, Proposition~\ref{pro:ak-stability} shows that $\mu_p(\cG_\epsilon \symdiff \cH_\epsilon) = O(\epsilon^{O(1)})$ for some $t$-intersecting family $\cH_\epsilon$ on $N$ points of $\mu_p$-measure $\wsup(t,p)$. Corollary~\ref{cor:wtp} shows that such a family is a Frankl family with appropriate parameters. In particular, $\cH_\epsilon$ depends on a constant number of coordinates (depending on $t,p$).
 
 Notice that $\mu_p(\cG_\epsilon \setminus \cF'_\epsilon) \leq \sqrt{\epsilon}$ (since each fiber we retained had $\mu_p$-measure at least $1-\sqrt{\epsilon}$), and so $\mu_p(\cG_\epsilon \symdiff \cF) = O(\sqrt{\epsilon})$. Thus $\mu_p(\cF \symdiff \cH_\epsilon) = O(\epsilon^{O(1)})$. In particular, if $\epsilon_1,\epsilon_2 \leq \epsilon$ we get $\mu_p(\cH_{\epsilon_1} \symdiff \cH_{\epsilon_2}) = O(\epsilon^{O(1)})$. Since $\cH_{\epsilon_1} \symdiff \cH_{\epsilon_2}$ depends on a constant number of coordinates, for small enough $\epsilon$ this forces $\cH_{\epsilon_1} = \cH_{\epsilon_2}$. There is therefore a Frankl family $\cH$ satisfying $\mu_p(\cF \symdiff \cH) = O(\epsilon^{O(1)})$ for all small enough $\epsilon > 0$. Taking the limit $\epsilon \to 0$ concludes the proof.
\end{proof}

We now consider the case $p = 1/2$. We thank Shay Moran for help with the proof of the following lemma.

\begin{lemma} \label{lem:infinity-phalf}
 For all $t \geq 1$, $\winfty(t,1/2) = 1/2$. Furthermore, if $t \geq 2$ then no $t$-intersecting family on infinitely many points has $\mu_p$-measure $1/2$.
\end{lemma}
\begin{proof}
 If $\cF$ is an intersecting family on infinitely many points then $\cF$ is disjoint from $\{ \overline{A} : A \in \cF \}$. Since both families have the same measure, it follows that $\mu_p(\cF) \leq 1/2$. Thus $\winfty(t,1/2) \leq 1/2$. On the other hand, clearly $\winfty(t,1/2) \geq \wsup(t,1/2)$, and so $\winfty(t,1/2) = 1/2$ due to Corollary~\ref{cor:wtp}.
 
 Suppose now that $t \geq 2$ and that $\cF$ is a $t$-intersecting family on infinitely many points of $\mu_{1/2}$-measure $1/2$. By possibly taking the up-set of $\cF$, we can assume that $\cF$ is monotone (see Definition~\ref{def:upset} for the relevant definitions). Let $\cF_- = \{ A \in 2^{\NN \setminus \{1\}} : A \in \cF \}$ and $\cF_+ = \{ A \in 2^{\NN \setminus \{1\}} : A \cup \{1\} \in \cF \}$. Since $\cF$ is monotone, $\cF_- \subseteq \cF_+$. Since $t \geq 2$, $\cF_+$ is intersecting, and so $\mu_{1/2}(\cF_+) \leq 1/2$. It follows that $\mu_{1/2}(\cF_+) = \mu_{1/2}(\cF_-) = 1/2$ as well, and so $\cF_1 = \{ A \in 2^\NN : A \cap (\NN \setminus \{1\}) \in \cF_- \}$ also has $\mu_{1/2}$-measure~$1/2$. Note that $\cF_1$ does not depend on $1$. Note also that $\cF_1 \subseteq \cF$ and that $\cF_1$ is $t$-intersecting (since $\cF_-$ is).
 
 Starting with $\cF_1$ but working with the element $2$ instead of $1$, construct a $t$-intersecting family $\cF_2 \subseteq \cF_1$ of $\mu_{1/2}$-measure~$1/2$, and note that $\cF_2$ depends on neither~$1$ nor~$2$. Continuing in this way, we construct a sequence of families $\cF_n$ such that $\cF_n \subseteq \cF_{n-1}$ is a $t$-intersecting family of $\mu_{1/2}$-measure~$1/2$ that does not depend on $1,\ldots,n$. Thus $\cF' = \bigcap_{n \in \NN} \cF_n$ is a $t$-intersecting family of $\mu_{1/2}$-measure~$1/2$ which represents a tail event. However, such a family must have measure~$0$ or~$1$ due to Kolmogorov's zero-one law. This contradiction shows that the original family $\cF$ could not have existed.
\end{proof}

Finally, we dispense of the case $p > 1/2$.

\begin{lemma} \label{lem:infinity-plarge}
 There exists a family $\cF$ on infinitely many points which is $\aleph_0$-intersecting and satisfies $\mu_p(\cF) = 1$ for all $p \in (1/2,1)$.
\end{lemma}
\begin{proof}
 Let $f(r)$ be any integer function which is $\omega(1)$ and $o(r)$, for example $\lfloor \sqrt{r} \rfloor$. We define $\cF$ to consist of all sets $A$ such that $|A \cap [2r+f(r)]| \geq r+f(r)$ for all large enough $r$. For any $A,B \in \cF$ it holds that for large enough $r$, $|A \cap B \cap [2r+f(r)]| \geq 2(r+f(r)) - (2r+f(r)) = f(r)$, and so the intersection $A \cap B$ is infinite.
 
 Hoeffding's bound states (in one version) that $\Pr[\bin(n,p) \leq qn] \leq e^{-2n(p-q)^2}$. This implies that for $A \sim \mu_p$,
\[
 \Pr[|A \cap [2r+f(r)]| < r + f(r)] \leq \exp \left[ -2(2r+f(r)) \left(p-\frac{r+f(r)}{2r+f(r)}\right)^2 \right].
\]
 Since $(r+f(r))/(2r+f(r)) \to 1/2$, there exists $\epsilon > 0$ such that for large enough $r$ we have $\Pr[|A \cap [2r+f(r)]| < r + f(r)] \leq e^{-4r\epsilon^2}$. This shows that $\sum_r \Pr[|A \cap [2r+f(r)]| < r + f(r)]$ converges. The Borell--Cantelli lemma thus shows that almost surely, only finitely many of these ``bad events'' happen, and so $\mu_p(\cF) = 1$.
\end{proof}

\section{Ahlswede--Khachatrian theorem for the Hamming scheme} \label{sec:hamming}

Chung, Frankl, Graham, and Shearer~\cite{CFGS} considered the difference between \emph{intersecting} and \emph{agreeing} families. Let us say that a family $\cF$ on $n$ points is \emph{$\cG$-intersecting} if any $A,B \in \cF$ satisfy $A \cap B \in \cG$, where $\cG$ is some monotone family on $n$ points. Thus $t$-intersecting families are $\cG$-intersecting for $\cG = \{ S \subseteq [n] : |S| \geq t \}$. A family $\cF$ is \emph{$\cG$-agreeing} if any $A,B \in \cF$ satisfy $\overline{A \symdiff B} \in \cG$. Since $\overline{A \symdiff B} = (A \cap B) \cup (\overline{A} \cap \overline{B})$, every $\cG$-intersecting family is a~fortiori $\cG$-agreeing, but the converse does not hold in general. Nevertheless, Chung et al.\ showed that the maximum $\mu_{1/2}$-measure of a $\cG$-agreeing family is the same as the maximum $\mu_{1/2}$-measure of a $\cG$-intersecting family, using a simple shifting argument. 
This explains why many arguments, such as ones using Shearer's lemma, seem to work not only for $\cG$-intersecting families but also for $\cG$-agreeing families. The latter are easier to work with since the definition is more symmetric, and this is taken to full advantage in~\cite{EFF0}, for example.

\smallskip

Ahlswede and Khachatrian~\cite{AK5}, motivated by the geometry of Hamming spaces, considered a more general question: what is the largest subset of $\ZZ_m^n$ of diameter $n-t$ with respect to Hamming distance? Two vectors $x,y \in \ZZ_m^n$ have Hamming distance $n-t$ if they agree on exactly $t$ coordinates. Thus a subset of diameter $n-t$ is the same as a collection of vectors, every two of which agree on at least $t$ coordinates. They showed that this corresponds (roughly) to $\mu_p$ for $p = 1/m$ (though they stated this in a different language).

Motivated by the recent success of Fourier-analytic techniques to analyze questions in extremal combinatorics, Alon, Dinur, Friedgut and Sudakov~\cite{ADFS} studied independent sets in graph products, and in particular rederived the results of~\cite{AK5} for the case $t = 1$. Shinkar~\cite{Shinkar} extended this to general $t$ and $m \geq t+1$, using methods of Friedgut~\cite{Friedgut}. Unfortunately, these spectral techniques cannot at the moment yield all results of~\cite{AK5}.

\smallskip

Any bound on $t$-agreeing families of vectors in $\ZZ_m^n$ readily yields matching results on the $\mu_{1/m}$-measure of $t$-intersecting families on $n$ points. Indeed, given a family $\cF \subseteq \{0,1\}^n$, let $\cG \subseteq \ZZ_m^n$ consist of all vectors obtained by taking each binary vector in $\cF$ and replacing each $0$ coordinate with one of the $m-1$ values $\{2,\ldots,m\}$. The resulting family is $t$-agreeing, and $|\cG| = m^n \mu_{1/m}(\cF)$. Therefore a bound on $|\cG|$ yields a bound on $\mu_{1/m}(\cF)$. In other words, $t$-agreeing families give a discrete model for $\mu_{1/m}$.

We extend these results in Section~\ref{sec:hamming-discrete}, by giving a discrete model for $\mu_{s/m}$ for any integers $s,m$ satisfying $s/m < 1/2$. Extending the work of~\cite{AK5}, we prove an Ahlswede--Khachatrian theorem in this setting. Section~\ref{sec:hamming-continuous} describes a similar continuous model for $\mu_p$ for any $p < 1/2$, and proves an Ahlswede--Khachatrian theorem in that setting.

\subsection{Discrete setting} \label{sec:hamming-discrete}

We start by defining the new discrete model.

\begin{definition} \label{def:ts-agreement}
 A \emph{family on $\ZZ_m^n$} is a subset of $\ZZ_m^n$. For $s \leq m/2$, a family $\cF$ on $\ZZ_m^n$ is \emph{$t$-agreeing up~to~$s$} if every $x,y \in \cF$ have $t$ coordinates $i_1,\ldots,i_t$ such that $x_{i_j} - y_{i_j} \in \{-(s-1),\ldots,s-1\}$ for $1 \leq j \leq t$.
 
 The uniform measure on $\ZZ_m$ is denoted by $\mu_m$, and the uniform measure on $\ZZ_m^n$ by $\mu_m^n$.
 
 The maximum measure of a $t$-agreeing up~to~$s$ family on $\ZZ_m^n$ is denoted $w(\ZZ_m^n,t,s)$.
\end{definition}

Our goal is to show that $w(\ZZ_m^n,t,s) = w(n,t,s/m)$ whenever $s/m \leq 1/2$, and to identify the optimal families when possible. Our candidate optimal families are obtained in the following way.

\begin{definition} \label{def:sigma}
 Fix $n,m \geq 1$ and $s \leq m/2$. For $y \in \ZZ_m^n$, the mapping $\sigma_y\colon \ZZ_m^n \to 2^{[n]}$ is given by $\sigma_y(x) = \{i \in [n] : 1 \leq x_i-y_i \leq s\}$.
 
 If $\cF$ is a family on $\ZZ_m^n$ and $\cG$ is a family on $n$ points then $\cF \approx \cG$ (or, $\cF$ is \emph{equivalent} to $\cG$) if $\cF = \sigma_y^{-1}(\cG)$ for some vector $y \in \ZZ_m^n$. In words, $\cF$ arises from $\cG$ by going through all sets $S \in \cF$, replacing each $i \in S$ by all values in $\{y_1+1,\ldots,y_1+s\}$, and each $i \notin S$ by all other values.
 
 We say that $\cF$ is equivalent to a $(t,r)$-Frankl family if $\cF \approx \cG$ for some family $\cG$ equivalent to a $(t,r)$-Frankl family (see Definition~\ref{def:frankl}).
\end{definition}

We will prove the following theorem.

\begin{theorem} \label{thm:hamming-discrete}
 Let $n,m,t \geq 1$ and $s \leq m/2$. Then $w(\ZZ_m^n,t,s) = w(n,t,s/m)$. Furthermore, if $s < m/2$ and $\cF$ is a family on $\ZZ_m^n$ of measure $w(n,t,s/m)$ then $\cF$ is equivalent to a $t$-intersecting family on $n$ points of $\mu_{s/m}$-measure $w(n,t,s/m)$. The same holds if $m=2$, $s=1$ and $t > 1$.
\end{theorem}

When $s = m/2$ and $s > 1$ there can be exotic families of maximum measure. For example, when $s = 2$ and $m = 4$ the family $\{00,01,12,13,20,21,32,33\}$ is intersecting up~to~$2$ but doesn't arise from any family on two points. Adding another coordinate ranging over $\{0,1\}$, we get a $2$-intersecting up~to~$2$ family which doesn't arise from any family on three points.

When $m = 2$, $s = 1$ and $t = 1$, a maximum measure family is one which contains exactly one vector of each pair of complementary vectors.

\smallskip

On the way toward proving the theorem, we will need to consider hybrid families in which some of the coordinates come from $\ZZ_m$, and others from $\{0,1\}$.

\begin{definition} \label{def:hybrid}
 A \emph{family on $\ZZ_m^n \times \{0,1\}^\ell$} is a subset of $\ZZ_m^n \times \{0,1\}^\ell$. Such a family $\cF$ is \emph{$t$-agreeing up~to~$s$} if every $x,y \in \cF$ have $t$ coordinates $i_1,\ldots,i_t$ such that $x_{i_j} - y_{i_j} \in \{-(s-1),\ldots,s-1\}$ (if $i_j \leq n$) or $x_{i_j} = y_{i_j} = 1$ (if $i_j > n$) for $1 \leq j \leq t$.
 
 More generally, two vectors $x,y \in \ZZ_m^n \times \{0,1\}^\ell$ $s$-agree on a coordinate $i \leq n$ if $x_i - y_i \in \{-(s-1),\ldots,s-1\}$, and they $s$-agree on a coordinate $i > n$ if $x_i = y_i = 1$. Thus a family on $\ZZ_m^n \times \{0,1\}^\ell$ is $t$-agreeing up~to~$s$ if every $x,y \in \cF$ $s$-agree on at least $t$~coordinates.
 
 We measure families on $\ZZ_m^n \times \{0,1\}^\ell$ using the product measures $\mu_{s,m}^{n,\ell} = \mu_m^n \times \mu_{s/m}^\ell$.
 
 For $y \in \ZZ_m^n$, the mapping $\sigma_y \colon \ZZ_m^n \times \{0,1\}^\ell \to 2^{[n+\ell]}$ is the product of $\sigma_y$ and the identity mapping.
 
 If $\cF$ is a family on $\ZZ_m^n \times \{0,1\}^\ell$ and $\cG$ is a family on $n+\ell$ points then $\cF \approx \cG$ if $\cF = \sigma_y^{-1}(\cG)$ for some vector $y \in \ZZ_m^n$. In words, $\cF$ arises from $\cG$ by going through all sets $S \in \cF$, and for each $i \leq n$, replacing each $i \in S$ by all values in $\{y_1+1,\ldots,y_1+s\}$, and each $i \notin S$ by all other values.
\end{definition}

We start by proving $w(\ZZ_m^n,t,s) = w(n,t,s/m)$. The identification of optimal families will require further refining the proof, but we present the two proofs separately for clarity. We start with a technical result about the stable set polytope.

\begin{proposition} \label{pro:half-integrality}
 Let $G = (V,E)$ be a graph, and let $\mu\colon V \to \RR$ be a set of weights. Consider the program
\[
\begin{array}{lll}
 \max & \sum_{x \in V} \mu_x v_x \\
 \mathrm{s.t.} \quad & 0 \leq v_x \leq 1 & \text{for all $x \in V$} \\
 & v_x + v_y \leq 1 & \text{for all $(x,y) \in E$}
\end{array} 
\]
The maximum of the program is attained (not necessarily uniquely) at some half-integral point (a point in which all entries are $\{0,1/2,1\}$).

Moreover, if there is a unique half-integral point at which the objective is maximized, then this point is the unique maximum.
\end{proposition}
\begin{proof}
 The stable set polytope of $G$ is defined as the set of all $V$-indexed vectors which satisfy the constraints stated in the program.
 Every linear functional over the polytope is maximized at some vertex. Moreover, if there is a unique vertex at which the maximum is attained then this vertex is the unique maximum.
 Therefore the proposition follows from the fact that each vertex of the polytope is half-integral.
 This classical fact is due to Nemhauser and Trotter~\cite{NemhauserTrotter}. We briefly sketch the proof below, following Schrijver~\cite[Theorem~64.7]{Schrijver}.

 Let $v$ be a vertex of the stable set polytope. Let $I = \{ x : 0 < v_x < 1/2 \}$ and $J = \{ y : 1/2 < v_y < 1 \}$. For small enough $\epsilon>0$, both $v + \epsilon (1_I - 1_J)$ and $v - \epsilon (1_I - 1_J)$ are in the stable set polytope (where $1_I$ is the characteristic function of $I$), and so $v$ can only be a vertex if $I = J = \emptyset$.
\end{proof}

The heart of the proof of $w(\ZZ_m^n,t,s) = w(n,t,s/m)$ is the following lemma.

\begin{lemma} \label{lem:hamming-discrete-ub-ind}
 Let $\cF$ be a family on $\ZZ_m^n \times \{0,1\}^\ell$ which is $t$-agreeing up~to~$s$, where $s \leq m/2$ and $n \geq 1$ (but possibly $\ell = 0$). There is a family $\cH$ on $\ZZ_m^{n-1} \times \{0,1\}^{\ell+1}$ which is $t$-agreeing up~to~$s$ and satisfies $\mu_{s,m}^{n-1,\ell+1}(\cH) \geq \mu_{s,m}^{n,\ell}(\cF)$.
\end{lemma}
\begin{proof}
 The first step is to decompose $\cF$ according to the values of all coordinates but the $n$'th one:
\[
 \cF = \bigcup_{x_1 \in \ZZ_m^{n-1}} \bigcup_{x_2 \in \{0,1\}^\ell} \{x_1\} \times \cF_{x_1,x_2} \times \{x_2\}.
\]
 It will be convenient to refer to the pair $x_1,x_2$ as a single vector $x = (x_1,x_2) \in \ZZ_m^{n-1} \times \{0,1\}^\ell$.
 
 We now construct a graph $G = (V,E)$ as follows. The vertices are $V = \{x \in \ZZ_m^{n-1} \times \{0,1\}^\ell : \cF_x \neq \emptyset\}$. We connect two vertices $x,y$ (possibly $x=y$) if $\{x,y\}$ is \emph{not} $t$-agreeing up~to~$s$; in other words, if the vectors $x,y$ $s$-agree on exactly $t-1$ coordinates. If $(x,y) \in E$ then the sets $\cF_x,\cF_y \subseteq \ZZ_m$ must be $s$-agreeing, in the terminology of Definition~\ref{def:s-agreeing}.
 
 There are now two cases to consider: the degenerate case $n + \ell = t$, and the non-degenerate case $n + \ell > t$. In the degenerate case the graph $G$ consists of isolated vertices with self-loops, since the only way that vectors $x,y$ of length $t-1$ can $s$-agree on $t-1$ coordinates is if $x = y$. Thus each set $\cF_x$ must be $s$-agreeing, and so $|\cF_x| \leq s$ by Lemma~\ref{lem:katona-cross-discrete}. We form the new family $\cH$ as follows:
\[
 \cH = \bigcup_{x_1 \in \ZZ_m^{n-1}} \bigcup_{x_2 \in \{0,1\}^\ell} \{x_1\} \times \cH_{x_1,x_2} \times \{x_2\}, \quad
 \cH_x = \begin{cases} \{1\} & \text{if } \cF_x \neq \emptyset, \\ \emptyset & \text{otherwise}. \end{cases}
\]
 By construction, $\mu_{s,m}^{1,0}(\cF_x) \leq \mu_{s,m}^{0,1}(\cH_x)$: either both sets are empty, or $\cF_x$ has measure at most $s/m$, while $\cH_x$ has measure $s/m$. Thus $\mu_{s,m}^{n,\ell}(\cF) \leq \mu_{s,m}^{n-1,\ell+1}(\cH)$. Moreover, it is not hard to verify that $\cH$ is $t$-agreeing up~to~$s$. This completes the proof in the degenerate case.
 
 The proof in the non-degenerate case is more complicated. First, notice that we no longer have self-loops, since every $x \in V$ agrees with itself on $n-1+\ell \geq t$ coordinates. If $(x,y) \in E$ then the non-empty sets $\cF_x,\cF_y$ are $s$-agreeing, and so $|\cF_x| + |\cF_y| \leq 2s$ by Lemma~\ref{lem:katona-cross-discrete}. In particular, if $|\cF_x| \geq 2s$ then $x$ must be an isolated vertex. This suggests pruning the graph by removing all isolated vertices. We denote the resulting graph by $G' = (V',E')$; note that $V'$ could be empty. All vertices $x \in V'$ now satisfy $0 < |\cF_x| < 2s$.
 
 Let $v_x = |\cF_x|/(2s)$. Then $0 \leq v_x \leq 1$, $v_x + v_y \leq 1$ for all edges, and
\[
 \mu_{s,m}^{n,\ell}(\cF) = \sum_{x \in V \setminus V'} \mu_{s,m}^{n-1,\ell}(x) \mu_m(\cF_x) + \frac{2s}{m} \sum_{x \in V'} \mu_{s,m}^{n-1,\ell}(x) v_x.
\]
 Proposition~\ref{pro:half-integrality} states that there exists a $\{0,1/2,1\}$-valued vector $w$ which satisfies $0 \leq w_x \leq 1$, $w_x + w_y \leq 1$ for all edges, and
 $\sum_{x \in V'} \mu_{s,m}^{n-1,\ell}(x) v_x \leq \sum_{x \in V'} \mu_{s,m}^{n-1,\ell}(x) w_x$.
 We will construct $\cH$ according to this vector:
\[
 \cH = \bigcup_{x_1 \in \ZZ_m^{n-1}} \bigcup_{x_2 \in \{0,1\}^\ell} \{x_1\} \times \cH_{x_1,x_2} \times \{x_2\}, \quad
 \cH_x =
 \begin{cases}
 \{0,1\} & \text{if } x \in V \setminus V', \\
 \{0,1\} & \text{if } x \in V' \text{ and } w_x = 1, \\
 \{1\} & \text{if } x \in V' \text{ and } w_x = 1/2, \\
 \emptyset & \text{if } x \in V' \text{ and } w_x = 0, \\
 \emptyset & \text{if } x \notin V'.
 \end{cases}
\]
 Note that the \emph{support} of $\cH$, which is the set of $x \in \ZZ_m^{n-1} \times \{0,1\}^\ell$ such that $\cH_x \neq \emptyset$, is a subset of the support of $\cF$.

 We start by showing that $\cH$ is $t$-agreeing up~to~$s$. Let $(x_1,\alpha,x_2),(y_1,\beta,y_2) \in \cH$, where $x_1,y_1 \in \ZZ_m^{n-1}$, $\alpha,\beta \in \{0,1\}$, and $x_2,y_2 \in \{0,1\}^\ell$. If there is no edge between $(x_1,x_2)$ and $(y_1,y_2)$ then clearly $(x_1,\alpha,x_2),(y_1,\beta,y_2)$ $s$-agree on at least $t$~coordinates. If there does exist an edge then $x = (x_1,x_2)$ and $y = (y_1,y_2)$ agree on $t-1$ coordinates, and so it suffices to show that $\alpha = \beta = 1$. By construction, $0 \in \cH_x$ if either $x \in V \setminus V'$ or $w_x = 1$. In the former case, since the support of $\cH$ is contained in the support of $\cF$, $x$ is isolated, and so there cannot be an edge $(x,y)$. In the latter case, $w_y = 0$ (since $w_x + w_y \leq 1$ for all edges), and so $y$ is not in the support of $\cH$. We conclude that $\cH$ is $t$-agreeing up~to~$s$.
 
 We proceed by comparing the measures of $\cF$ and $\cH$:
\begin{align*}
 \mu_{s,m}^{n-1,\ell+1}(\cH) &= \sum_{x \in \ZZ_m^{n-1} \times \{0,1\}^\ell} \mu_{s,m}^{n-1,\ell}(x) \mu_{s/m}(\cH_x) \\ &=
 \sum_{x \in V \setminus V'} \mu_{s,m}^{n-1,\ell}(x) +
 \sum_{\substack{x \in V' \\ w_x = 1}} \mu_{s,m}^{n-1,\ell}(x) +
 \frac{s}{m} \sum_{\substack{x \in V' \\ w_x = 1/2}} \mu_{s,m}^{n-1,\ell}(x).
\end{align*}
 The first term is at least as large as $\sum_{x \in V \setminus V'} \mu_{s,m}^{n-1,\ell}(x) \mu_m(\cF_x)$, and the other two terms are at least $(2s/m) \sum_{x \in V'} \mu_{s,m}^{n-1,\ell}(x) w_x$. Therefore
\begin{align*}
 \mu_{s,m}^{n-1,\ell+1}(\cH) &\geq \sum_{x \in V \setminus V'} \mu_{s,m}^{n-1,\ell}(x) \mu_m(\cF_x) + \frac{2s}{m} \sum_{x \in V'} \sum_{x \in V'} \mu_{s,m}^{n-1,\ell}(x) w_x \\ &\geq
 \sum_{x \in V \setminus V'} \mu_{s,m}^{n-1,\ell}(x) \mu_m(\cF_x) + \frac{2s}{m} \sum_{x \in V'} \sum_{x \in V'} \mu_{s,m}^{n-1,\ell}(x) v_x \\ &= 
 \mu_{s,m}^{n-1,\ell+1}(\cF).
\end{align*}
 This completes the proof.
 
 We note that when $(s,m) = (1,2)$, necessarily $v_x = 1/2$ for all $x \in V'$ (since $0 < 2v_x < 2$ is an integer), and so we can use $w = v$.
\end{proof}

As a corollary, we can deduce the part $w(\ZZ_m^n,t,s) = w(n,t,s/m)$ of Theorem~\ref{thm:hamming-discrete}.

\begin{lemma} \label{lem:hamming-discrete-ub}
 Let $n \geq t \geq 1$, $m \geq 2$ and $s \leq m/2$. Then $w(\ZZ_m^n,t,s) = w(n,t,s/m)$.
\end{lemma}
\begin{proof}
 Let $\cF$ be a family on $\ZZ_m^n$ which is $t$-intersecting up~to~$s$. Applying Lemma~\ref{lem:hamming-discrete-ub-ind} $n$ times, we obtain a $t$-intersecting family $\cH$ on $\{0,1\}^n$ satisfying $\mu_m(\cF) \leq \mu_{s/m}(\cH) \leq w(n,t,s/m)$. This shows that $w(\ZZ_m^n,t,s) \leq w(n,t,s/m)$.
 
 For the other direction, let $\cH$ by a $t$-intersecting family on $n$ points satisfying $\mu_{s,m}(\cH) = w(n,t,s/m)$. Define $\cF = \sigma_0^{-1}(\cH)$, where $0$ is the zero vector. It is straightforward to verify that $\cF$ is $t$-agreeing up~to~$s$ and that $\mu_m(\cF) = \mu_{s/m}(\cH) = w(n,t,s/m)$. It follows that $w(\ZZ_m^n,t,s) \geq w(n,t,s/m)$.
\end{proof}

We proceed with the characterization of families achieving the bound $w(\ZZ_m^n,t,s)$. We start with the case $s/m < 1/2$. The idea is to prove a stronger version of Lemma~\ref{lem:hamming-discrete-ub-ind} which states that if $\cF$ and $\cH$ have the same measure and $\cH$ is equivalent to a Frankl family then so is $\cF$. This suggests analyzing the graph $G=(V,E)$ constructed in the proof of Lemma~\ref{lem:hamming-discrete-ub-ind} in the case of a family equivalent to a Frankl family.

\begin{lemma} \label{lem:hamming-discrete-uniqueness-gr}
 Let $n,t \geq 1$ and $\ell,s,m$ be given, and suppose that $s < m/2$. %and $n+\ell \geq t$. % was: $n+\ell > t$
 Let $\cF$ be a family on $\ZZ_m^n \times \{0,1\}^\ell$ which is equivalent to some $(t,r)$-Frankl family, for some $r \geq 0$. The graph $G' = (V',E')$ constructed in the proof of Lemma~\ref{lem:hamming-discrete-ub-ind} is either empty (has no vertices) or is connected and non-bipartite.
 
 The same holds when $s=1$ and $m=2$.
\end{lemma}
\begin{proof}
 Fix $m,t,r$. We first consider the case $s < m/2$.

 We will first prove the lemma for $n=1$ and $\ell = t+2r-1$. Then we will prove it for $n=1$ and arbitrary $\ell$. Finally, we will tackle the general case.
 
 Recall that the graph $G'$ was constructed by first constructing a larger graph $G = (V,E)$ and then removing all isolated vertices (vertices having no edges). % For the sake of the inductive proof, we also allow the case $(n,\ell) = (1,t-1)$.
  We will assume throughout that $\cF$ depends on the $n$th coordinate (that is, the underlying Frankl family has $n$ in its support), since otherwise $G$ has no edges and so $G'$ is empty.

 We will prove one more property of $G'$, which we call \emph{property~Z}. This property states that if $x \in V'$ and $x_i = 0$ for $i > n-1$ (that is, $i$ is a coordinate taking values in $\{0,1\}$) then $x$ is connected to some $y \in V'$ with $y_i = 1$.
 
 \smallskip
 
 We start with the case $(n,\ell) = (1,t+2r-1)$. If $r = 0$ then $G$ consists of a single vertex $[t-1]$ with a self-loop (we identify zero-one vectors with sets), and so it is connected and non-bipartite. If $r > 0$ then $V = \binom{[t+2r-1]}{\geq t+r-1}$. If $A \in \binom{[t+2r-1]}{\geq t+r-1}$ and $B \in \binom{[t+2r-1]}{\geq t+r}$ then $|A \cap B| \geq (t+r-1)+(t+r)-(t+2r-1) = t$, and so $V' \subseteq \binom{[t+2r-1]}{t+r-1}$; symmetry dictates that $V' = \binom{[t+2r-1]}{t+r-1}$. Moreover, the sequence
\[
 \{1,\ldots,t-1,t,\ldots,t+r-1\}, \{1,\ldots,t-1,t+r,\ldots,t+2r-1\}, \{2,\ldots,t-1,t,\ldots,t+r\}
\]
 corresponds to a path connecting $\{1,\ldots,t+r-1\}$ and $\{2,\ldots,t+r\}$. This shows that $G'$ is connected. Since the path has even length, we conclude that there is a path of even length connecting any two vertices. In particular, there is a path of even length connecting $\{1,\ldots,t-1,t,\ldots,t+r-1\}$ and $\{1,\ldots,t-1,t+r,\ldots,t+2r-1\}$, and together with the corresponding edge, we obtain an odd cycle. This shows that $G'$ is non-bipartite.
 
 To prove property~Z, we consider two cases. When $r = 0$, property~Z holds vacuously, since the only vertex contains no zero coordinates. When $r > 0$, we can assume without loss of generality that $x = \{2,\ldots,t+r\}$ and $i=1$. In that case, $x$ is connected to $y = \{1,\ldots,t-1,t+r,\ldots,t+2r-1\}$. 
 
 \smallskip

 We now prove by induction the case $n=1$ and $\ell \geq t+2r-1$. Let $G_{\ell},G'_{\ell}$ be the graphs corresponding to a particular value of $\ell$, where we assume without loss of generality that $\cF$ depends on the first $t+2r$ points (we can make this assumption since $\cF$ has to depend on the first point). Suppose that we have shown that $G'_\ell$ is connected and non-bipartite. Notice that $V_{\ell+1} = V \times \{0,1\}$ and $E_{\ell+1} = \{ ((x,i),(y,j)) : (x,y) \in E_\ell, (i,j) \neq (1,1) \}$. This shows that $V'_{\ell+1} = V'_\ell \times \{0,1\}$. Since $G'$ is connected and non-bipartite, there is an even-length path connecting any two $x,y \in V'$, which lifts to even-length paths connecting $(x,0),(y,0)$ and $(x,1),(y,1)$ (flipping the extra coordinate at each step); and an odd-length path connecting $x,y \in V'$ which lifts to paths connecting $(x,0),(y,1)$ and $(x,1),(y,0)$. This shows that $G'_{\ell+1}$ is connected. An odd-length path from $x$ to itself lifts to an odd-length path from $(x,0)$ to itself (flipping the extra coordinate at each step but the first), showing that $G'$ is non-bipartite.

 To prove property~Z, let $x \in V'_{\ell+1}$ have $x_i = 0$. If $i < \ell+1$ then property~Z for $G_\ell$ implies the existence of a neighbor $y$ with $y_i = 1$. If $i = \ell+1$ then $x$ is connected to some $y$ (since $x$ is not isolated), and since $V'_{\ell+1}$ is independent of the last coordinate, we can assume that $y_i = 1$ (the value of this coordinate does not change the number of coordinates on which $x$ and $y$ agree).

 \smallskip

 Suppose now that the lemma holds for families on $\ZZ_m^n \times \{0,1\}^\ell$. We will show that it holds for families on $\ZZ_m^{n+1} \times \{0,1\}^{\ell-1}$ as well. We denote the relevant graphs $G_n,G'_n$ and $G_{n+1},G'_{n+1}$, respectively. Without loss of generality we can assume that
\[
 V_{n+1} = \bigcup_{x_1 \in \ZZ_m^{n-1}} \bigcup_{x_2 \in \{0,1\}^{\ell-1}} \{x_1\} \times \sigma_0^{-1}(\{x' : (x_1,x',x_2) \in V_n\}) \times \{x_2\}.
\]
 In words, $V_{n+1}$ is obtained by applying $\sigma_0^{-1}$ to the $n$'th coordinate.
 
 We claim that $V'_{n+1}$ is obtained from $V'_n$ in the same way, showing that property~Z holds for $G_{n+1}$. Indeed, suppose that $(x_1,x',x_2) \in V'_n$ is not isolated, say it is connected to $(y_1,y',y_2) \in V'_n$, and let $x'' \in \sigma_0^{-1}(x')$. If $x'=y'=1$ then $(x_1,x'',x_2)$ is connected to $(y_1,x'',y_2)$. If $x'=1$ and $y'=0$ then $(x_1,x'',x_2)$ is connected to $(y_1,x''+s,y_2)$ (note $s+1 \leq x''+s \leq 2s$). If $x'=0$ then property~Z allows us to assume that $y'=1$, and so $(x_1,x'',x_2)$ is connected to $(y_1,y'',y_2)$ for some $1 \leq y'' \leq s$; indeed, Lemma~\ref{lem:interval-intersection} shows that the only elements $s$-agreeing with all of $1,\ldots,s$ are $1,\ldots,s$.
 
 %Moreover, each edge $((x_1,x',x_2),\allowbreak(y_1,y',y_2))$ in $G_n$ lifts to edges $(x_1,x'',x_2),(y_1,y'',y_2) \in E_{n+1}$, where $x' = \sigma_0(x'')$, $y' = \sigma_0(y'')$, and either $x' = y' = 1$ or $x'' - y'' \notin \{-(s-1),\ldots,(s-1)\}$. Note that if $(x',y') \neq (1,1)$ then there always exist appropriate choices of $x'',y''$: if $x'=y'=0$ then $(x'',y'')=(s+1,2s+1)$ works, and if $(x',y')=(0,1)$ then $(x'',y'')=(s+1,1)$ works. This shows that $V'_{n+1}$ is obtained from $V'_n$ in the same way that $V_{n+1}$ is obtained from $V_n$.
 
% We now consider two cases. \comment{Why do we need two cases?} Suppose first that $x' = 1$ for all $(x_1,x',x_2) \in V'_n$.  Consider any two vertices $(x_1,x'',x_2),(y_1,y'',y_2) \in V'_{n+1}$. A path connecting $(x_1,1,x_2)$ and $(y_1,1,y_2)$ in $V'_n$ lifts to a path connecting $(x_1,x'',x_2)$ and $(y_1,y'',y_2)$. 

% In the second case, necessarily $(x_1,0,x_2) \in V'_n$ iff $(x_1,1,x_2) \in V'_n$. \comment{This doesn't seem to hold! Which condition do we actually need?}

 Toward showing that $V'_{n+1}$ is connected, we prove first that $(x_1,a,x_2) \in V'_{n+1}$ is connected to $(x_1,b,x_2) \in V'_{n+1}$, by an even-length path, whenever $\sigma_0(a) = \sigma_0(b)$. Suppose first that $\sigma_0(a) = \sigma_0(b) = 1$. Let $(y_1,y',y_2) \in V'_n$ be a neighbor of $(x_1,1,x_2)$ in $G'_n$. If $\sigma_0(y') = 1$ then $(x_1,a,x_2),(y_1,a,y_2),(x_1,b,x_2)$ is a path in $G'_{n+1}$. If $\sigma_0(y) = 0$ then consider the following path in $G'_{n+1}$:
\[
 (x_1,1,x_2),(y_1,s+2,y_2),(x_1,2,x_2),(y_1,s+3,y_2),\ldots,(x_1,s-1,x_2),(y_1,2s,y_2),(x_1,s,x_2).
\]
 This contains a sub-path connecting $(x_1,a,x_2)$ and $(x_1,b,x_2)$.
 
 Suppose next that $\sigma_0(a) = \sigma_0(b) = 0$. Property~Z shows that $(x_1,0,x_2) \in V'_n$ has a neighbor $(y_1,1,y_2) \in V'_n$ in $G'_n$. Consider the following path in $G'_{n+1}$:
\[
 (x_1,s+1,x+2),(y_1,1,y_2),(x_1,s+2,x_2),(y_1,2,y_2),\ldots,(y_1,s-1,y_2),(x_1,2s,x_2),(y_1,s,y_2).
\]
 This path shows that $(x_1,a,x_2)$ and $(x_1,b,x_2)$ are connected whenever $s+1 \leq a,b \leq 2s$. Since $(y_1,s,y_2)$ neighbors $(x_1,c,x_2)$ for all $2s \leq c \leq m$, we deduce that $(x_1,a,x_2)$ and $(x_1,b,x_2)$ are connected for all $s+1 \leq a,b \leq m$.
 
 Consider now any two vertices $(x_1,x,x_2),(y_1,y,y_2) \in V'_{n+1}$ such that $(x_1,\sigma_0(x),x_2) \in V'_n$ neighbors $(y_1,\sigma_0(y),y_2) \in V'_n$ in $G'_n$. We will show that $(x_1,x,x_2)$ and $(y_1,y,y_2)$ are connected in $V'_{n+1}$ by an odd-length path by showing that $(x_1,X,x_2)$ and $(y_1,Y,y_2)$ are connected for some $X,Y \in \ZZ_m$ satisfying $\sigma_0(x) = \sigma_0(X)$ and $\sigma_0(y) = \sigma_0(Y)$. Such $X,Y$ are given by the following table:
\[
 \begin{array}{|c|c|c|c|c|}
 	\hline
 	\sigma_0(x),\sigma_0(y) & 0,0 & 0,1 & 1,0 & 1,1 \\\hline
 	X,Y & s+1,2s+1 & s+1,1 & 1,s+1 & 1,1 \\\hline
 \end{array}
\]
 This shows that $V'_{n+1}$ is connected. Moreover, any odd-length cycle in $V'_n$ lifts to an odd-length cycle in $V'_{n+1}$, showing that $G'_{n+1}$ is non-bipartite.

 When $s = 1$ and $m = 2$, notice that the general case differs from the case $n = 1$ only by a translation of the coordinates. Since the proof of the case $n = 1$ did not use the bound $s < m/2$, we conclude that the lemma holds even when $s = 1$ and $m = 2$.
\end{proof}

Now we can prove the strengthening of Lemma~\ref{lem:hamming-discrete-ub-ind}.

\begin{lemma} \label{lem:hamming-discrete-uniqueness-ind}
 Let $n,t \geq 1$ and $\ell,s,m$ be given, and suppose that $s < m/2$.
 Let $\cF$ be a family on $\ZZ_m^n \times \{0,1\}^\ell$ which is $t$-agreeing up~to~$s$, and let $\cH$ be the family on $\ZZ_m^{n-1} \times \{0,1\}^{\ell+1}$ constructed in Lemma~\ref{lem:hamming-discrete-ub-ind}. If $\mu_{s,m}^{n-1,\ell+1}(\cH) = \mu_{s,m}^{n,\ell}(\cF)$ and $\cH$ is equivalent to a $(t,r)$-Frankl family then $\cF$ is also equivalent to a $(t,r)$-Frankl family.
 
 The same holds for $(s,m) = (1,2)$, assuming that in the proof of Lemma~\ref{lem:hamming-discrete-ub-ind} we use $w_x = 1/2$ (see comment at the end of the proof).
\end{lemma}
\begin{proof}
 As in the proof of Lemma~\ref{lem:hamming-discrete-ub-ind}, we consider separately the degenerate case $n + \ell = t$ and the non-degenerate case $n + \ell > t$.
 
 When $n + \ell = t$, $\cH$ must be equivalent to the family $\{[t]\}$. Without loss of generality, we can assume that $\cH = [s]^{n-1} \times \{1\}^{\ell+1}$ (this corresponds to the choice $y = 0$ in the definition of equivalence). Thus $\cF$ has the form
\[
 \cF = \bigcup_{x_1 \in [s]^{n-1}} \bigcup_{x_2 = \{1\}^\ell} \{x_1\} \times \cF_{x_1,x_2} \times \{x_2\},
\]
 where each $\cF_{x_1,x_2}$ is $s$-agreeing. We remind the reader that the support of $\cF$ is the set of pairs $(x_1,x_2)$ such that $\cF_{x_1,x_2} \neq \emptyset$; the support $\cH$ is defined in the same way.
 
 Since $\mu_{s,m}^{n-1,\ell+1}(\cH) = \mu_{s,m}^{n,\ell}(\cF)$, moreover $|\cF_{x_1,x_2}| = s$, and so Lemma~\ref{lem:katona-cross-discrete} shows that $\cF_{x_1,x_2}$ is an interval (when $s = 1$ this is trivial). If $(x_1,x_2),(y_1,y_2)$ are in the support of $\cF$ then $\cF_{x_1,x_2},\cF_{y_1,y_2}$ are cross-$s$-agreeing, and so Lemma~\ref{lem:katona-cross-discrete} shows that they are equal (again, when $s = 1$ this is trivial). This shows that $\cF$ is equivalent to $\{[t]\}$ as well.
 
 \smallskip
 
 Consider now the non-degenerate case $n + \ell > t$. Recall that in the proof of Lemma~\ref{lem:hamming-discrete-ub-ind} we constructed two graphs, $G = (V,E)$ and $G' = (V',E')$, and a vector $w\colon V' \to \RR$, and defined $\cH$ in terms of this data. When computing the measure of $\cH$, we used the estimate
\[
 \sum_{x \in V \setminus V'} \mu_{s,m}^{m-1,\ell}(x) \geq \sum_{x \in V \setminus V'} \mu_{s,m}^{m-1,\ell}(x) \mu_m(\cF_x).
\]
 This estimate can only be tight if $\cF_x = \ZZ_m$ for $x \in V \setminus V'$. 
 We also used the estimate
\[
 \sum_{\substack{x \in V' \\ w_x = 1}} \mu_{s,m}^{n-1,\ell}(x) +
 \frac{s}{m} \sum_{\substack{x \in V' \\ w_x = 1/2}} \mu_{s,m}^{n-1,\ell}(x) \geq
 \frac{2s}{m} \sum_{x \in V'} \mu_{s,m}^{n-1,\ell}(x) w_x.
\]
 If $2s < m$, this can only be tight if it is never the case that $w_x = 1$. Recall that $w_x$ was a $\{0,1/2,1\}$-valued vector maximizing $\sum_{x \in V'} \mu_{s,m}^{n-1,\ell}(x) w_x$ under the constraints $0 \leq w_x \leq 1$ and $w_x + w_y \leq 1$ whenever $(x,y) \in E'$. Since $w_x \leq 1/2$ for all $x \in V'$, we see that in fact $w_x = 1/2$ for all $V'$ (since the constant~$1/2$ vector is always feasible). This shows that $\cH$ and $\cF$ have the same support.
  
 Moreover (still assuming $2s < m$), if there were a different $\{0,1/2,1\}$-valued vector maximizing $\sum_{x \in V'} \mu_{s,m}^{n-1,\ell}(x) w_x$ then the construction of Lemma~\ref{lem:hamming-discrete-ub-ind} would have shown that $\cF$ does not have maximum measure. We conclude that $w_x$ is the unique $\{0,1/2,1\}$-valued maximizer, and so the unique maximizer according to Proposition~\ref{pro:half-integrality}.
 
 When $(s,m) = (1,2)$, $w_x = 1/2$ for all $V'$ by assumption, and so $\cH$ and $\cF$ have the same support; the property proved in the preceding paragraph won't be needed in this case. 
 
 Lemma~\ref{lem:hamming-discrete-uniqueness-gr} shows that $G'$ is either empty or connected. Recall that $v_x = |\cF_x|/(2s)$. When computing the measure of $\cH$, we used the estimate $\sum_{x \in V'} \mu_{s,m}^{n-1,\ell}(x) w_x \geq \sum_{x \in V'} \mu_{s,m}^{n-1,\ell}(x) v_x$. When $s < m/2$, $w$ is the unique maximizer, and so $v_x = 1/2$ for all $x \in V$, that is, $|\cF_x| = s$ for all $x \in V'$; when $(s,m) = (1,2)$, the same trivially holds. For any two vertices $x,y \in V'$ connected by an edge, the sets $\cF_x,\cF_y$ must be $s$-agreeing, and so Lemma~\ref{lem:katona-cross-discrete} shows that $\cF_x = \cF_y$ (when $s=1$ this is trivial). Since $G$ is connected, we see that all $\cF_x$ are equal. It follows that for some $a \in \ZZ_m$, $\cF_x = \sigma_a^{-1}(\cH_x)$ for all $x \in V$. Therefore $\cF$ is equivalent to the same $(t,r)$-Frankl family as $\cH$. %\comment{Add more details?}
\end{proof}

Finally, we can determine the maximum measure families.

\begin{lemma} \label{lem:hamming-discrete-uniqueness}
 Fix $n \geq t \geq 1$, $m \geq 2$, and either $s < m/2$ or both $(s,m) = (1,2)$ and $t > 1$. If $\cF$ is a family on $\ZZ_m^n$ which is $t$-intersecting up~to~$s$ and has measure $w(\ZZ_m^n,t,s)$ then $\cF$ is equivalent to a family on $n$ points whose $\mu_{s/m}$-measure is $w(n,t,s/m)$.
\end{lemma}
\begin{proof}
 Let $\cF_0 = \cF, \cF_1, \ldots, \cF_n$ be the sequence of families constructed by applying Lemma~\ref{lem:hamming-discrete-ub-ind}, so that $\cF_\ell$ is a family on $\ZZ_m^{n-\ell} \times \{0,1\}^\ell$. The lemma states that $\mu_{s,m}^{n-\ell,\ell}(\cF_\ell) \leq \mu_{s,m}^{n-\ell-1,\ell+1}(\cF_{\ell+1})$. Since $\mu_{s/m}(\cF_n) \leq w(n,t,s/m) = w(\ZZ_m^n,t,s)$, we conclude that $\mu_{s,m}^{n-\ell,\ell}(\cF_\ell) = \mu_{s,m}^{n-\ell-1,\ell+1}(\cF_{\ell+1})$ for all $0 \leq \ell < n$. Also, Theorem~\ref{thm:weighted-main} shows that $\cF_n$ is equivalent to a $(t,r)$-Frankl family for some $r \geq 0$. Applying Lemma~\ref{lem:hamming-discrete-uniqueness-ind}, we conclude that the same holds for $\cF_0$.
\end{proof}

This completes the proof of Theorem~\ref{thm:hamming-discrete}.

\subsection{Continuous setting} \label{sec:hamming-continuous}

The continuous analog of the setting of Section~\ref{sec:hamming-discrete} is given by the following definitions.

\begin{definition} \label{tp-agreement}
 A \emph{continuous family on $n$ points} is a measurable subset of $\unitcirc^n$ (recall that $\unitcirc$ is the unit circumference circle). For $p \leq 1/2$, a continuous family $\cF$ on $n$ points is \emph{$t$-agreeing up~to~$p$} if any two vectors $x,y \in \cF$ have $t$ coordinates $i_1,\ldots,i_t$ such that the distance between $x_{i_j}$ and $y_{i_j}$ is less than $p$ for all $1 \leq j \leq t$.
 
 We denote the measure of a continuous family $\cF$ by $\mu(\cF)$.
 
 The maximum measure of a continuous family on $n$ points which is $t$-agreeing up~to~$p$ is denoted $w(\unitcirc^n,t,p)$.
\end{definition}

We will prove the following theorem.

\begin{theorem} \label{thm:hamming-continuous}
 Let $n,m,t \geq 1$ and $p < 1/2$. Then $w(\unitcirc^n,t,p) = w(n,t,p)$.
\end{theorem}

The proof uses a reduction to Theorem~\ref{thm:hamming-discrete} which is very similar to the one used to deduce Lemma~\ref{lem:katona-cross-continuous} from Lemma~\ref{lem:katona-cross-discrete}.

\begin{proof}
 We start with the easy direction: $w(\unitcirc^n,t,p) \geq w(n,t,p)$. Let $\cF$ be a $t$-intersecting family on $n$ points with $\mu_p(\cF) = w(n,t,p)$. Define a mapping $\tau(\unitcirc) \to \{0,1\}$ by $\tau([0,p)) = 1$ and $\tau([p,1)) = 0$. It is not hard to check that $\tau^{-1}(\cF)$ is a continuous family on $n$ points which is $t$-agreeing up~to~$p$ and has measure $w(n,t,p)$. Thus $w(\unitcirc^n,t,p) \geq w(n,t,p)$.
 
 For the other direction, let $\cF$ be a continuous family on $n$ points which is $t$-agreeing up~to~$p$. Let $\epsilon > 0$ be a parameter satisfying $p + \epsilon^{1/n} < 1/2$. Since $\cF$ is measurable, there is a sequence $F_i$ of cylindrical sets of total measure at most $\mu(\cF) + \epsilon$ which covers $\cF$. Since an $L_\infty$ ball of radius $\delta$ around any point has volume $(2\delta)^n$, it follows that any point in $\bigcup_i F_i$ is at $L_\infty$-distance $\epsilon^{1/n}/2$ from $\cF$. This implies that $\bigcup_i F_i$ is $t$-agreeing up~to~$p+\epsilon^{1/n}$.
 
 Choose $I$ so that $\sum_{i>I} \mu(F_i) < \epsilon$, and let $\cF^* = \bigcup_{i \leq I} F_i$. Let $M$ be a large integer, and partition $\unitcirc^n$ into $M^n$ cubes of side length $1/M$. Let $\cF^*_M$ consist of the union of all cubes contained entirely inside $\cF^*$. Thus $\mu(\cF^*_M) \geq \mu(\cF^*) - O(I/M) \geq \mu(\cF) - \epsilon - O(I/M)$. We can view $\cF^*_M$ as a family on $\ZZ_M^n$, and this family is $t$-agreeing up~to~$\lfloor (p+\epsilon^{1/n}) M \rfloor$. Theorem~\ref{thm:hamming-discrete} thus shows that $\mu(\cF^*_M) \leq w(n,t,\lfloor (p+\epsilon^{1/n}) M \rfloor/M)$. Therefore
\[
 \mu(\cF) \leq w(n,t,\lfloor (p+\epsilon^{1/n}) M \rfloor/M) + \epsilon + O\left(\frac{I}{M}\right).
\]
 Since $w(n,t,q)$ is continuous for $q \leq 1/2$, taking the limit $M\to\infty$ we deduce that $\mu(\cF) \leq w(n,t,p+\epsilon^{1/n}) + \epsilon$. Taking the limit $\epsilon\to0$, we conclude that $\mu(\cF) \leq w(n,t,p)$, and so $w(\unitcirc^n,t,p) \leq w(n,t,p)$.
\end{proof}

It is tempting to try and prove Theorem~\ref{thm:hamming-continuous} directly, along the lines of Theorem~\ref{thm:hamming-discrete}. Besides requiring a stronger version of Lemma~\ref{lem:katona-cross-continuous}, one would also need a version of Proposition~\ref{pro:half-integrality} for infinite polytopes, which we doubt holds.

\section{Arguments to infinity} \label{sec:reduction}

Suppose that we know an intersection theorem such as the Ahlswede--Khachatrian theorem in the uniform setting, that is for subsets of $\binom{[n]}{k}$. Can we deduce a matching theorem in related settings? This question has been considered by Ahlswede and Khachatrian~\cite{AK5} in their work on sets in Hamming space having small diameter (see the introduction to Section~\ref{sec:hamming}), and by Dinur and Safra~\cite{DinurSafra} in their seminal work on the inapproximability of vertex cover (see also Tokushige~\cite{Tokushige} for a similar argument). Ahlswede and Khachatrian were interested in $t$-agreeing families on $\ZZ_m^n$, while Dinur and Safra were interested in $t$-intersecting families on $n$ points, measured according to the $\mu_p$ measure. In this section, we describe their arguments in an abstract framework, and explain their power and limitations.

We describe and motivate the abstract framework in Section~\ref{sec:abstract-framework}. The argument of Ahlswede and Khachatrian appears in Section~\ref{sec:reduction-ak}, and the one of Dinur and Safra in Section~\ref{sec:reduction-ds}.

\paragraph{Notation} We use the falling power notation $n^{\ul{k}} = n!/(n-k)!$. For a family $\cF$, we denote by $\cF^{=i}$ its subfamily consisting of all subsets of size $i$.

\subsection{Abstract framework} \label{sec:abstract-framework}

The abstract framework described below originates in the author's PhD thesis~\cite{Filmus}. %In the following definition, the reader should think of $\cX$ as the collection of all susbets of $\NN$ of size at least $t$.

\begin{definition} \label{def:abstract-framework}
 Let $\cX$ be a monotone family on infinitely many points. A family $\cF$ on $n$ points is \emph{$\cX$-intersecting} if all $A,B \in \cF$ satisfy $A \cap B \in \cX$.
 
 For example, a family is $t$-intersecting iff it is $\cX_t$-intersecting, where $\cX_t$ is the collection of all subsets of $\NN$ containing at least $t$ points.
 
 For $0 \leq k \neq n$, we define $w(n,\cX,k)$ as the maximum of $\cF/\binom{n}{k}$ over all $\cX$-intersecting $\cF \subseteq \binom{[n]}{k}$.
 
 For all $p \in (0,1)$, we define $w(n,\cX,p)$ as the maximum of $\mu_p(\cF)$ over all $\cX$-intersecting families $\cF$ on $n$ points.
\end{definition}

Our goal is to relate $w(n,\cX,k)$ and $w(n,\cX,p)$. We can expect such a relation in the following situation.

\begin{definition} \label{def:abstract-framework-ub}
 Let $\cX$ be a monotone family on infinitely many points, and let $p \in (0,1)$. We write
\[
 w(n,\cX,k) \xrightarrow{k/n \to p} w
\]
if whenever $(k_i,n_i)$ is a sequence such that $n_i\to\infty$ and $k_i/n_i \to p$ then $w(n_i,\cX,k_i) \to w$.
\end{definition}

As an example, consider the Ahlswede--Khachatrian theorem~\cite{AK2,AK3}.

\begin{definition} \label{def:uniform-frankl}
 Let $t \geq 1$, $r \geq 0$, and $n \geq k \geq t+2r$. The $(t,r)$-Frankl family $\cF_{t,r}^{n,k}$ is
\[
 \cF_{t,r}^{n,k} = \{ S \in \binom{[n]}{k} : |S \cap [t+2r]| \geq t+r \}.
\]
 A family $\cF \subseteq \binom{[n]}{k}$ is \emph{equivalent} to a $(t,r)$-Frankl family if it equals the family in the definition of $\cF_{t,r}^{n,k}$ with $[t+2r]$ replaced by any other subset of $[n]$ of size $t+2r$.
\end{definition}

\begin{proposition}[Ahlswede and Khachatrian~\cite{AK2,AK3}] \label{pro:ak}
 Let $1 \leq t \leq k \leq n$, and suppose that $\cF \subseteq \binom{[n]}{k}$ is $t$-intersecting.
 
 If
\[
 \frac{r}{t+2r-1} < \frac{k-t+1}{n} < \frac{r+1}{t+2r+1}
\]
 then $|\cF| \leq |\cF_{t,r}^{n,k}|$, with equality only if $\cF$ is equivalent to a $(t,r)$-Frankl family.
 
 If
\[
 \frac{k-t+1}{n} = \frac{r+1}{t+2r+1}
\]
 then $|\cF| = |\cF_{t,r}^{n,k}| = |\cF_{t,r+1}^{n,k}|$, with equality only if $\cF$ is equivalent to a $(t,r)$-Frankl family or to a $(t,r+1)$-Frankl family.
\end{proposition}

The theorem shows that the maximum measure of a $t$-intersecting family roughly depends only on $k/n$. Indeed, as we show below, it implies that for $\frac{r}{t+2r-1} <  p < \frac{r+1}{t+2r+1}$ we have $w(n, \cX_t, k) \xrightarrow{k/n \to p} \mu_p(\cF_{t,r})$. 

The following definition captures the situation depicted in the theorem more abstractly.

\begin{definition} \label{def:abstract-framework-ub-2}
 Let $\cX$ be a monotone family on infinitely many points, and let $\cH$ be an $\cX$-intersecting family on $m$ points. For $p \in (0,1)$, we say that $\cH$ is \emph{$p$-optimal} (with respect to $\cX$) if whenever $(k_i,n_i)$ is a sequence such that $n_i\to\infty$ and $k_i/n_i \to p$, for large enough $i$ it holds that
\[
 w(n_i,\cX,k_i) = \frac{1}{\binom{[n_i]}{k_i}} |\{ A \in \binom{[n_i]}{k_i} : A \cap [m] \in \cH \}|.
\]
 We say that $\cH$ is \emph{uniquely $p$-optimal} if in the former situation, for large enough $i$ the bound $w(n_i,\cX,k_i)$ is attained only on families of the given form, and their equivalents (obtained by applying any permutation on the $n_i$ points).
\end{definition}

\begin{lemma} \label{lem:abstract-framework-opt-defs}
 Let $\cX$ be a monotone family on infinitely many points, let $\cH$ be an $\cX$-intersecting family on $m$ points, and suppose that $\cH$ is $p$-optimal for some $p \in (0,1)$. Then
\[
 w(n,\cX,k) \xrightarrow{k/n \to p} \mu_p(\cH).
\]
 Moreover, $w(n,\cX,p) \geq \mu_p(\cH)$ for all $n \geq m$.
\end{lemma}
\begin{proof}
 It is clear that $w(n,\cX,p) \geq \mu_p(\cH)$, since we can extend $\cH$ to a family on $n$ points having the same measure.
 
 Let now $(k_i,n_i)$ be a sequence satisfying $n_i\to\infty$ and $k_i/n_i \to p$. Define $\cF_i = \{ A \in \binom{[n_i]}{k_i} : A \cap [m] \in \cH \}$. Without loss of generality, we can assume that $w(n_i,\cX,k_i) =|\cF_i|/\binom{n_i}{k_i}$ for all $i$ (rather than just for large enough $i$). Notice that
\[
 \cF_i = \bigcup_{S \in \cH} \{ S \} \times \binom{[n_i-m]}{k_i-|S|},
\]
 and so
\[
 \frac{|\cF_i|}{\binom{[n_i]}{k_i}} = \sum_{S \in \cH} \frac{\binom{[n_i-m]}{k_i-|S|}}{\binom{[n_i]}{k_i}} = \sum_{S \in \cH} \frac{k_i^{\ul{|S|}} (n_i-k_i)^{\ul{m-|S|}}}{n_i^{\ul{m}}}.
\]
 The right-hand side tends to
\[
 \sum_{S \in \cH} p^{|S|} (1-p)^{m-|S|} = \mu_p(\cH). \qedhere
\]
\end{proof}

When $p \in (0,1/2)$ is not of the form $\frac{r}{t+2r-1}$, Lemma~\ref{lem:abstract-framework-opt-defs} shows that $w(n,\cX_t,k) \xrightarrow{k/n \to p} w(n,\cX_t,p)$, using both Proposition~\ref{pro:ak} and its weighted analog Theorem~\ref{thm:weighted-main}.

\smallskip

Lemma~\ref{lem:abstract-framework-opt-defs} gives a lower bound on $w(n,\cX,p)$. The following subsections describe two matching upper bounds.

\subsection{Argument of Ahlswede and Khachatrian} \label{sec:reduction-ak}

The first argument we describe originates from Ahlswede and Khachatrian~\cite{AK5}.

\begin{lemma} \label{lem:reduction-ak}
 Let $\cX$ be a monotone family on infinitely many points, let $\cH$ be an $\cX$-intersecting family on $m$ points, and suppose that $\cH$ is $p$-optimal for some $p \in (0,1)$. Then $w(m,\cX,p) = \mu_p(\cH)$.
 
 Moreover, for each $m$ there is a finite set $P_m$ such that the following holds. If $\cH$ is uniquely $p$-optimal for some $p \notin P_m$ and $\cF$ is an $\cX$-intersecting family on $m$ points satisfying $\mu_p(\cF) = \mu_p(\cH)$ then $\cF$ is equivalent to $\cH$ (obtained from $\cH$ by applying a permutation).
\end{lemma}
\begin{proof}
 In view of Lemma~\ref{lem:abstract-framework-opt-defs}, it suffices that show that if $\cF$ is an $\cX$-intersecting family on $m$ points then $\mu_p(\cF) \leq \mu_p(\cH)$.

 Define a sequence $(k_i,n_i)$ by taking $k_i = m + i$ and $n_i = \lfloor k_i/p \rfloor$, so that $n_i \to \infty$ and $k_i/n_i \to p$ (the exact choice is arbitrary). Extend $\cF$ to a family $\cF_i$ just as in the proof of Lemma~\ref{lem:abstract-framework-opt-defs}:
\[
 \cF_i = \bigcup_{S \in \cF} \{ S \} \times \binom{[n_i-m]}{k_i-|S|}.
\]
 As in the proof of Lemma~\ref{lem:abstract-framework-opt-defs}, $|\cF_i|/\binom{n_i}{k_i} \longrightarrow \mu_p(\cF)$. Since $w(n_i,\cX,k_i) \longrightarrow \mu_p(\cH)$ by Lemma~\ref{lem:abstract-framework-opt-defs}, we deduce that $\mu_p(\cF) \leq \mu_p(\cH)$.
 
 \smallskip
 
 Suppose now that $\cH$ is uniquely $p$-optimal and $\mu_p(\cF) = \mu_p(\cH)$. Then
\[
 \sum_{i=0}^m p^i (1-p)^{m-i} (|\cF^{=i}| - |\cH^{=i}|) = 0.
\]
 If it is not the case that $|\cF^{=i}| = |\cH^{=i}|$ for all $i$ then $p$ is a root of one of finitely many polynomials (whose exact number depends only on $m$). Assuming that $p$ is not one of these finitely many exceptions, we deduce that $|\cF^{=i}| = |\cH^{=i}|$ for all $i$. This implies that $|\cF_i| = |\cH_i|$ for all $i$, where $\cH_i$ is defined analogously to $\cF_i$. Since $\cH$ is uniquely $p$-optimal, it follows that $\cF_i$ is equivalent to $\cH$ for large enough $i$.
 
 Suppose that $\cF_i$ is equivalent to $\cH$. We can assume that $\cH$ depends on all $m$ points (otherwise, replace $\cH$ by a family on fewer points). Let $\cF_i$ be obtained from the family in Definition~\ref{def:abstract-framework-ub-2} by applying the permutation $\pi$ on $[n_i]$. We want to show that $\pi([m]) = [m]$, and so $\cF$ is equivalent to $\cH$.
 
 Indeed, consider any $j \in [m]$. Since $\cH$ depends on $j$, there exists an inclusion-minimal set $S \in \cH$ containing $j$. Thus $\pi(S \cup \{m+1,\ldots,m+k_i-|S|\}) \in \cF_i$. If $\pi(j) > m$ then, since $\cF_i$ is invariant under permutations of the last $n_i-m$ coordinates, we can find (assuming $i$ is large enough) some coordinate $\ell > m$ such that $\pi(S \setminus \{j\} \cup \{m+1,\ldots,m+k_i-|S|,\ell\}) \in \cF_i$, implying that $S \setminus \{j\} \in \cH$, which contradicts the fact that $S$ is inclusion-minimal.
 %If $\cH$ depends on all coordinates then it is not hard to check that $\cF$ must be equivalent to $\cH$. We obtain the same conclusion in general by considering the restriction of $\cH$ to the coordinates it depends on. %\comment{More details?}
\end{proof}

Lemma~\ref{lem:reduction-ak} allows us to deduce the value of $\wsup(t,p)$ for all $p \in (0,1/2)$ not of the form $\frac{r}{t+2r-1}$. For $p = \frac{r}{t+2r-1}$, we can deduce $\wsup(t,p)$ by continuity. The lemma also describes all the families of maximum measure on a large enough number of points, but only for most values of $p \in (0,1/2)$; for the countably many values $\bigcup_{m \geq 1} P_m \cup \{ \frac{r}{t+2r-1} : r \geq 0 \}$, the lemma fails to describe the families having maximum measure.

Compared to Theorem~\ref{thm:weighted-main}, the lifted version of Proposition~\ref{pro:ak} using Lemma~\ref{lem:reduction-ak} suffers from three shortcomings. First, it describes $\wsup(t,p)$ but not $w(n,t,p)$. Second, it works only for $p \leq 1/2$. Third, it does not describe the maximum measure families for all values of $p$. It seems that the only way to handle these shortcomings is to prove Theorem~\ref{thm:weighted-main} directly in the weighted setting.

\subsection{Argument of Dinur and Safra} \label{sec:reduction-ds}

The second argument we describe originates from Dinur and Safra~\cite{DinurSafra}, and also appears in Tokushige~\cite{Tokushige}.

\begin{lemma} \label{lem:reduction-ds}
 Let $\cX$ be a monotone family on infinitely many points, let $\cH$ be an $\cX$-intersecting family on $m$ points, and suppose that $\cH$ is $p$-optimal for some $p \in (0,1)$. Then $w(m,\cX,p) = \mu_p(\cH)$.
 %Moreover, for each $m$ there is a finite set $P_m$ such that the following holds. If $\cH$ is uniquely $p$-optimal for some $p \notin P_m$ and $\cF$ is an $\cX$-intersecting family on $m$ points satisfying $\mu_p(\cF) = \mu_p(\cH)$ then $\cF$ is equivalent to $\cH$ (obtained from $\cH$ by applying a permutation).
\end{lemma}
\begin{proof}
 In view of Lemma~\ref{lem:abstract-framework-opt-defs}, it suffices that show that if $\cF$ is an $\cX$-intersecting family on $m$ points then $\mu_p(\cF) \leq \mu_p(\cH)$.

 For a large integer $N$, extend $\cF$ to an $\cX$-intersecting family $\cF_N$ on $N$ points having the same $\mu_p$-measure. Observe that
\[
 \mu_p(\cF) = \mu_p(\cF_N) = \sum_{k=0}^N p^k (1-p)^{N-k} |\cF_N^{=k}|.
\]
 Fix $\epsilon > 0$. Since $\cH$ is $p$-optimal, for large enough $N$ we can bound $|\cF_N^{=k}|/\binom{N}{k} \leq \mu_p(\cH) + \epsilon$ for all $k$ satisfying $|k/N-p| \leq 1/N^{1/3}$ (the specific error term $1/N^{1/3}$ is  arbitrary; any $o(1/N^{1/2})$ error term would work). Hoeffding's bound implies that $\Pr[|\bin(N,p)/N-p| > 1/N^{1/3}] < \epsilon$ for large enough $N$, and so
\[
 \mu_p(\cF) \leq \sum_{k \colon |k/N-p| \leq 1/N^{1/3}} p^k (1-p)^{N-k} \binom{N}{k} (\mu_p(\cH) + \epsilon) + \epsilon \leq \mu_p(\cH) + 2\epsilon.
\]
 Taking the limit $\epsilon\to0$, we deduce $\mu_p(\cF) \leq \mu_p(\cH)$.
\end{proof}

One can also derive the second part of Lemma~\ref{lem:reduction-ak} using this technique, with a similar but more complicated argument which we leave to the reader.

\bibliographystyle{plain}
\bibliography{AK}

\end{document}